\documentclass[a4paper]{amsart}
\usepackage{amscd, amssymb}
\usepackage{xypic}
\usepackage[matrix, arrow]{xy}
\usepackage{array}

\usepackage{color}

\definecolor{darkgreen}{cmyk}{1,0,1,.2}
\definecolor{m}{rgb}{1,0.1,1}

\begin{document}



\def\sp{\operatorname{span}}
\def\trace{\operatorname{trace}}
\def\cK{\mathcal{K}}
\def\A{\mathcal{A}}
\def\B{\mathcal{B}}
\def\H{\mathcal{H}}
\def\h{\mathcal{H}}
\def\E{\mathcal{E}}
\def\G{\mathcal{G}}
\def\F{\mathcal F}
\def\X{\mathcal X}
\def\AS{\mathcal{AS}}
\def\mor{\operatorname{mor}}
\def\spin{\operatorname{spin}}
\def\hot{\hat{\otimes}}
\def\rt{\operatorname{rt}}
\def\lt{\operatorname{lt}}
\def\Res{\operatorname{Res}}

\newcommand*{\KK}{\mathrm{KK}}
\newcommand*{\RKK}{\mathrm{RKK}}
\newcommand*{\K}{\mathrm{K}}
\newcommand*{\Ktop}{\mathrm{K}^{\mathrm{top}}}
\newcommand*{\KX}{\mathrm{KX}}
\newcommand*{\RRKK}{\mathscr{R}\KK}
\newcommand*{\Prim}{\textup{Prim}}
\newcommand{\ord}{\operatorname{ord}}

\def\RRKK{\mathcal{R}\KK}

\newcommand{\lam}{\Lambda}
\newcommand{\dulam}{\widehat{\Lambda}}

\newcommand{\Cc}{\mathcal{C}}

\newcommand*{\Comp}{{\mathbb{K}}}
\newcommand*{\Bound}{{\mathbb{B}}}
\newcommand*{\cross}{\mathbin{\rtimes}}

\newcommand{\T}{\mathbb{T}}
\newcommand{\Q}{\mathbb{Q}}
\newcommand{\Z}{\mathbb{Z}}
\newcommand{\C}{\mathbb{C}}
\newcommand{\N}{\mathbb{N}}
\newcommand{\R}{\mathbb{R}}
\newcommand{\Ktimes}{\times\!\!\!\!\!\times}

\newcommand{\Lambdah}{\widehat{\Lambda}}
\newcommand{\Deltah}{\widehat{\Delta}}
\newcommand{\ip}[1]{\langle #1 \rangle}

\newcommand{\Kn}{\mathcal{K}_{\nu}}
\newcommand{\Ko}{\mathcal{K}_{\omega}}
\newcommand{\Kob}{\mathcal{K}_{\bar{\omega}}}
\newcommand{\ltg}{\ell^2G}
\newcommand{\tG}{\tilde{G}}
\newcommand{\tN}{\tilde{N}}
\newcommand{\Go}{G_{\omega}}
\newcommand{\Co}{\mathbb{C}_{\omega}}
\newcommand{\ob}{\bar{\omega}}
\newcommand{\Gob}{G_{\bar{\omega}}}
\newcommand{\Cob}{\mathbb{C}_{\bar{\omega}}}
\newcommand{\ltgt}{\ltg_{\tau}}
\newcommand{\talpha}{\tilde{\alpha}}
\newcommand{\ttau}{\tilde{\tau}}
\newcommand{\FUpi}{\mathcal{F}_U\hot_{C_0(X),p_i}\Gamma_0(E)}
\newcommand{\FUpzero}{\mathcal{F}_U\hot_{C_0(X),p_0}\Gamma_0(E)}
\newcommand{\FUpone}{\mathcal{F}_U\hot_{C_0(X),p_1}\Gamma_0(E)}
\newcommand*{\defeq}{\mathrel{:=}}
\newcommand{\Hom}{\operatorname{Hom}}
\def\RR{\mathbb R}

\newcommand{\sour}{\mathscr{P}}
\newcommand{\PD}{\mathrm{PD}}
\newcommand{\pd}{\mathscr{PD}}
\newcommand{\lefschetz}{\mathscr{L}}
\newcommand{\Isom}{\mathrm{Isom}}
\newcommand{\Eul}{\mathrm{Eul}}
\newcommand{\Cliff}{\mathrm{Cliff}}
\newcommand*{\Ztwo}{{\mathbb Z/2}}
\newcommand{\dR}{\mathrm{dR}}
\newcommand{\comb}{\mathrm{comb}}
\newcommand{\Spinor}{\mathrm{Spinor}}
\newcommand*{\Dslash}{{\mathsf /\!\!\!\!D}}
\newcommand{\End}{\mathrm{End}}
\newcommand{\DD}{\mathscr{D}}
\newcommand*{\DslashS}{{\mathsf /\!\!\!\!D_{\mathrm{S}}}}
\newcommand*{\capslash}{{\; /\!\!\!\!\cap \;}}

\newcommand{\thmcite}[1]{{\bfseries\upshape \cite{#1}}}
\newcommand{\thmcitemore}[2]{{\bfseries\upshape \cite[#2]{#1}}}
\newcommand{\citemore}[2]{{\upshape \cite[#2]{#1}}}
\newcommand{\ltwoforms}{L^2(\Lambda^*_\C X)}
\newcommand{\ltwoformsg}{L^2(\Lambda^*_\C X)\hot \ell^2G}

\def\Cl{\mathcal{C}l}
\def\Aut{\operatorname{Aut}}
\def\Gx{X\rtimes G}
\def\Bott{\operatorname{Bott}}
\def\CC{\mathbb C}
\def\ClV{C_{\!{}_{V}}}
\def\comp{\operatorname{comp}}
\def\ddS{\stackrel{\scriptscriptstyle{o}}{S}}
\def\intW{\stackrel{\scriptscriptstyle{o}}{W}}
\def\intU{\stackrel{\scriptscriptstyle{o}}{U}}
\def\intT{\stackrel{\scriptscriptstyle{o}}{T}}
\def\E{\mathcal E}
\def\EE{\mathbb E}
\def\Egx{\mathcal{E}(G)\times X}
\def\EGG{\mathcal{E}(\mathcal{G})}
\def\EGT{\mathcal{E}(G/G_0)}
\def\EG{\underline{EG}}
\def\EGN{\mathcal{E}(G/N)}
\def\EH{\mathcal{E}(H)}
\def\I{{\operatorname{I}}}
\def\i{{\operatorname{i}}}
\def\Ind{\operatorname{Ind}}
\def\Id{\operatorname{Id}}
\def\infl{\operatorname{inf}}
\def\L{\mathcal L}
\def\lk{\langle}
\def\rk{\rangle}
\def\NN{\mathbb N}
\def\oplus{\bigoplus}
\def\pt{\operatorname{pt}}
\def\QQ{\mathbb Q}
\def\res{\operatorname{res}}
\def\sm{\backslash}
\def\top{\operatorname{top}}
\def\ZM{{\mathcal Z}M}
\def\ZZ{\mathbb Z}

\def\om{\omega}
\def\Om{\Omega}

\setcounter{section}{-1}

\def\Inf{\operatorname{Inf}}
\def\SO{\operatorname{SO}}
\def\GL{\operatorname{GL}}
\def\Ad{\operatorname{Ad}}
\def\id{\operatorname{id}}
\def\U{\mathcal U}
\def\PU{\mathcal PU}
\def\BG{\operatorname{BG}}
\def\eps{\epsilon}
\def\ev{\operatorname{ev}}
\def\om{\omega}
\def\Om{\Omega}
\def\F{\mathcal{F}}
\def\TT{\mathbb T}
\def\dach{{\!\widehat{\ \ }}}
\def\SX{\operatorname{Stab}(X)}
\def\SZ{\operatorname{Stab}(Z)}
\def\SY{\operatorname{Stab}(Y)}
\def\SR{\operatorname{Stab}(\RR^3)}
\def\ST{\operatorname{Stab}(\TT^2)}
\def\SXhat{\SX\dach}
\def\SRhat{\SR\dach}
\def\SYhat{\SY\dach}
\def\SZhat{\SZ\dach}
\def\SThat{\ST\dach}
\def\Rep{\operatorname{Rep}}

\def\Flk{\F_{L,K}}
\def\mg{\mu_{G,A}}
\def\mgx{\mu_{\Gx,A}}
\def\ts{\hot}
\def\ga{\gamma}
\def\Br{\operatorname{Br}}
\def\Ab{\operatorname{Ab}}
\def\la{\lambda}
\def\gr{\operatorname{gr}}
\def\br{\Br^G_{\gr}(X)}
\def\mod{\operatorname{mod}}
\def\hato{\hat{\hot}}
\def\hotimes{{\hot}}
\def\hatoX{\hat{\hot}_{C_0(X)}}
\def\Ctau{C_{\tau}(X)}
\def\ctau{C_{\tau}}
\theoremstyle{plain}
 \newtheorem{theorem}{Theorem}[section]
\newtheorem{corollary}[theorem]{Corollary}
\newtheorem{problem}[theorem]{Problem}
 \newtheorem{lemma}[theorem]{Lemma}
\newtheorem{lem}[theorem]{Lemma}
\newtheorem{prop}[theorem]{Proposition}
\newtheorem{proposition}[theorem]{Proposition} \newtheorem{lemdef}[theorem]{Lemma and Definition}
\theoremstyle{defn}
\newtheorem{defn}[theorem]{Definition}
\theoremstyle{definition}
\newtheorem{definition}[theorem]{Definition}
\newtheorem{defremark}[theorem]{\bf Definition and Remark}
\theoremstyle{remark}
\newtheorem{remark}[theorem]{Remark}
\newtheorem{assumption}[theorem]{Assumption}
\newtheorem{example}[theorem]{Example}
\newtheorem{note}[theorem]{Note}
\numberwithin{equation}{section} \emergencystretch 25pt
\renewcommand{\theenumi}{\roman{enumi}}
\renewcommand{\labelenumi}{(\theenumi)}

\newcommand*{\abs}[1]{\lvert#1\rvert}

\newcommand{\Rn}{\R^{n}}

\newcommand{\ltwo}{L^{2}}
\newcommand{\ltwotau}{L^2_\tau (F_\epsilon)}
\newcommand{\ltwotaugk}{L^2_\tau (F_{p,\epsilon,g}) }
\newcommand{\ltwotauk}{L^2_\tau (F_{p,\epsilon}) }
\newcommand{\ltwotaukv}{L^2_\tau (V_{p,\epsilon}) }
\newcommand{\ltwoformsrn}{\ltwo (\Lambda^{*}_\C \Rn)}

\newcommand{\repon}{\mathrm{R}\bigl(\mathrm{O}(n,\R) \bigr)}
\newcommand{\repgamma}{\mathrm{R}(\Gamma)}
\newcommand{\indaon}{\mathrm{ind}_{\mathrm{a}}^{\mathrm{O}(n,\R)}}
\newcommand{\indton}{\mathrm{ind}_{\mathrm{t}}^{\mathrm{O}(n,\R)}}
\newcommand{\indagamma}{\mathrm{ind}_{\mathrm{a}}^{\Gamma}}
\newcommand{\indtgamma}{\mathrm{ind}_{\mathrm{t}}^{\Gamma}}
\newcommand{\on}{\mathrm{O}(n,\R)}
\newcommand{\kgamma}{\K^{0}_\Gamma }
\newcommand{\ext}{\Lambda^{*}_\C}
\newcommand{\card}{\mathrm{card}}
\def\ph{\varphi}
\def\supp{\operatorname{supp}}
\def\ind{\operatorname{ind}}
\newcommand{\Eulc}{\chi^{\Eul}}
\newcommand{\VK}{\textup{VK}}
\newcommand{\Laurs}{\C[X_1, X_1^{-1}, \ldots, X_n, X_n^{-1}]}
\newcommand{\Grd}{\mathcal{G}}

\title[Crossed products by proper actions]
  {Structure and K-theory of crossed products by proper actions}

\author[Echterhoff]{Siegfried
  Echterhoff}
   \address{Westf\"alische Wilhelms-Universit\"at M\"unster,
  Mathematisches Institut, Einsteinstr. 62 D-48149 M\"unster, Germany}
\email{echters@uni-muenster.de}

  \author[Emerson]{Heath
  Emerson}
  \address{Department of Mathematics and Statistics,
  University of Victoria,
  PO BOX 3045 STN CSCVictoria,
  B.C.Canada}
  \email{hemerson@math.uvic.ca}
  
  \thanks{The research for this paper was partially supported by the German Research Foundation
  (SFB 478 and  SFB 878) and the EU-Network Quantum SpacesÐNoncommutative Geometry (Contract
No. HPRN-CT-2002-00280).}
\maketitle

\begin{abstract}
We study the C*-algebra 
crossed product $C_0(X)\rtimes G$
of a locally compact group $G$ acting properly on  a locally compact Hausdorff space $X$. 
Under some mild extra conditions, which are automatic if $G$ is discrete or a Lie group,
 we describe in detail, and in terms of the action,
  the primitive ideal space of such crossed products
   as a topological space, in particular with respect to its fibring over the 
 quotient space $G\backslash X$. We also give some results on the $\K$-theory of such 
 C*-algebras. These more or less compute the $\K$-theory in the case of  
 isolated orbits with non-trivial (finite) stabilizers. We also give a purely $\K$-theoretic proof
 of a result  due to Paul Baum and Alain Connes on \(\K\)-theory with 
 complex coefficients of crossed products by finite 
 groups. 
\end{abstract}

\section{Introduction}

The C*-algebra crossed products \(C_0(X)\rtimes G\) associated to 
finite group actions on smooth compact 
manifolds give the simplest non-trivial examples of noncommutative 
spaces. Such group actions also play a 
role in various other parts of mathematics and physics,
 \emph{e.g.} the
linear action 
of a Weyl group on a complex torus in 
representation theory. 

As soon as 
the action of \(G\) on \(X\) is not free, the primitive ideal space of the 
crossed product \(C_0(X)\rtimes G\) is non-Hausdorff, although the quotient space
\(G\backslash X\) is Hausdorff. In fact 
\(\mathrm{Prim}\bigl( C_0(X)\rtimes G\bigr)\) fibres over 
\(G\backslash X\), in a canonical way, with finite fibres. 
As a fibration of 
sets this is easy enough to describe in direct, geometric terms, and 
it is `well-known to experts': the 
primitive ideal space is 
in a natural set bijection with 
\(G\backslash \SXhat\), where \(\SXhat = \{ (x,\pi) \mid 
x\in X, \; \pi \in \widehat{G}_x\}\), and the first coordinate 
projection defines the required fibration; the 
fibre over \(Gx\) is thus the unitary dual
 \(\widehat{G}_x\) of the stabilizer $G_x$. However, for purposes, for 
 example, 
 of \(\K\)-theory computation,  what is wanted here is a description of the 
 open sets of \(G\backslash \SXhat\) which correspond to the 
 open subsets of the primitive ideal space equipped with the 
 Fell topology, since this describes the space of ideals of 
 \(C_0(X)\rtimes G\) and potentially leads to a method of \(\K\)-theory 
 computation using excision. This description is the main contribution 
 of this article.

Proper actions of general locally compact groups naturally generalize actions 
of compact, or finite,
 groups -- because every proper action 
is `locally induced' from actions by compact subgroups. 
Such crossed products are important in operator algebras, 
because of the Baum-Connes conjecture. For an amenable, locally compact group \(G\) with \(G\)-compact universal proper \(G\)-space 
\(\EG\), Kasparov's
Fredholm representation ring \(\textup{R}(G) \defeq \KK^G(\C, \C)\) is 
canonically isomorphic to the \(\K\)-theory of 
$C_0(\EG)\rtimes G$. This is one statement of the Baum-Connes 
conjecture (and is due to Higson and Kasparov). In fact there 
are several possible statements of the Baum-Connes conjecture, 
but the general idea is that 
computation of \(\K\)-theory of the C*-algebra crossed products involving 
\emph{arbitrary} group actions can be, in 
certain circumstances, be  
reduced to the case where the action is proper. 
Thus the importance of computing 
\(\K\)-theory for proper actions. Since the analysis of the structure of the crossed product 
runs along similar lines in the case of proper actions of locally compact groups
as for actions of compact (or finite) groups, we treat the more general problem 
in this article.

If \(G\) is a locally compact group acting properly on $X$, 
we show that $C_0(X)\rtimes G$ is isomorphic to a certain
generalized fixed-point algebra, denoted 
 $C_0(X\times_G\cK)$, with respect to 
the diagonal action of $G$ on $C_0(X)\otimes\cK(L^2(G))$, with action 
 of $G$ on the second factor  given by the adjoint of 
the right regular representation $\rho$ of $G$ (in fact, we show a 
more general result along these lines  for crossed products $B\rtimes G$
where $B$ is a `fibred' over some proper $G$-space $X$).
It follows that $C_0(X)\rtimes G$ is the algebra of $C_0$-sections of a
 continuous bundle of C*-algebras over the orbit space 
$G\backslash X$ with fiber over $Gx$ isomorphic to the fixed point 
algebra $\cK(L^2(G))^{G_x}$, where $G_x$ denotes the 
(compact) stabilizer at $x$ which acts via conjugation by the restriction of right regular representation
of $G$ to $G_x$. 

This result has been shown by Bruce Evans
 in \cite{Evans} for compact group actions, but we are not aware of any reference  for the 
 more general class of proper actions.

The Peter-Weyl theorem implies that 
the fixed-point algebras $\cK(L^2(G))^{G_x}$ decompose into  direct sums
of algebras of compact operators indexed over the unitary duals
$\widehat{G}_x$ of  $G_x$. It follows that 
the primitive ideal space of $A=C_0(X)\rtimes G$ is 
in a natural set bijection with 
\(G\backslash \SXhat\), where \(\SXhat = \{ (x,\pi) \mid 
x\in X, \; \pi \in \widehat{G}_x\}\), which is a bundle, by the 
first coordinate projection, over the space \(G\backslash X\) of 
orbits, the fibre over \(Gx\) being the unitary dual
 \(\widehat{G}_x\) of the stabilizer $G_x$ (\emph{c.f.} the 
 first paragraph of this Introduction). 
 
It is possible to describe 
the topology on \(G\backslash \SXhat\) corresponding to the 
Fell topology on \(\Prim (A)\), in direct terms of the action. 
We do this in the case where the action of $G$ on $X$ satisfies
Palais's slice property, which means that $X$ is locally induced 
from the \emph{stabilizer} subgroups of \(G\) (see \S 1 for this notion). By a famous theorem of Palais,
this property is always satisfied if $G$ is a Lie group.

In general, \(G\backslash X\) is always an open subset of
\(G\backslash \SXhat\), and therefore corresponds to 
an ideal of \(C_0(X)\rtimes G\), as remarked above. Therefore to compute the \(\K\)-theory of the crossed product, it suffices to compute the 
\(\K\)-theory of the quotient space \(G\backslash X\) together with the boundary maps in the associated
six-term sequence. In the case of isolated 
fixed points and discrete $G$ this is fairly straightforward, at least up to torsion; in general, it is 
non-trivial, but we show in examples how the knowledge of the ideal structure of
$C_0(X)\rtimes G$ can still 
help in $\K$-theory computations, even when fixed points are 
not isolated. 

The problem of computing the \(\K\)-theory \emph{in general}
does not have an obvious solution. Indeed, it may not have any solution at all. 
However, if one ignores torsion, the problem gets much easier, at least for 
compact group actions. The reason is 
that if \(G\) is compact, the \(\K\)-theory of \(C_0(X)\rtimes G\) is a module 
over the representation ring \(\Rep (G)\), and similarly the \(\K\)-theory 
tensored by \(\C\) is a module over \(\Rep (G)\otimes_\Z \C\). For many 
groups \(G\) of interest, like finite groups, or connected groups, the 
complex representation ring \(\Rep (G)\otimes_\Z \C\) 
is quite a tractable ring, and the 
module structure of the complex \(\K\)-theory of 
\(C_0(X)\rtimes G\) gives significant additional 
information when used in conjunction with a `localization principal' 
developed mainly by Atiyah and Segal in the 1960's in connection with 
the Index Theorem. These ideas were exploited by Paul Baum and Alain Connes
in the 1980's to give a very beautiful formula for the \(\K\)-theory (tensored by 
\(\C\)) of the crossed product 
in the case of \emph{finite} group actions. We give a full proof of the 
theorem of Baum and Connes in this article, without attempting to generalize it 
to proper actions of locally compact groups.  The formula of 
Baum and Connes means that the difficulty in computing 
\(\K_*( C_0(X)\rtimes G\bigr)\) for finite group actions, is 
concentrated in the problem of computing the torsion subgroup. 
We do not shed much light on this problem.

This paper is to some extent expository. Our goal is to 
provide a readable synopsis of what is known, and what can 
be proved without too much difficulty, about crossed products
by proper actions, and their \(\K\)-theory. The paper is 
organized as follows: after giving  some preliminaries
on proper actions in \S \ref{sec-prel} we give a detailed discussion of the bundle structure
of $C_0(X)\rtimes G$ in \S \ref{sec-bundles}. 
In \S \ref{sec-mrg} we discuss the  natural bijection
between $\Prim(C_0(X)\rtimes G)$ and the quotient space 
$G\backslash \SXhat$ and in  \S \ref{sec-topology} we  show that
 this bijection is a homeomorphism for a quite naturally defined topology on $\SXhat$ if the action 
 satisfies the slice property.
All $\K$-theoretic discussions can be found in \S \ref{sec-K-theory}.

All spaces considered in this paper with the obvious exceptions of 
primitive ideal spaces of C*-algebras and the like, are assumed locally 
compact Hausdorff. 
\medskip 

A good part of this paper has been written while the first named author 
visited the University of Victoria in Summer 2008. He is very grateful to the second named author
and his colleagues for their warm hospitality during that stay! 
The authors are also grateful for some useful conversations with Wolfgang L\"uck and Jan Spakula.

\section{Preliminaries on proper actions}\label{sec-prel}
Assume that $G$ is a locally compact group. Suppose that \(G\) acts 
on the (locally compact 
Hausdorff) space $X$. The action is \emph{proper} if the map
$$G\times X\to X\times X: (g,x)\mapsto (gx,x)$$
is proper -- \emph{i.e.} if inverse images of compact sets are compact. 
Since $X$ is locally compact, this 
is equivalent to the condition that for every compact subset $K\subseteq X$ the set
$\{ g\in G: g^{-1}K\cap K\neq \emptyset\}$ is compact in $G$. 
In particular, \emph{any} action of a compact group on a locally compact 
Hausdorff space is proper.  We use to Palais's fundamental paper \cite{Pal}
as  a basic reference.

It is immediately clear that if \(G\) acts properly on \(X\) then the 
stabilizers $G_x:=\{g\in G: gx=x\}$ are compact.
Moreover, as 
one can show (e.g. see \cite[Theorem 1.2.9]{Pal}) without much 
difficulty, the quotient space $G\backslash X$ endowed with the 
quotient topology is 
a locally compact Hausdorff space.

If $H\subseteq G$ is a closed subgroup of 
$G$ which acts on some space $Y$, then the induced $G$-space
$G\times_HY$ is defined as the quotient space $(G\times Y)/H$ with respect to the diagonal action
$h\cdot (g,y)=(gh^{-1}, hy)$. The action of $H$ on $G\times Y$ is obviously free; it is a good exercise 
to prove that it is also proper (see \cite[\S 1.3]{Pal}). Hence 
$G\times_HY$ is a locally compact Hausdorff space. It carries a natural 
\(G\)-action by left translation on the first 
factor. The process of going from \(H\) acting on \(Y\) to 
\(G\) acting on \(G\times_HY \) is called `induction.'

\begin{prop}[{c.f. \cite[Corollary]{ech-induced}}]\label{prop-induction}
Suppose that $X$ is a $G$-space and $H$ is a closed subgroup of $G$.
Then the following are equivalent:
\begin{enumerate}
\item[(1)] There exists an $H$-space $Y$ such that $X\cong G\times_HY$ as $G$-spaces.
\item[(2)] There exists a continuous $G$-map $\varphi:X\to G/H$.
\end{enumerate}
In case of (1), the corresponding $G$-map $\varphi:G\times_HY\to G/H$ is given by $\varphi([g,y])=gH$
and in case of (2), the corresponding $H$-space $Y$ is the closed subset
$Y:=\varphi^{-1}(\{eH\})$ of $X$. The $G$-homeomorphism from $\Phi:G\times_HY\to X$
is then given by $\Phi([g,y])=gy$.
\end{prop}

Following Palais \cite{Pal}, we shall call a closed subset $Y\subseteq X$ a {\em global $H$-slice} 
if there exists a map $\varphi:X\to G/H$ as in part (2) of the above theorem with 
$Y=\varphi^{-1}(\{eH\})$, i.e., $Y\subseteq X$ is a global $H$-slice if and only if $Y$ is $H$-invariant and $X\cong G\times_HY$ as a $G$-space.
If $U$ is a $G$-invariant open subset of $X$, then we say that $Y\subset U$ is a {\em local $H$-slice}
of $X$, if $Y$ is a global $H$-slice for the $G$-space $U$. 
 The following observation is well known, but by lack of a direct 
reference, we give the proof.

\begin{lem}\label{lem-induced-proper}
Suppose that $H$ is a closed subgroup of $G$ and that $Y$ is an $H$-space.
Then $G\times_HY$ is a proper $G$-space if and only if $Y$ is a proper $H$-space.
\end{lem}
\begin{proof} Suppose first that $X=G\times_HY$ is a proper $G$-space. Then it is also 
a proper $H$-space, and since $\varphi:Y\to G\times_HY; \varphi(y)=[e,y]$ includes $Y$ as 
an $H$-invariant closed subset of $X$, it must be a proper $H$-space, too.

Conversely, if $Y$ is a proper $H$-space and $K\subseteq G\times_H Y$ is any compact set,
then we may choose compact sets $C\subseteq G$ and $D\subseteq Y$ such that 
$K\subseteq C\times_HD:=\{[c,d]\in G\times_HY: c\in C, d\in D\}$.
Let $F:=\{h\in H: h^{-1}D\cap D\neq \emptyset\}$. Since $H$ acts properly on $Y$, $F$ 
is a compact subset of $H$. 
Suppose now that $g\in G$ such $g^{-1}(C\times_HD)\cap (C\times_HD)\neq \emptyset$.
Then there exist $c_1, c_2\in C$ and $d_1, d_2\in D$ such that 
$[g^{-1}c_1,d_1]=[c_2, d_2]$, which in turn implies that there exists en element $h\in H$ such 
that $(g^{-1}c_1h, h^{-1}d_1)=(c_2, d_2)$. Then $h\in F$ and $g^{-1}c_1h=c_2$ implies that
$g=c_1hc_2^{-1}\in CFC^{-1}$, which is compact in $G$.
\end{proof}

The above lemma shows in particular that every $G$-space 
which is induced from some compact subgroup $L$ of $G$ must be proper. This observation 
has a partial converse, as the following theorem of Abels (\cite{Abels}) shows:

\begin{theorem}[Abels]\label{thmlocalind}
Suppose that $X$ is a proper $G$-space. Then for each $x\in X$ there exist a 
$G$-invariant open neighborhood $U_x$ of $x$, a compact subgroup \(L_x\) of 
\(G\), and a continuous $G$-map
$\varphi_x:U_x\to G/L_x$.
Thus $Y_x:=\varphi^{-1}(\{L_x\})$ is a local $L_x$-slice for $X$ and $U_x\cong G\times_{L_x}Y_x$.
\end{theorem}

Thus, combining the Theorem with Proposition \ref{prop-induction}, we see 
that every proper $G$-space is locally induced from compact subgroups.

\begin{remark}
\label{rem:slices_the_basics}
We make two minor remarks about slices whose easy 
verification we leave to the
reader: 

\begin{itemize}
\item We may 
always choose $L_x$ and $\varphi_x$ in the theorem 
in such a way that $\varphi_x(x)= eL_x$, and 
hence $x\in Y_x$ (at the expense of changing the 
subgroup \(L_x\) to a conjugate subgroup). 
\item If \(Y\subset U\) is a slice for the subgroup \(L_x\) then the 
isotropy subgroups of \(G\) in \(U\) are all subconjugate to \(L_x\), 
\emph{i.e.} are conjugate to subgroups of \(L_x\). 
\item If $U\subseteq X$ is a $G$-invariant subset with $L_x$-slice $Y_x$ 
then the intersection $Y_x\cap V$ is a $L_x$-slice for any given $G$-invariant 
subset  $V\subseteq U$.
\end{itemize}
\end{remark}

\begin{example}
Let \(G \defeq \prod_{n\in \N} \ \Z/2\), realize 
\(\Z/2\) as \(\{\pm 1\}\subset \TT\),  and let 
\(X \defeq \prod_{n \in \N} \TT\), with action of $G$ on $X$ given by 
translation. This is a free and proper action of a 
compact, totally disconnected group. It is easy to 
construct many local slices. Let \(I\) be a small interval 
neighbourhood of \(1\in \TT\), \(J = I \cup -I\). 
A \(G\)-neighbourhood basis 
(in the sense of giving a neighbourhood in the quotient space)
of \(x \defeq (1,1,1, \cdots) \in X\) 
is supplied by basic product sets of the form 
\(U = J\times \cdots \times J \times \TT \times \TT \times \cdots\). 
Thus any open \(G\)-invariant 
neighbourhood of \(x\) must contain 
one of these. 
Each of these open subsets has \(2^k\) components, 
where \(k\) is the 
number of factors of \(J\) occurring, and 
under any \(G\)-map from \(U\) to a totally disconnected space like 
\(G\) itself must thus map the above open subset to a 
finite subset of the target.  
A slice at \( x \) produces a \(G\)-map 
from any sufficiently small one of these \(G\)-neighbourhoods, 
with target some 
\(G/L_x\). Since this is also a totally disconnected space, 
the map must have image in a finite subset, since 
the subset must be \(G\)-invariant, it must be that 
\(G/L_x\) is itself finite.  Hence all slices must use closed 
subgroups \(L_x\) of finite index. 

In particular there is no slice through \(x\) with 
\(L_x \) the trivial subgroup of \(G\), \emph{i.e.} no 
slices through \(x\) 
with group \(L_x\) exactly equal to the isotropy group 
\(G_x\). 

If \(L_x\) is co-finite, it must be a subgroup of 
one of the `obvious' ones 
\(L_x = \{ (x_i) \mid x_i =1\; \textup{if}\; i\le n\}\),
 the quotient  
 \(G/L_x\) \(G\)-equivariantly identifies with 
 a finite product of \(\Z/2\)'s and it is easy to 
 produce a slice, \emph{i.e.} a \(G\)-map
 \(U \defeq J\times J \times \cdots J \times \TT \times \TT\times \cdots\) 
 to \(G/L_x\), by identifying \(J \cong I\times \Z/2\) and 
 using \(Y_x \defeq I\times \cdots \times I \times \TT \times \TT\times \cdots\). 
 
\end{example} 

In particular, whenever one has a slice, 
the stabilizer \(G_x\) of \(x\) is a closed subgroup 
of \(L_x\), but \(G_x\subseteq L_x\) may be strict. 
However, this can always be avoided when \(G\) is a 
\emph{Lie group}. This is the content of the 
following well-known result of Palais. 

\begin{theorem}[Palais's Slice Theorem]\label{thm-Palais}
Suppose that the Lie group $G$ acts properly on the locally compact
$G$-space $X$. Then for every $x\in X$ there exists an open $G$-invariant neighborhood
$U_x$ of $X$ which admits a $G_x$-slice $Y_x\subseteq U_x$ with $x\in Y_x$.
\end{theorem}

We should point out that Palais's original theorem (see \cite[Proposition 2.3.1]{Pal}) 
is stated for completely regular proper spaces, and therefore is actually slightly more general; we will not need the extra generality here. 

Motivated by 
Palais's theorem we give the following

\begin{defn}\label{defn-SP}\rm
Let $G$ be a locally compact group and let $X$ be a proper $G$-space. 
We say that $(G,X)$ satisfies \emph{Palais's slice property} (SP) 
if the conclusion of  Theorem \ref{thm-Palais} holds
for $(G,X)$, i.e., if $X$ is locally induced from the stabilizers $G_x$. 
\end{defn}

\begin{remark}\label{rem-subconj}
We emphasise that property (SP) for $(G,X)$ implies 
that every point $x\in X$ has a $G$-invariant neighborhood $U_x$ such  that 
the stabilizers $G_y$ for all $y\in U_x$ are sub-conjugate to $G_x$ (\emph{c.f.}
Remark \ref{rem:slices_the_basics}.) 
\end{remark}

\section{Proper actions and C*-algebra bundles}\label{sec-bundles}

If $X$ is a locally compact $G$-space, then there is a corresponding action on the 
C*-algebra $C_0(X)$ of all functions on $X$ which vanish at $\infty$ given by 
$(g\cdot f)(g')=f(g^{-1}g')$. The main object of this paper is the study of the 
crossed product \mbox{$C_0(X)\rtimes G$} in case where $X$ is a proper $G$-space.
For the general theory of crossed products we refer to Dana Williams's book
\cite{Dana-book}.

The construction of crossed products for group actions on spaces goes back to
early work of Glimm (see \cite{Glimm0}). Consider 
the space $C_c(G\times X)$ of continuous functions with compact supports on $G\times X$
equipped with convolution and involution given by the formulas
$$\ph*\psi(g,x)=\int_G\ph(t,x)\psi(t^{-1}g, t^{-1}x)\, dt\quad\text{and}\quad
\ph^*(g,x)=\Delta(g^{-1})\overline{\ph(g^{-1}, g^{-1}x)}.$$
Let $L^1(G, X)$ denote the completion of $C_c(G\times X)$ with respect to the 
norm $\|f\|=\int_G\|f(g, \cdot)\|_{\infty}dg$. Then $C_0(X)\rtimes G$ is the envelopping 
$C^*$-algebra of the Banach-*-algebra $L^1(G,X)$. 
It enjoys the universal property for covariant  representations of the pair $(C_0(X), G)$
as explained in detail in \cite[Proposition 2.40]{Dana-book}.

It follows from \cite{Dana-cont} that 
the crossed product $C_0(X)\rtimes G$ for any proper $G$-space $X$ has a 
canonical structure as an algebra of sections of a continuous C*-algebra bundle
over $G\backslash X$. In this section we want to give a more detailed description of this bundle.

 
Recall that if $Z$ is any locally compact space, then
 a C*-algebra $A$ is called a {\em $C_0(Z)$-algebra} (or an {\em upper semi-continuous bundle of 
 C*-algebras over $Z$}) if there exists  
 a $*$-homomorphism 
 $\phi:C_0(Z)\to ZM(A)$, the center of the multiplier algebra of $A$, such that
 $\phi(C_0(Z))A=A$. For $z\in Z$ put  $I_z:=\phi(C_0(Z\smallsetminus\{z\}))A$ and $A_z=A/I_z$.
 Then $A_z$ is  called the \emph{fibre of $A$ at $z$}. 
 Then every $a\in A$ can be viewed as a section of the bundle of C*-algebras
 $\{A_z: z\in Z\}$ via $a: z\mapsto a_z:=a+I_z$. The resulting positive function 
 $z\mapsto \|a_z\|$ is always upper semi continuous and we have $\|a\|=\sup_{z\in Z}\|a_z\|$ 
 for all $a\in A$. We say that $A$ is a 
 {\em continuous bundle of C*-algebras over $Z$} if in addition 
  all functions $z\mapsto \|a_z\|$ are continuous. We refer to \cite[C.2]{Dana-book}
 for the general properties of $C_0(Z)$-algebras.
 In what follows we shall usually suppress the name of the structure map $\phi:C_0(Z)\to ZM(A)$
 and we shall simply write $fa$ for $\phi(f)a$ if $f\in C_0(Z)$ and $a\in A$.
 
 A $*$-homomorphism $\Psi: A\to B$ between two $C_0(Z)$-algebras $A$ and $B$
 is called $C_0(Z)$-linear, if it commutes with the $C_0(Z)$-actions, that is if
 $\Psi(fa)=f\Psi(a)$ for all $f\in C_0(Z)$, $a\in A$. A $C_0(Z)$-linear homomorphism
 $\Psi$ induces $*$-homomorphisms $\Psi_z:A_z\to B_z$  for all $z\in Z$ by defining
 $\Psi_z(a+I_z^A)=\Psi(a)+I_z^B$ for all $z\in Z$, $a\in A$. 
 
If $A$ is a $C_0(Z)$-algebra and $G$ acts on $A$ 
by $C_0(Z)$-linear automorphisms, then the full crossed product $A\rtimes G$ is again
a $C_0(Z)$-algebra with respect to the composition
$$C_0(Z)\stackrel{\Phi}{\longrightarrow} ZM(A)\stackrel{i_{M(A)}}{\longrightarrow} ZM(A\rtimes G),$$
where $i_{M(A)}:M(A)\to M(A\rtimes G)$ denotes the extension to $M(A)$ of the canonical inclusion
$i_A: A\to M(A\rtimes G)$ (see \cite[Proposition 2.3.4]{Dana-book} for the definition of $i_A$). The action $\alpha$ then induces actions $\alpha^z$ of $G$  on each fiber $A_z$ and
it follows then from the exactness of the maximal crossed product functor that
the fibre $(A\rtimes G)_z$ of the crossed product at a point $z\in Z$ coincides with $A_z\rtimes G$
(e.g., see \cite[Theorem 8.4]{Dana-book}).
The following well-known lemma 
is often useful. 
 
 \begin{lemma}\label{lem-iso}
 Suppose that $\Phi:A\to B$ is a $C_0(Z)$-linear $*$-homomorphism between
 the $C_0(Z)$-algebras $A$ and $B$. Then 
$\Phi$ is injective (resp. surjective, resp. bijective) if and only if every fibre map $\Phi_z$ is injective
(resp. surjective, resp. bijective).
\end{lemma}
\begin{proof} Since $\|a\|=\sup_{z\in Z}\|a_z\|$ for all $a\in A$, it is clear that 
$\Phi$ is injective if $\Phi_z$ is injective for all $z\in Z$. Conversely assume $\Phi$ 
is injective. Then $B':=\Phi(A)\subseteq B$ is a $C_0(X)$-subalgebra of $B$ and 
there exists a $C_0(X)$-linear inverse $\Phi^{-1}:B'\to A$ which induces  fibre-wise inverses
$\Phi^{-1}_z$ for $\Phi_z:A_z\to B'_z\subseteq B_z$. 

Surjectivity of $\Phi$
 clearly implies surjectivity of $\Phi_z$ for all $z\in Z$. Conversely, if all $\Phi_z$ are surjective,
 then $\Phi(A)\subseteq B$ satisfies the conditions of \cite[Proposition C.24]{Dana-book},
 which then implies that $\Phi(A)=B$.
  \end{proof}

Assume now that $X$ is a proper $G$-space. We
 proceed with some general constructions of bundles over $G\backslash X$: 
For this assume that $B$ is any C*-algebra equipped with an action 
$\beta: G\to \Aut(B)$ of $G$. We define the algebra $C_0(X\times_{G,\beta}B)$
(we shall simply write $C_0(X\times_GB)$ if there is no doubt about the given action)
as the set of all bounded continuous functions 
$$F:X\to B\quad\text{such that}\quad
 F(gx)=\beta_g(F(x))$$
  for all $x\in X$ and $g\in G$ and such that the function 
$x\mapsto \|F(x)\|$ (which is constant on $G$-orbits) vanishes at infinity on $G\backslash X$.
It is easily checked that  $C_0(X\times_GB)$ becomes a C*-algebra when equipped 
with pointwise addition, multiplication, involution and the sup-norm. Note that this construction is well
known under the name of the {\em generalized fixed point algebras} for the proper action of
$G$ on $C_0(X,B)$ (e.g. see \cite{RW-pull-back, Kas, Meyer-fix}).

\begin{lemma}\label{lem-bundle}
$C_0(X\times_G B)$ is the section algebra of a 
continuous bundle of C*-algebras over \(G\backslash X\) with fibre over the orbit 
$Gx$ 
isomorphic to the fixed point algebra $B^{G_x}$, where $G_x=\{g\in G: gx=x\}$ is the 
stabilizer of $x$ in $G$. 
\end{lemma}
\begin{proof} Assume first that the action of $G$ on $X$ is transitive, i.e., $X=Gx $ for some $x\in X$.
Then it  is straightforward to check that evaluation at $x$ induces an isomorphism 
$C_0(Gx \times_GB)\to B^{G_x}; F\mapsto F(x)$.

For the general case we first note that multiplication with functions in $C_0(G\backslash X)$
provides $C_0(X\times_GB)$ with the structure of a $C_0(G\backslash X)$-algebra.
The ideal $I_{Gx }=C_0((G\backslash X)\smallsetminus \{Gx \})C_0(X\times_GB)$ then coincides with the set of functions $F\in C_0(X\times_GB)$ which vanish on $Gx $, and hence with
the kernel of the restriction map $F\mapsto F|_{Gx }$ from $C_0(X\times_GB)$ into $C_0(Gx \times_GB)$. If we compose this with the evaluation at $x$ we now see that the 
map $F\mapsto F(x)$ factors through an injective $*$-homomorphism of 
the fiber $C_0(X\times_GB)_{Gx }$  into $B^{G_x}$.
We need to show that this map is surjective.
 
Since images of $*$-homomorphism between C*-algebras are closed, it suffices to show that 
the evaluation map has dense image.
For this fix $b\in B^{G_x}$. For any neighborhood $U$ of $x$
 choose a positive function $f_U\in C_c(X)$ such that $\supp f_U\subseteq U$ and 
$\int_G f_U(g^{-1}x)\, dg=1$.
Then define $F_U\in C_0(X\times_GB)$ by
$F_U(y):=\int_G f_U(g^{-1}y)\beta_g(b)\,dg$ for all $y\in X$.
One checks that $F_U\in C_0(X\times_GB)$ and that $F_U(x)\to b$ as 
$U$ shrinks to $x$. This shows the desired density result.

 Finally, the fact that the $C_0(X\times_GB)$  is a continuous bundle 
  follows from the fact that the continuous function $x\mapsto \|F(x)\|$ is constant on 
  $G$-orbits in $X$, and hence factors through a continuous function on $G\backslash X$.
\end{proof}

For an induced proper $G$-space $X=G\times_HY$ we get the following
result.

\begin{lem}\label{lem-induced-bundle}
Suppose that $H$ is a closed subgroup of $G$ and $B$ is a $G$-algebra. Then there 
is a canonical isomorphism 
$$\Phi:C_0\big((G\times_HY)\times_GB\big)\to C_0(Y\times_HB)$$
given by $\Phi(F)(y)=F([e,y])$ for $F\in C_0\big((G\times_HY)\times_GB\big)$ and $y\in Y$.
\end{lem}
\begin{proof} It is straightforward to check that $\Phi$ is a well defined $*$-homomorphism with
inverse $\Phi^{-1}: C_0(Y\times_HB)\to C_0\big((G\times_HY)\times_GB\big)$ given by
$\Phi^{-1}(F)([g,y])=\beta_g(F(y))$, where $\beta:G\to \Aut(B)$ denotes the given action on $B$.
\end{proof}

In some cases the algebra $C_0(X\times_GB)$ has a much easier description.

\begin{definition}\label{def-fundamentaldomain}
Suppose that $G$ acts on the locally compact space $X$. A closed subspace $Z\subseteq X$ is 
called a {\em topological fundamental domain} for the action of $G$ on $X$ if the mapping
$Z\to G\backslash X; z\mapsto Gz$ is a homeomorphism.
\end{definition}

Of course,  a topological fundamental domain as in the definition, does not exist in most cases, but
the following examples show that there are at least some interesting cases where they do exist:

\begin{example}\label{ex:fundamentaldomain}
For the first example we consider the obvious action of $\SO(n)$ on $\RR^n$. Then 
the set $Z=\{(x,0,\ldots,0): x\geq 0\}$ is a topological fundamental domain for this action.
\end{example}

\begin{example}\label{ex-D4}
For the second example
let $G$ (the dihedral group $D_4$ ) be generated by a rotation $R$ around 
the origin in $\RR^2$ by 
$\frac{\pi}{2}$  radians, and the reflection $S$ through the line 
$l_S :=  \{(x,0): x\in \RR\}$. Writing $R=\left(\begin{smallmatrix} 0&-1\\1&0\end{smallmatrix}\right)$
and $S=\left(\begin{smallmatrix}1&0\\0&-1\end{smallmatrix}\right)$, the group $G$ has the
elements $\{E, R, R^2, R^3, S, SR, SR^2, SR^3\}$, where $E$ denotes the unit matrix. 
Since $G\subseteq \GL(2,\ZZ)$, it's action on $\RR^2$ fixes $\ZZ^2$, and therefore 
factors through an action on $\TT^2\cong\RR^2/\ZZ^2$. If we study this action on the 
fundamental domain $(-\frac{1}{2},\frac{1}{2}]^2\subseteq\RR^2$ for the translation action
of $\ZZ^2$ on $\RR^2$, it is an easy exercise to check that
the set
$$Z:=\{(e^{2\pi i s}, e^{2\pi i t}): 0\leq t\leq\frac{1}{2}, 0\leq s\leq t\}$$
is a topological fundamental domain for the action of $G$ on $\TT^2$.
\medskip

Of course, if we restrict the above action to the subgroup $H:=\lk R\rk\subseteq G$, we obtain 
an example of a group action with no topological fundamental domain.
\end{example}

\begin{proposition}\label{prop:fundamentaldomain}
Suppose that $Z\subseteq X$ is a topological fundamental domain for the proper action of $G$ on $X$.
Then, for any $G$-algebra $B$, there is an isomorphism
$$C_0(X\times_G B)\cong \{f:Z\to B: f(z)\in B^{G_z}\}$$
given by $F\mapsto F|_Z$.
\end{proposition}
\begin{proof}
This is an easy consequence of (the proof of) Lemma \ref{lem-bundle} together with
Lemma \ref{lem-iso}.
\end{proof}



Assume now that we have two commuting actions $\alpha,\beta:G\to \Aut(B)$ 
of $G$ on the same C*-algebra $B$. Then $\beta$ induces an action $\tilde{\beta}$
on the crossed product $B\rtimes_{\alpha}G$ in the canonical way. On the other 
hand, we also obtain an action $\tilde{\alpha}:G\to \Aut(C_0(X\times_{G,\beta}B))$ via
$$\big(\tilde{\alpha}_g(F)\big)(x)=\alpha_g(F(x)).$$
We want to show the following

\begin{proposition}\label{prop-commute1}
In the above situation we have a canonical isomorphism
$$C_0(X\times_{G,\beta}B)\rtimes_{\tilde\alpha}G\cong 
C_0\big(X\times_{G,\tilde\beta}(B\rtimes_\alpha G)\big).$$
\end{proposition}

For the proof we first need 

\begin{lemma}\label{lem-fix}
Let $K$ be a compact group and $G$ a locally compact group such that $\beta:K\to \Aut(B)$, $\alpha:G\to \Aut(B)$ are commuting actions of $K$ and $G$ on the C*-algebra $B$. Then the fixed-point algebra $B^K$ for the action of $K$ on $B$  is $G$-invariant and the inclusion $\iota:B^K\to B$ induces an isomorphism 
$$\iota\rtimes G: B^K\rtimes G\stackrel{\cong}{\longrightarrow} (B\rtimes G)^K,$$
where the fixed-point algebra on the right hand side is taken with respect to the action $\tilde{\beta}$
of $K$ on $B\rtimes G$ induced by $\beta$ in the canonical way. 
\end{lemma}
\begin{proof}
Note that the lemma is not obvious, 
since there might exist $G$-invariant sub-algebras $D\subseteq B$ such that the {\em full} crossed
product $D\rtimes G$ does not include faithfully into $B\rtimes G$ (while this would always be true for the reduced crossed products).

For the proof  we use the
fact that $B^K$ identifies with the compact operators of a $B\rtimes K$-Hilbert module
defined as follows: we
make $B$ into a pre-$B\rtimes K$ Hilbert module, 
with completion denoted by $X_B$, 
by defining the $B\rtimes K$-valued inner product 
$$\lk a,b\rk_{B\rtimes K}=\big(k\mapsto \beta_k(a^*)b\big)\in C(K,B)$$
for $a,b\in X_B$, and right action of $C(K,B)\subseteq B\rtimes K$ on $B\subseteq X_B$ by
$$a\cdot f=\int_Ka\beta_k(f(k^{-1}))\,dk, \quad a\in X_B, f\in C(K,B).$$
We can also check that $X_B$ carries a structure of a {\em full} left Hilbert-$B^K$-module
given by $_{B^K}\lk a,b\rk=\int_K \beta_k(ab^*)\,dk$ and a left action of $B^K$ on $B\subseteq X_B$ given by multiplication in $B$. One easily checks that these Hilbert-module structures on $X_B$ are compatible in the sense that
$$_{B^K}\lk a,b\rk c=a\lk b,c\rk_{B\rtimes K}$$
for all $a,b,c\in X_B$. Thus, the left action of $B^K$ on $X_B$ identifies $B^K$ with $\mathcal K(X_B)$.
We can then consider the descent module $X_B\rtimes G$, which is a 
$(B\rtimes K)\rtimes G$-Hilbert module with $\mathcal K(X_B\rtimes G)\cong B^K\rtimes G$ (e.g., see \cite{Combes} for the formulas for the actions and inner products).
%

Similarly,  via the action of $K$ on $B\rtimes G$  we obtain a 
$(B\rtimes G)\rtimes K$-Hilbert bimodule $X_{B\rtimes G}$ 
with compact operators isomorphic to $(B\rtimes G)^K$.
Since the actions of $K$ and $G$ on $B$ commute, we can identify 
$$(B\rtimes G)\rtimes K\cong B\rtimes (G\times K)\cong (B\rtimes K)\rtimes G.$$
We now check that under this identification we obtain an isomorphism 
of right  Hilbert $B\rtimes (G\times K)$-modules between $X_B\rtimes G$ and $X_{B\rtimes G}$,
such that $\iota\rtimes G$ intertwines the left actions of $B^K\rtimes G$ and $(B\rtimes G)^K$.
Since isomorphisms between Hilbert modules induce isomorphisms between their compact operators,
this will imply that $\iota\rtimes G$ is an isomorphism.
Note that by construction both modules $X_B\rtimes G$ and $X_{B\rtimes G}$ contain $C_c(G,B)$ 
as a dense $C_c(G\times K, B)$-submodule, where we view $C_c(G\times K,B)$ as a dense 
subalgebra of $B\rtimes (G\times K)$. One then checks that the identity map on $C_c(B\rtimes G)$
induces the desired module isomorphism such that $\iota\rtimes G$ commutes with the left 
actions of $C_c(G, B^K)=C_c(G,B)^K$ sitting as dense subalgebra in $B^K\rtimes G$ and $(B\rtimes G)^K$, respectively. Thus the identity on $C_c(G,B)$ extends to the desired isomorphism 
$X_B\rtimes G\cong X_{B\rtimes G}$.
\end{proof}

\begin{proof}[Proof of Proposition \ref{prop-commute1}]
We shall describe the isomorphism via a covariant pair $(\Phi, U)$.
For this let $(i_B, i_G) : (B, G) \to M(B\rtimes_{\alpha}G)$ denote the canonical inclusions
(see \cite[Proposition 2.3.4]{Dana-book}).
We then define 
$$\Phi:C_0(X\times_{G,\beta}B)\to 
M\big(C_0\big(X\times_{G,\tilde\beta}(B\rtimes_\alpha G)\big)\big)$$
by sending $F\in C_0(X\times_{G,\beta}B)$ to the multiplier of
$C_0\big(X\times_{G,\tilde\beta}(B\rtimes_\alpha G)\big)$ given by pointwise multiplication 
with $x\mapsto i_B(F(x))$. Similarly, for $g\in G$ we define 
$U_g\in  M\big(C_0\big(X\times_{G,\tilde\beta}(B\rtimes_\alpha G)\big)\big)$ by 
the pointwise application of $i_G(g)$. One checks that $(\Phi, U)$  gives a well defined 
covariant homomorphism of $(C_0(X\rtimes_{G,\beta}B), G)$ into $M\big(C_0\big(X\times_{G,\tilde\beta}(B\rtimes_\alpha G)\big)\big)$
whose integrated form $\Phi\rtimes U$
is $C_0(G\backslash X)$-linear, since $\Phi$ is $C_0(G\backslash X)$-linear.

Thus it suffices to check that $\Phi\rtimes U$
induces isomorphisms of the fibres. For this fix any $x\in X$. We then obtain a commutative 
diagram
$$
\begin{CD}
C_0(X\times_{G,\beta}B)\rtimes_{\tilde{\alpha}}G     @> \Phi\times U >>   
C_0\big(X\times_{G,\tilde\beta}(B\rtimes_\alpha G)\big) \\
@V\epsilon_x VV      @VV\epsilon_x V\\
B^{G_x}\rtimes G       @>>i_B\times i_G>   (B\rtimes G)^{G_x}
\end{CD}
$$ 
and the result follows from Lemma \ref{lem-fix}. 
\end{proof}

We are now going to describe the crossed product $C_0(X)\rtimes G$ in terms 
of section algebras of suitable C*-algebra bundles. 
Consider the algebra $\cK=\cK(L^2(G))$ 
equipped with the action $\Ad\rho:G\to \Aut (\cK)$, where $\rho:G\to U(L^2(G))$ denotes
the right regular representation of $G$ given by $\rho(g)\xi(t)=\sqrt{\Delta_G(g)}\xi(tg)$ 
for $\xi\in L^2(G)$. Let $\lt,\rt:G\to \Aut(C_0(G))$ denote the actions given by left and right translation on $G$, respectively.
It then follows from the extended version of the Stone-von Neumann theorem
(see \cite[Theorem 4.24]{Dana-book}, but see \cite{Rief-heis} or a more direct proof) that
\begin{equation}\label{eq-SvN}
M\times \lambda: C_0(G)\rtimes_{\lt} G\stackrel{\cong}{\longrightarrow}\cK(L^2(G)),
\end{equation}
where $M\times \lambda$ is the integrated form of the covariant pair $(M,\lambda)$
with $M: C_0(G)\to B(L^2(G))$ being the representation by multiplication operators and
$\lambda: G\to U(L^2(G))$ the left regular representation of $G$. 
Let $\tilde\rt: G\to \Aut(C_0(G)\rtimes_{\lt}G)$ denote the action induced from the 
right translation action $\rt:G\to \Aut(C_0(G))$.
For $f\in C_0(G)$ one  checks that $M(\rt_g(f))=\rho(g)M(f)\rho(g^{-1})$.
From this and the fact that $\rho$ commutes with $\lambda$ it follows that
$$M\times \lambda(\tilde{\rt}_g(\varphi))=\rho(g)\big(M\times\lambda(\varphi))\rho(g^{-1})$$
for all $g\in G$ and $\ph\in C_0(G)\rtimes_{\lt} G$. We shall use all this for the proof of
following result, where we write $\cK$ for $\cK(L^2(G))$:

\begin{theorem}\label{thm-crossed}
Let $G$ be a locally compact group acting properly on the locally compact space
$X$ with corresponding action $\tau:G\to \Aut(C_0(X))$ and let $\beta:G\to \Aut(B)$
be any action of $G$ on a C*-algebra $B$. Then there is a canonical isomorphism
$$(C_0(X)\otimes B)\rtimes_{\tau\otimes\beta}G\cong C_0\big(X\times_{G,\Ad\rho\otimes\beta}(\cK\otimes B)\big).$$
\end{theorem}
\begin{proof} We consider the
 commuting actions $\rt\otimes\id_B$ and $\lt\otimes\beta$ of 
$G$ on $C_0(G,B)=C_0(G)\otimes B$. It follows from Proposition  \ref{prop-commute1}  that 
there is a canonical isomorphism
$$C_0\big(X\times_{G,\widetilde{\rt\otimes\id_B}}(C_0(G,B)\rtimes_{\lt\otimes\beta} G)\big)
\cong C_0(X\times_{G,\rt\otimes\id_B}C_0(G,B))\rtimes_{\widetilde{\lt\otimes\beta}}G.$$
If we define $X\times_{G,\rt}G=G\backslash(X\times G)$ with respect to the 
action $g(x,h)=(gx, hg^{-1})$, and if we equip this space with the $G$-action given by left translation 
on the second factor, called $\tilde{\lt}$ below, we obtain a 
$\tilde{\lt}\otimes\beta-\widetilde{\lt\otimes\beta}$-equivariant isomorphism 
$$C_0(X\times_{G,\rt} G)\otimes B\cong C_0(X\times_{G,\rt\otimes \id_B}C_0(G,B)),$$
which induces an isomorphism of the respective crossed-products.
We further observe that the map $X\times_{G,\rt}G\to X; [x,g]\mapsto gx$ is a homeomorphism 
(with inverse given by $x\mapsto [x,e]$) which transforms $\tilde{\lt}$ into the given action $\tau$ on 
$X$. Combining all this, we obtain a canonical isomorphism
$$C_0\big(X\times_{G,\rt\otimes\id_B}C_0(G,B)\big)\rtimes_{\widetilde{\lt\otimes\beta}}G\cong C_0(X,B)\rtimes_{\tau\otimes \beta}G.$$
Using the isomorphism $C_0\big(X\times_{G,\rt\otimes\id_B}C_0(G,B)\big)\rtimes_{\widetilde{\lt\otimes\beta}}G
\cong C_0\big(X\times_{G,\widetilde{\rt\otimes\id_B}}(C_0(G,B)\rtimes_{{\lt\otimes\beta}}G)\big)$,
all remains to do is  to identify $C_0(G,B)\rtimes_{\lt\otimes\beta} G$ with 
$\cK(L^2(G))\otimes B$
equivariantly with respect to the action $\widetilde{\rt\otimes\id_B}$ and $\Ad\rho\otimes \beta$.
But such isomorphism is well known from Takesaki-Takai duality: it is straightforward 
to check that the isomorphism
$\Phi: C_0(G,B)\to C_0(G,B)$ given by $\Phi(f)(g)=\beta_{g^{-1}}(f(g))$ transforms
the action $\lt\otimes\beta$ to the action $\lt\otimes\id_B$ and the action 
$\rt\otimes\id_B$ to the action $\rt\otimes\beta$. 
This induces an isomorphism
\begin{align*}
C_0\big(X\times_{G,\widetilde{\rt\otimes\id_B}}(C_0(G,B)\rtimes_{\lt\otimes\beta} G)\big)
&\cong 
C_0\big(X\times_{G,\tilde{\rt}\otimes\beta}(C_0(G)\rtimes_{\lt}G)\otimes B)\big)\\
&\cong C_0\big(X\times_{G,\Ad\rho\otimes\beta}(\cK\otimes B)\big).
\end{align*}
\end{proof}

\begin{corollary}\label{cor-proper}
Let $G$ be a locally compact group acting properly on the locally compact space $X$
and let $\cK=\cK(L^2(G))$.
Then 
$$C_0(X)\rtimes_{\tau} G\cong C_0(X\times_{G,\Ad\rho}\cK).$$
\end{corollary}

\begin{remark}\label{rem-evaluation}
Explicitly, for any orbit $Gx\in G\backslash X$ the evaluation map
$q_x: C_0(X\times_G\cK)\to \cK^{G_x}$ can be described on the level of
$C_0(X)\rtimes G$  by the composition 
of maps 
$$C_0(X)\rtimes G\stackrel{q_{Gx }}{\longrightarrow} C_0(Gx)\rtimes G
\cong C_0(G/G_x)\rtimes G\stackrel{M\times \lambda}{\longrightarrow} \cK^{G_x},$$
where the first map is induced by the $G$-equivariant restriction map $C_0(X)\to C_0(Gx)$ and
the second is induced by the $G$-homeomorphism $G/G_x\to Gx; gG_x\mapsto gx$.
\end{remark}

\begin{example}
\label{ex:weyl_rank_one}
Let \(X \defeq \TT\), \(G \defeq \Z/2\) acting by conjugation on the 
circle. Then 
\[C(\TT)\rtimes \Z/2  \cong \{ f\in C\bigl([0,1], M_2(\C)\bigr) \mid f(0) \; \textup{and}\; 
f(1)\; \textup{are diagonal}\}.\]
This is immediate from Corollary \ref{cor-proper}, using the 
basis \( \{\frac{1}{\sqrt{2}}(1,1), \frac{1}{\sqrt{2}}(1,-1)\}\) for \(\ell^2(\Z/2)\)
 to identify it with 
\(\C^2\) (this diagonalizes the \(\Z/2\)-action.)  A continuous function 
\(f\colon \TT \to \cK\bigl( \ell^2(\Z/2))\) such that 
\( f(gx) = \Ad g \bigl (f(x)\bigr)\) is determined by its restriction to 
\( \{ z\in \TT \mid \textup{Im}(z)\ge 0\}\), where, using the 
above basis, we can identify it with
a map  
\(f\colon [0,1]\to M_2(\C)\) such that 
\(f(0)\) and \(f(1)\) commute with the matrix 
\(\left[ \begin{matrix} 1 & 0\\ 0 & -1\end{matrix}\right]\).
The commutant of this matrix consists of the diagonal 
matrices. 
\end{example}

Theorem \ref{thm-crossed} can be extended to 
the case of proper actions on general C*-algebras $B$, \emph{i.e.} such that 
$B$ is an $X\rtimes G$-algebra for  some proper $G$-space $X$. 
Thus, \(B\) is a 
$C_0(X)$-algebra equipped with a $G$-action $\beta:G\to \Aut(B)$
such that the structure map $\phi:C_0(X)\to ZM(B)$ is $G$-equivariant.
In this situation  the generalized fixed point algebra $B^{G,\beta}$ can be constructed as 
follows: we consider the algebra $C_0(X\times_GB)$ as studied above. 
If $b\in B$, we write $b(y)$ for the evaluation of $B$ in the fiber $B_y$, $y\in X$.
Similarly, for $F\in C_0(X\times_GB)$ we write $F(x,y)$ for the evaluation of
the element  $F(x)\in B$ in the fiber $B_y$.
Then $C_0(X\times_GB)$
becomes a $C_0(G\backslash(X\times X))$-algebra via the structure map
$$\Phi: C_0(G\backslash(X\times X))\to ZM(C_0(X\times GB)); \big(\Phi(\varphi)F\big)(x,y)=
\varphi([x,y])F(x,y).$$
We then define $B^{G,\beta}$ (or just $B^G$ if $\beta$ is understood) 
as the restriction of $C_0(X\times_GB)$ to
$G\backslash\Delta(X)\cong G\backslash X$, where $\Delta(X)=\{(x,x):x\in X\}$.
Note that with this notation we have 
$C_0(X\times_{G,\beta}B)\cong 
\big(C_0(X)\otimes B\big)^{G,\tau\otimes \beta}$ 
(if $\tau$ denotes the corresponding action on $C_0(X)$). 
Moreover, if $G$ is compact, this notation coincides with the usual notation of the 
fixed-point algebra $B^{G}$.

With this notation we get

\begin{theorem}\label{thm-extended}
Suppose that $B$ is an $X\rtimes G$-algebra for the proper $G$-space $X$
via some action $\beta:G\to \Aut(B)$. Then 
$B\rtimes G$ is isomorphic to $(\cK\otimes B)^{G, \Ad\rho\otimes\beta}$.
\end{theorem}
\begin{proof}
Consider the $C_0(X\times X)$-algebra $C_0(X,B)$. By Theorem \ref{thm-crossed}
we have $C_0(X,B)\rtimes_{\tau\otimes\beta} G\cong C_0(X\times_{G,\Ad\rho\otimes\beta}(\cK\otimes B))$. The crossed product $C_0(X,B)\rtimes G$ carries a canonical 
structure as a $C_0(G\backslash(X\times X))$-algebra which is induced from the $C_0(X\times X)$-structure 
of $C_0(X,B)$.
A careful look at the proof of the isomorphism $C_0(X,B)\rtimes G\cong C_0(X\times_G(\cK\otimes B))$
reveals that this isomorphism preserves the $C_0(G\backslash(X\times X))$-structures 
on both algebras. Using the $G$-isomorphism $C_0(X,B)|_{\Delta(X)}\cong B$ 
which is induced from the $*$-homomorphism $C_0(X,B)\cong C_0(X)\otimes B\to B; (f\otimes b)\mapsto fb$, we now
obtain a chain of isomorphisms 
\begin{align*}
B\rtimes G&\cong \big(C_0(X,B)|_{\Delta(X)}\big)\rtimes G\\
&\cong \big(C_0(X,B)\rtimes G)|_{G\backslash\Delta(X)}\\
&\cong 
\big(C_0(X\times_G(\cK\otimes B))\big)|_{G\backslash\Delta(X)}\\
&=(\cK\otimes B)^{G, \Ad\rho\otimes\beta}.
\end{align*}
\end{proof}

\begin{remark} Using this, it is not difficult to check that 
for any $X\rtimes G$-algebra $B$ for some proper $G$-space $X$ the 
crossed product $B\rtimes G$ is a $C_0(G\backslash X)$-algebra with fiber 
at an orbit $Gx$ isomorphic to  $(\cK\otimes B_x)^{G_x, \Ad\rho\otimes\beta^x}$,
where $\beta^x:G_x\to \Aut(B_x)$ is the action induced from $\beta$ in the canonical way.
\end{remark}

The results of this section fit into the framework of 
generalized fixed-point algebras (see the work of 
Marc Rieffel and also Ralf 
Meyer in \cite{Rief-proper, Rief-proper1, Meyer-fix}); 
our aim here is not generality, but explicitness, and we 
have taken a direct approach.

\section{The Mackey-Rieffel-Green theorem
for proper actions}\label{sec-mrg}
The Mackey-Rieffel-Green theorem
 (or Mackey-Rieffel-Green machine)
supplies, under some suitable conditions 
on a given C*-dynamical system $(A,G,\alpha)$, a systematic way of describing the 
irreducible representations (or primitive ideals) 
 of the crossed product $A\rtimes_\alpha G$ in terms of the 
associated action of $G$ on $\Prim(A)$ by inducing representations (or ideals) 
from the stabilizers for this action. 
We refer to \cite{Dana-book, EchWil} for  recent discussions of 
this general machinery, and to \cite{Green1, GootRos} for some important contributions towards 
 this theory.

In this section we want to give a self-contained exposition of the Mackey-Rieffel-Green machine
in the special case of crossed products by proper actions on spaces, in which the
result will follow easily from the bundle description of the crossed product $C_0(X)\rtimes G$
as obtained in the previous section and the following explicit description of the fibers.
This explicit description of the fibers will also play an important r\^ole in our 
description of the Fell topology on 
$(C_0(X)\rtimes G)\dach$ as given in \S \ref{sec-topology} below.

From Corollary \ref{cor-proper} and  Lemma \ref{lem-bundle}, if 
$G$ acts properly on $X$, then
 the fibre $(C_0(X)\rtimes G)_{G(x)}$ of the crossed product 
$C_0(X)\rtimes G\cong C_0\big(X\times_{G,\Ad\rho}\cK(L^2(G)\big)$ at the orbit $G(x)$ is 
isomorphic to \(C_0(G/G_x)\rtimes G\), equivalently, to 
the fixed-point algebra $\cK(L^2(G))^{G_x,\Ad\rho}$ for the compact stabilizer 
$G_x$ at $x$, where $\rho: G\to U(L^2(G))$ denotes the right regular representation of $G$. We now analyse the structure of this fibre, using 
the Peter-Weyl theorem. 

Let us recall some basic constructions in representation theory.
If $\H$ is any Hilbert space we denote by $\H^*$ its {\em adjoint Hilbert space}, that is
$$\H^*=\{\xi^*: \xi\in \H\}$$
with the linear operations $\lambda\xi^*+\mu\eta^*=(\bar{\lambda}\xi+\bar\mu\eta)^*$ 
and the inner product $\lk \xi^*,\eta^*\rk=\lk\eta,\xi\rk$. Note that $\H^*$ identifies canonically 
with the space of continuous linear functionals on $\H$. 
If $\sigma:G\to \U(\H)$ is a representation of the group $G$ on the Hilbert space 
$\H$, then its {\em adjoint representation} $\sigma^*:G\to \U(H^*)$ is given by
$$\sigma^*(g)\xi^*:=(\sigma(g)\xi)^*.$$
Assume now that $K$ is a compact subgroup of $G$ and let $\sigma:K\to U(V_\sigma)$
be a unitary representation of $K$. We then define a representation 
$\pi^{\sigma}=P^\sigma\times U^\sigma$ of the crossed product $C_0(G/K)\rtimes G$ 
as follows: we define
\begin{equation}\label{eq-ind1}
\H_{U^\sigma}:=\{\xi\in L^2(G, V_\sigma): \xi(gk)=\sigma(k^{-1})\xi(g)\;\forall g\in G, k\in K\}.
\end{equation}
Then define the covariant representation $(P^\sigma, U^\sigma)$ of $C_0(G/K)\rtimes G$ 
on $\H_{U^\sigma}$ by
\begin{equation}\label{eq-ind2}
\big(P^\sigma(\ph)\xi\big)(g)=\ph(gK)\xi(g)\quad\text{and}\quad (U^\sigma(t)\xi)(g)=\xi(t^{-1}g).
\end{equation}

Note that in classical representation theory of locally compact groups the covariant pair 
$(P^\sigma, U^\sigma)$ is often called 
the ``system of imprimitivity''  induced from the representation $\sigma$ of $K$
(e.g. see \cite{Blatt}).

In what follows, if $\sigma\in \widehat{K}$, we denote by $p_\sigma\in C^*(K)$ the central projection 
corresponding to $\sigma$, i.e., we have $p_{\sigma^*}=d_\sigma\chi_{\sigma^*}$, where 
$\chi_{\sigma^*}(k)=\trace\sigma^*(k)$ denotes the character of the adjoint $\sigma^*$ of $\sigma$
and $d_\sigma$ denotes the dimension of $V_\sigma$.
Note that it follows from the Peter-Weyl theorem  (e.g. see \cite[Chapter 7]{DE})
that $\sigma(p_\sigma)=1_{V_\sigma}$ and $\tau(p_\sigma)=0$ for all $\tau\in \widehat{K}$ 
not equivalent to $\sigma$, and 
that  $\sum_{\sigma\in \widehat{K}}p_\sigma$ converges strictly 
to the unit $1\in M(C^*(K))$.

\begin{lemma}\label{lem-fixed}
Let $G$ be a locally compact group and let $K$ be a compact subgroup of $G$ acting 
on $\cK=\cK(L^2(G))$ via $k\mapsto \Ad\rho(k)$.
 For each $\sigma\in \widehat{K}$ let
  $L^2(G)_\sigma:=\rho(p_{\sigma^*})L^2(G)$.  Then the following are true:
\begin{enumerate}
\item $L^2(G)= \oplus_{\sigma\in \widehat{K}} L^2(G)_\sigma$;
\item each space $L^2(G)_\sigma$ is $\rho(K)$-invariant 
and decomposes into a tensor product $\H_{U^\sigma}\otimes V_\sigma^*$
such that $\rho(k)|_{L^2(G)_\sigma}=1_{\H_{U^\sigma}}\otimes \sigma^*(k)$ for all $k\in K$;
\item $\cK^K\cong \oplus_{\sigma\in \widehat{K}} \cK(\H_{U^\sigma})$, where the
 isomorphism is given 
by sending an operator $T\in  \cK(\H_{U^\sigma})$ to the operator 
$T\otimes 1_{V_{\sigma}^*}\in \cK(L^2(G)_\sigma)$ under the decomposition of {\rm{(ii)}};
\item the projection of $C_0(G/K)\rtimes G\cong \cK^K$ onto the factor $\cK(\H_{U^\sigma})$
in the decomposition in {\rm{(iii)}} is equal to the representation $\pi^\sigma=P^\sigma\times U^\sigma$
constructed above in \eqref{eq-ind1} and \eqref{eq-ind2}.
\end{enumerate}
\end{lemma}
\begin{proof} The proof is basically a consequence of the Peter-Weyl Theorem for the compact group 
$K$. Item (i)  follows from the fact that the  central 
projections $p_{\sigma^*}$ add up to the unit in $M(C^*(K))$ with respect to
the strict topology. 

For the proof of (ii) we first consider the induced $G$-representation 
$U^{\lambda_K}$, where $\lambda_K$ denotes the left regular representation of $K$.
It acts on the Hilbert space 
$$\H_{U^\lambda}:=\{\xi\in L^2(G, L^2(K)): \xi(gk, l)=\xi(g, kl)\;\forall g\in G, k,l\in K\}.$$
But a short computation shows that $U^{\lambda_K}\cong \lambda_G$, 
the left regular representation of $G$, where a unitary intertwining operator is 
given by 
$$\Phi_\lambda: \H_{U^\lambda}\stackrel{\cong}{\longrightarrow} L^2(G); \Phi_\lambda(\xi)(g)=\xi(g,e).$$
By the Peter-Weyl theorem we know that $L^2(K)$ decomposes into the direct sum
$\oplus_{\sigma\in \widehat{K}} V_\sigma\otimes V_\sigma^*$ in such a way that
the left regular representation decomposes as 
$\lambda_K\cong \oplus_{\sigma} \sigma\otimes 1_{V_\sigma^*}$
and the right regular representation decomposes as 
$\rho_K\cong \oplus_{\sigma} 1_{V_\sigma}\otimes \sigma^*$. 
(e.g. see \cite[Theorem 7.2.3]{DE} together with the obvious isomorphism $V_\sigma\otimes V_\sigma^*\cong \End(V_\sigma)$) This induces a decomposition
\begin{equation}\label{eq-decom}
L^2(G)\cong \H_{U^\lambda}\cong\oplus_{\sigma\in \widehat{K}} \H_{U^\sigma}\otimes V_\sigma^*.
\end{equation}
To see how the isomorphism $\Phi_\lambda$ restricts to the direct summand
$\H_{U^\sigma}\otimes V_\sigma^*$ we should note that the inclusion of 
$V_\sigma\otimes V_\sigma^*$ into $L^2(K)$ is given by sending 
an elementary vector $v\otimes w^*$ to the function $k\mapsto \sqrt{d_\sigma}\lk \sigma(k^{-1})v, w\rk$.
The corresponding inclusion of $\H_{U^\sigma}\otimes V_\sigma^*$ into $L^2(G)$
is therefore given by sending an elementary vector $\xi\otimes w^*$  
to the $L^2$-function $g\mapsto \sqrt{d_\sigma} \lk \xi(g), w\rk$. One can easily check directly, using the 
orthogonality relations for matrix coefficients on $K$, that this defines an 
isometry $\Phi_\sigma$ from $\H_{U^\sigma}\otimes V_\sigma^*$ into $L^2(G)$. 
To show that the image lies in $L^2(G)_\sigma$ we compute
\begin{align*}
\big(\rho(p_{\sigma^*})\Phi_\sigma(\xi\otimes v^*)\big)(g)&=
\int_K p_{\sigma^*}(k) \sqrt{d_\sigma}\lk \xi(gk), v\rk\, dk\\
&=\int_K p_{\sigma^*}(k)\sqrt{d_\sigma}\lk \xi(g),\sigma(k)v\rk\,dk\\
&=\sqrt{d_\sigma}\lk \xi(g), v\rk=\Phi_\sigma(\xi\otimes v^*)(g)
\end{align*}
which proves the claim.  Using \rm{(i)} and (\ref{eq-decom}), this also implies that the image is all of $L^2(G)_\sigma$.
One easily checks that $\Phi_\sigma$ intertwines 
the representation $1\otimes\sigma^*$ on $\H_{U^\sigma}\otimes V_\sigma^*$
with the restriction of $\rho$ to $K$. 

For the proof of (iii) we first observe that $T\in \cK(L^2(G))^K$ if and only if 
$T$ commutes with $\rho(k)$ for all $k\in K$, which then implies, via integration, 
 that $T$ commutes with  $\rho(p_{\sigma})$ for all $\sigma\in \widehat{K}$. It follows that 
 $\cK^K$ lies in $\oplus_{\sigma\in \widehat{K}} \cK(L^2(G)_\sigma)\subseteq \cK(L^2(G))$.
 Now, using the decomposition $L^2(G)_\sigma=\H_{U^\sigma}\otimes V_\sigma^*$ as in (ii) we get
 $$\cK(L^2(G)_\sigma)^{\Ad\rho(K)}\cong 
 \big(\cK(\H_{U^\sigma})\otimes \cK(V_{\sigma}^*)\big)^{\Id\otimes \Ad\sigma^*(K)}=
 \cK(\H_{U^\sigma})\otimes \C1_{V_{\sigma}^*}.$$
 
Finally, item \rm{(iv)} now follows from the fact that the restriction of the representation 
$M\times \lambda: C_0(G/K)\rtimes G\stackrel{\cong}{\to} \cK(L^2(G))^K$ to the 
subspace $\H_{U^\sigma}\otimes V_\sigma^*$ clearly coincides with $(P^\sigma\times U^\sigma)\otimes 1_{V_\sigma^*}$.
 \end{proof}

 \begin{example} The above decomposition becomes easier in the case where the 
compact subgroup $K$ of $G$ is abelian, since in 
this case the irreducible representations $\sigma$ of $K$ 
are one-dimensional. We therefore get $p_\sigma=\bar{\sigma}$ viewed as an element of 
$C(K)\subseteq C^*(K)$ and then, for $\xi\in L^2(G)$ we have 
$\xi\in L^2(G)_\sigma=\rho(p_\sigma)L^2(G)$ if and only if
$$\xi(g)=(\rho(p_\sigma)\xi)(g)=\int_K \bar{\sigma}(k)\xi(gk)\,dk$$
for almost all $g\in G$. For $l\in K$ we then get
\begin{align*}
\xi(gl)&=\int_K \bar{\sigma}(k)\xi(glk)\,dk\stackrel{k\mapsto l^{-1}k}{=}
\int_K \bar{\sigma}(l^{-1}k)\xi(gk)\,dk\\
&=\sigma(l)\xi(g),
\end{align*}
which shows that in this situation we have $L^2(G)_\sigma=\H_{U^\sigma}$ for all $\sigma\in \widehat{K}$. Thus we get the direct decompositions
$$L^2(G)=\oplus_{\sigma\in \widehat{K}} \H_{U^\sigma}\quad\text{and}\quad
\cK(L^2(G))^{K,\Ad\rho}=\oplus_{\sigma\in \widehat{K}} \cK(\H_{U^\sigma}).$$
This picture becomes even more transparent if $G$ happens also to be abelian. 
In that case one checks that the Fourier transform $\mathcal F:L^2(G)\to L^2(\widehat{G})$ 
maps the subspace $L^2(G)_\sigma$ of $L^2(G)$  to the subspace $L^2(\widehat{G}_\sigma)$ 
of $L^2(\widehat{G})$ in which 
$$\widehat{G}_\sigma:=\{\chi\in \widehat{G}: \chi|_K=\sigma\}$$
(we leave the details 
as an exercise to the reader). In the special case $G=\TT$ and $K=\C(n)$, the group of all $n$th roots 
of unity, we get $\widehat{\C(n)}\cong \ZZ/n\ZZ$ and, using Fourier transform, the composition 
of Lemma \ref{lem-fixed} becomes 
$$L^2(\TT)\cong \ell^2(\ZZ)\cong\oplus_{[l]\in \ZZ/n\ZZ} \ell^2(l\ZZ)
\quad\text{and}\quad \cK(L^2(\TT))^{C(n)}\cong \oplus_{[\l]\in \ZZ/n\ZZ} \cK(\ell^2(l\ZZ)).$$
\end{example}
\medskip

If $L\subseteq K$ are two compact subgroups of the locally compact 
group $G$, then we certainly have $\cK(L^2(G))^{\Ad\rho(K)}\subseteq \cK(L^2(G))^{\Ad\rho(L)}$.
For later use, it is important for us to have a precise understanding of  this inclusion.
In the following lemma we denote by $\Rep(L)$ the equivalence classes of 
{\bf all} unitary representations of a group $L$ and $\Rep(A)$ denotes the 
equivalence classes of all non-degenerate $*$-representation of a C*-algebra $A$.
In this notation we obtain a map 
$$\Ind_L^G: \Rep(L)\to\Rep(C_0(G/L)\rtimes G); \sigma\mapsto \Ind_L^G\sigma:=P^\sigma\times U^\sigma,$$
and similarly for $K$, with $P^\sigma$ and $U^\sigma$ defined as in (\ref{eq-ind1}) and (\ref{eq-ind2}).
Moreover, induction of unitary representations gives a mapping $\Ind_L^K:\Rep(L)\to\Rep(K)$
and the inclusion $C_0(G/K)\rtimes G$ into $C_0(G/K)\rtimes G$ induced by the obvious inclusion 
of $C_0(G/K)$ into $C_0(G/L)$ induces a mapping 
$$\Res_{G/L}^{G/K} : \Rep(C_0(G/L)\rtimes G)\to \Rep(C_0(G/K)\rtimes G).$$

\begin{lemma}\label{lem-ind-res}
Suppose that $L\subseteq K$ and $G$ are as above. Then the diagram
\begin{equation}\label{diagram-ind-res}
\begin{CD}
\Rep(K)  @>\Ind_K^G >>  \Rep(C_0(G/K)\rtimes G)\\
@A\Ind_L^K AA    @AA \Res_{C_0(G/L)}^{C_0(G/K)} A\\
\Rep(L) @>> \Ind_L^G > \Rep(C_0(G/L)\rtimes G)
\end{CD}
\end{equation}
commutes.
\end{lemma}
\begin{proof} Let $\sigma\in \Rep(L)$. Recall from (\ref{eq-ind1})  that the Hilbert space $\H_{U^\sigma}$ for the
induced representation $ \Ind_L^G\sigma=P^\sigma\times U^\sigma$ is defined as
$$\H_{U^\sigma}=\{\xi\in L^2(G, V_\sigma): \xi(gl)=\sigma(l^{-1})\xi(g)\; \forall g\in G, l\in L\},$$
and similar constructions  give the 
Hilbert spaces for the representations $\Ind_L^K\sigma$ and $\Ind_K^G\tau$ for some $\tau\in \Rep(K)$.
In particular, for $\tau:=\Ind_L^K\sigma$, we deduce  the formula
\begin{align*}
\H_{U^\tau}=\{\eta\in & L^2(G, L^2(K, V_\sigma)): \eta(gk, h)=\eta(g, kh)\\
&\quad\;\text{and}\;
\eta(g,hl)=\sigma(l^{-1})\eta(g,h)\;\forall g\in G, k,h\in K, l\in L\}.
\end{align*}
It is then straightforward to check that the operator
$$V: \H_{U^\sigma}\to \H_{U^\tau}; \;(V\xi)(g,h)=\xi(gh)$$
is a unitary  with inverse $V^{-1}$ given by $(V^{-1}\eta)(g)=\eta(g,e)$ 
such that $V$ intertwines the $G$-representations $U^\sigma$ and $U^\tau$. 
Moreover, for any $\varphi\in C_0(G/K)$ we
get 
\begin{align*}
(P^\tau(\varphi)V\xi)(g,h)&=\varphi(gK)(V\xi)(g,h)=\varphi(gK)\xi(gh)=\varphi(ghK)\xi(gh)\\
&=(P^\sigma(\varphi)\xi)(gh)=(VP^\sigma(\varphi)\xi)(g,h).
\end{align*}
This proves that that $V$ intertwines $\Res_{G/L}^{G/K} \circ \Ind_L^G\sigma$
with $\Ind_K^L\tau=\Ind_K^G\circ\Ind_L^K\sigma$.
\end{proof}

By Lemma \ref{lem-fixed} we have  isomorphisms 
$$C_0(G/K)\rtimes G\stackrel{M\rtimes\lambda}{\cong} \cK(L^2(G))^{\Ad\rho(K)}=
\oplus_{\tau\in \widehat{K}} \cK(\H_{U^\tau})\otimes 1_{V_\tau^*}$$
where the right equation is induced by the decomposition
$$L^2(G)= \oplus_{\tau\in \widehat{K}} \H_{U^\tau}\otimes V_\tau^*.$$
Thus we we see that, as a subalgebra of $\cK(L^2(G))$, the algebra 
$\cK(L^2(G))^{\Ad\rho(K)}$
decomposes in blocks of compact operators $\cK(\H_{U^\tau})$ such that each block
$\cK(\H_{U^\tau})$ appears with the multiplicity $\dim V_\tau$ in this decomposition.
If $L\subseteq K$, we get a similar
decomposition of $\cK(L^2(G))^{\Ad\rho(L)}$ indexed over all $\sigma \in \widehat{L}$ with
multiplicities $\dim V_\sigma$.
Since $L\subseteq K$ we get  $\cK(L^2(G))^{\Ad\rho(K)}\subseteq \cK(L^2(G))^{\Ad\rho(L)}$.
To understand this inclusion, we need to know how many copies of each block
$\cK(\H_{U^\tau})$, $\tau\in \widehat{K}$, appear in any given 
block $\cK(\H_{U^\tau})$, $\sigma\in \widehat{L}$, of the algebra $\cK(L^2(G))^{\Ad\rho(L)}$.
This multiplicity number $m_{\tau}^{\sigma}$ clearly coincides with the multiplicity
of the representation $\pi^\tau=P^\tau\times U^\tau$ of $C_0(G/K)\rtimes G\cong \cK(L^2(G))^{\Ad\rho(K)}$
in the restriction of the representation $\pi^\sigma=P^\sigma\times U^\sigma$ of
$C_0(G/L)\rtimes G\cong \cK(L^2(G))^{\Ad\rho(L)}$ to the subalgebra $C_0(G/K)\rtimes G$.
By the above lemma, this multiplicity coincides with the multiplicity of $\tau$ in the induced 
representation $\Ind_L^K\sigma$, and by Frobenius reciprocity, this  equals the multiplicity
of $\sigma$ in the restriction $\tau|_L$. Thus we conclude

\begin{proposition}\label{prop-include}
Suppose that $L\subseteq K$ are two compact subgroups of the locally compact group $G$. 
For each $\tau\in \widehat{K}$ and $\sigma\in \widehat{L}$ let $m^{\sigma}_{\tau}$ denote the 
multiplicity of $\sigma$ in the restriction $\tau|_L$. Then, under the inclusion 
 $\cK(L^2(G))^{\Ad\rho(K)}\subseteq \cK(L^2(G))^{\Ad\rho(L)}$ each block 
 $\cK(\H_{U^\tau})$ of $\cK(L^2(G))^{\Ad\rho(K)}$ appears with multiplicity $m_\tau^\sigma$ in each 
 block  $\cK(\H_{U^\sigma})$ of $\cK(L^2(G))^{\Ad\rho(L)}$.
\end{proposition}

We should note that since each block $\cK(\H_{U^\sigma})$ appears with multiplicity $\dim V_\sigma$ 
in the representation of  $\cK(L^2(G))^{\Ad\rho(L)}$ on $L^2(G)$, and similarly for $\tau\in \widehat{K}$,
we get the equation 
$$\dim V_\tau=\sum_{\sigma\in \widehat{L}} m_\tau^\sigma \dim V_\sigma,$$
for the total multiplicity $\dim V_\tau$ of  $\cK(\H_{U^\tau})$ in $\cK(L^2(G))^{\Ad\rho(K)}$.
Let us illustrate the above results in a concrete example:

\begin{example}\label{ex-D4-1} Let us consider the action of the finite group 
$G=D_4=\lk R,S\rk\subseteq \GL(2,\ZZ)$ with 
$R=\left(\begin{smallmatrix} 0&-1\\1&0\end{smallmatrix}\right)$ and 
$S=\left(\begin{smallmatrix} 1&0\\0&-1\end{smallmatrix}\right)$
on $\TT^2$ as described in Example \ref{ex-D4}. It was show in that example that we have
the topological fundamental domain 
$$Z:=\{(e^{2\pi i s}, e^{2\pi i t}): 0\leq t\leq\frac{1}{2}, 0\leq s\leq t\}$$
and it follows then from Proposition \ref{prop:fundamentaldomain} that the crossed
product $C_0(\TT^2)\rtimes G$ is isomorphic to the sub-homogeneous algebra
$$A:=\{f\in C(Z, \cK(\ell^2(G))): f(z,w)\in \cK(\ell^2(G)))^{\Ad\rho(G_{(z,w)})}\},$$
where $G_{(z,w)}$ denotes the stabilizer of the point $(z,w)$ under the action of $G$.
In what follows we will identify $Z$ with the triangle $\{(s,t)\in \RR^2:  0\leq t\leq\frac{1}{2}, 0\leq s\leq t\}$
and we write $G_{(s,t)}$ for the corresponding stabilizers of the points $(e^{2\pi i s}, e^{2\pi i t})$.
A straightforward computation shows that 
\begin{itemize}
\item $G_{(s,t)}=\{E\}$ if  $0< t <\frac{1}{2}, 0< s< t$,
\item $G_{(s,s)} =\lk RS\rk=:K_1$ if  $0<s<\frac{1}{2}$,
\item $G_{(0, t)}=\lk R^2S\rk=:K_2$ if $0<t<\frac{1}{2}$,
\item $G_{(s,\frac{1}{2})}=\lk S\rk=:K_3$ if  $0<s<\frac{1}{2}$,
\item $G_{(0,\frac{1}{2})}=\lk S, R^2\rk=:H$, and 
\item $G_{(0,0)}=G_{(\frac{1}{2},\frac{1}{2})} =G$.
\end{itemize}
It follows that $\cK(\ell^2(G))^{G_{(s,t)}}=\cK(\ell^2(G))\cong M_8(\CC)$ whenever 
$0< t <\frac{1}{2}, 0< s< t$.
In the three cases where $G_{(s,t)}=K_i$, $i=1,2,3$, is a subgroup of order two, 
 we get  two one-dimensional representations $\{1_{K_i}, \eps_{K_i}\}$, $i=1,2,3$,
 so by choosing  suitable bases of 
 $\ell^2(G)$,  in each case 
the algebra $\cK(\ell^2(G))^{G_{(s,t)}}$ has the form
$$\left(\begin{matrix} A_{1_{K_i}}&0\\ 0&A_{\eps_{K_i}}\end{matrix}\right)\quad A_{1_{K_i}},A_{\eps_{K_i}}\in M_4(\CC),$$
where the $4\times 4$-blocks act on the four-dimensional subspaces 
$\rho(p_{1_{K_i}})\ell^2(G)$ and $\rho(p_{\eps_{K_i}})\ell^2(G)$, respectively,
where $\rho$ denotes the right regular 
representation of $G$ restricted to the respective stabilizer.  We refer to the discussion before Lemma \ref{lem-fixed} for the definition of $p_{1_{K_i}}$ and $p_{\eps_{K_i}}$.

At the corner $(0,\frac{1}{2})$ we have the stabilizer $H=\lk S, R^2\rk\cong \ZZ/2\times\ZZ/2$, so we get
four one-dimensional representations $1,\mu_1,\mu_2,\mu_3$ of this group given by
$$\mu_1(R^2)=-\mu_1(S)=1,\quad \mu_2(R^2)=\mu_2(S)=-1\quad\text{and}\quad
\mu_3(R^2)=-\mu_3(S)=-1.$$
Therefore $\cK(\ell^2(G))^{G_{(0,\frac{1}{2})}}$ decomposes as
$$\left(\begin{matrix} B_1& & &\\ &B_{\mu_1}& & \\ & & B_{\mu_2} & \\ & & & B_{\mu_3}\end{matrix}\right)\quad B_1,B_{\mu_1},B_{\mu_2},B_{\mu_3}\in M_2(\CC)$$
with corresponding rank-two projections $\rho(p_{\mu})$, $\mu\in \widehat{H}$.

The representation theory of \(G\), the stabilizer of the remaining corners $(0,0)$ and $(\frac{1}{2},\frac{1}{2})$ of $Z$, is as follows: 
there is the `standard' representation 
\(\lambda \colon G\to O(2,\R) \subset U(2)\) (this is irreducible).
 The other irreducible representations are one-dimensional and correspond to
  the representations of the quotient group $G/\lk R^2\rk\cong \ZZ/2\times \ZZ/2$.
  They are listed as 
  $\{1_G, \chi_1,\chi_2,\chi_3\}$ with
  $$\chi_1(R)=-\chi_1(S)=1,\quad \chi_2(R)=\chi_2(S)=-1\quad\text{and}\quad
\chi_3(R)=-\chi_3(S)=-1.$$
Therefore 
the set of irreducible representations of \(G\) is 
 \(\{1, \chi_1,\chi_2,\chi_3, \lambda\}\).
Representing $C^*(G)\cong \cK(\ell^2(G))^{G}$ as a subalgebra of 
$\cK(\ell^2(G))\cong M_8(\CC)$ gives one block $M_2(\CC)$ with multiplicity $2$ and 
four one-dmensional blocks. With respect to a suitable chosen base of $\ell^2(G)\cong \CC^8$, 
we obtain a representation as matrices of the form
$$\left(\begin{matrix} 
C_\lambda&  & & & & \\
&C_\lambda& & & &\\
& & d_1& & & &\\
& & & d_{\chi_1}& & &\\
& & & & d_{\chi_2}& &\\
& & & & & d_{\chi_3}&
\end{matrix}\right)\quad C_\lambda\in M_2(\CC), d_1, d_{\chi_1},d_{\chi_2},d_{\chi_3}\in \CC,$$
where the lower diagonal entries act on the
images of the projections $\rho(p_{\chi})$
for $\chi\in \{1_G,\chi_1,\chi_2,\chi_3\}$ and the
block 
$\left(\begin{matrix}C_\lambda&  \\
& C_\lambda \end{matrix}\right)$
acts on the four-dimensional space $\rho(p_{\lambda})\ell^2(G)$.

To understand the structure of the algebra 
$$C(\TT^2)\rtimes G\cong\{f\in C(Z,\cK(\ell^2(G))): f(s,t)\in\cK(\ell^2(G))^{G_{(s,t)}}\},$$ 
 we need to understand 
what happens at the three corners $(0,0), (\frac{1}{2},\frac12), (0,\frac12)$ of the fundamental domain $Z$
when approached on the border lines of $Z$. 
So assume that $f\in C(\TT^2)\rtimes G$  is represented as a function
 $f:Z\to M_8(\CC)\cong \cK(\ell^2(G))$. Since the 
 stabilizers of the corners contain the stabilizers of the adjacent border lines,  we 
 see from the above discussion that the fibers $\cK(\ell^2(G))^{G_{(s,t)}}$ at the corners must be contained in the intersections of the
fibers at the  adjacent border lines. Proposition \ref{prop-include} above tells us, 
how these inclusions look like:
Let us consider the corner $(0,0)$. The adjacent border lines have stabilizers 
$K_1=G_{(s,s)}=\lk RS\rk$ and $K_2=G_{(0,t)}=\lk R^2S\rk$ respectively.
A short computation shows that the restriction of $\lambda$ to $K_1$ and $K_2$
 decomposes into  the direct sum  $1_{K_i}\oplus \eps_{K_i}$ for $i=1,2$. 
 So each of the two $4\times 4$-blocks
 in the decomposition 
 $$\cK(\ell^2(G))^{K_i}=\left\{\left(\begin{matrix} A_{1_{K_i}}&0\\ 0&A_{\eps_{K_i}}\end{matrix}\right):\quad A_{1_{K_i}}, A_{\eps_{K_i}}\in M_4(\CC)\right\},$$
contains exactly one copy of the two by two blocks $C_\lambda$  in the decomposition of $\cK(\ell^2(G))^G$.
Consider now the one-dimensional representations $1_G,\chi_1,\chi_2,\chi_3$ of $G$. If we 
restrict these representations to $K_1$ we see that $1_G$ and $\chi_2$ restrict to trivial
character $1_{K_1}$ and $\chi_1,\chi_3$ restrict to the non-trivial character $\eps_{K_1}$.
Thus, the $4\times 4$ block $A_{1_{K_1}}$ contains the diagonal entries
$d_1, d_{\chi_2}$  and the block $A_{\eps_{K_1}}$ contains the diagonal entries 
$d_{\chi_1},d_{\chi_3}$.

On the other side, if we restrict $1_G,\chi_1,\chi_2,\chi_3$ to the subgroup $K_2=G_{(0,t)}$
we see that $1_G,\chi_3$ restrict to  $1_{K_2}$ and $\chi_1,\chi_2$ 
restrict to the non-trivial character $\eps_{K_2}$. We therefore see that, different from the 
case $K_1=G_{(s,s)}$, $A_{1_{K_2}}$ contains the diagonal entries corresponding to the characters
$d_1, d_{\chi_3}$ and the block $A_{\eps_{K_2}}$ contains the diagonal entries corresponding to 
$d_{\chi_1}, d_{\chi_2}$. So, even in this simple example, we get a quite intricate structure of the algebra 
$C(\TT^2)\rtimes G$ at the fibers with nontrivial stabilizers. We shall revisit this example in the following
section.
\end{example}

We now proceed with the general theory:

\begin{definition}
\label{def:stabilizer_group_bundle}
Let $X$ be a proper $G$-space. We define the \emph{stabilizer group bundle} $\SX$ as
$$\SX=\{(x,g): x\in X, g\in G_x\}.$$
and we define
$$\SXhat:=\{(x,\sigma): x\in X, \sigma\in \widehat{G}_x\}.$$
\end{definition}

Note that $G$ acts on $\SXhat$ by 
$$g(x,\sigma)=(gx, g\sigma)\quad\text{with}\quad g\sigma:=\sigma\circ C_g^{-1},$$
where $C_g: G_{x}\to G_{gx}$ is the isomorphism given by conjugation with $g$.
For each $(x,\sigma)\in \SXhat$ consider the induced representation 
$\pi_x^\sigma=P_x^\sigma\times U_x^\sigma$ acting on the Hilbert space
$\H_{U^\sigma}$ as defined in (\ref{eq-ind1}) and (\ref{eq-ind2}).
One easily checks that $\pi_x^\sigma$ is unitarily equivalent to 
$\pi_{gx}^{g\sigma}$,
 the equivalence being given by the unitary
 $$W:\H_\sigma\to  \H_{\sigma\circ C_g^{-1}}; \;\big(W\xi\big)(t)=\sqrt{\Delta(g)}\xi(tg).$$
In what follows, we shall also write $\Ind_{G_x}^G(x,\sigma)$ for the representation $\pi_x^\sigma$,
indicating that it is the induced representation in the classical sense of Mackey, Glimm and others.
Thus, as a corollary of the above lemma, we obtain a proof of the following theorem,
which is  a well-known special case of the general Mackey-Green-Rieffel machine
for crossed products.

\begin{theorem}[Mackey-Rieffel-Green]\label{thm-mrg}
The map 
$$\Ind: \SXhat\to (C_0(X)\rtimes G)\dach; (x,\sigma)\mapsto \Ind_{G_x}^G(x,\sigma)=\pi_x^\sigma$$
factors through a set bijection (which we also denote $\Ind$) between the orbit space
$G\backslash\SXhat$ and the space $\big(C_0(X)\rtimes G\big)\dach$ of 
equivalence classes of irreducible $*$-representations of $C_0(X)\times G$.
\end{theorem}
\begin{proof} The proof is an easy combination of the bundle structure of 
$C_0(X)\rtimes G\cong C_0(X\times_{G,\Ad\rho}\cK)$ and the description of the 
fibre $\cK^{G_x}$ as given in the previous lemma. 
 \end{proof}
 {\ }
 
 If we look at  the trivial representation $1_{G_x}:G_x\to \{1\}$,
it follows from Lemma \ref{lem-fixed} that the corresponding summand of the fibre 
 $\cK^{G_x}$ of $C_0(X)\rtimes G$ is given by  $\cK(L^2(G)_{1_{G_x}})$
 with 
 \begin{equation}\label{eq-space}
 L^2(G)_{1_{G_x}}=\{\xi\in L^2(G): \xi(gk)=\xi(g)\;\forall k\in G_x\}\cong L^2(G/G_x).
 \end{equation}

Let $c:X\to [0,1]$ be a \emph{cut-off function} for the proper $G$-space $X$, which means 
that $c$ is a continuous function with compact support on any $G$-compact 
subset of $X$ (a closed subset $Y$ of $X$ is called {\em $G$-compact} if it is $G$-invariant
with $G\backslash Y$ compact)
such that 
$\int_G c(g^{-1}x)^2\, dg=1$ for all $x\in X$.
Then 
\begin{equation}\label{eq-projection}
p_X(g,x)=\sqrt{\Delta(g^{-1})}c(g^{-1}x)c(x) 
\end{equation}
determines a projection $p_X\in M(C_0(X)\rtimes G)$ via  convolution given by the 
same formula as for the multiplication on $C_c(G\times X)\subset C_0(X)\rtimes G$.
\emph{Note that $p_X\in C_0(X)\rtimes G$ if and only if $X$ is $G$-compact}, while 
in general, $f\cdot p_X\in C_0(X)\rtimes G$ for every $f\in C_0(G\backslash X)\subseteq
ZM(C_0(X)\rtimes G)$.

\begin{remark} 
\label{rem:cut-offs_finite}
We make the following remarks about cut-off functions and their associated 
projections. 
\begin{itemize}
\item If \(X\) is compact \(G\) must be compact too, and in this case 
with respect to normalised Haar measure on \(G\), 
the constant function \(c(x)\defeq 1\) for all 
\(x\in X\) is a cut-off function. In this case \(C^*(G)\) is a 
subalgebra of \(C(X)\rtimes G\), and the projection 
\(p_X\) is in the subalgebra: it is the constant function \(1\) on \(G\) 
and projects to the space of \(G\)-fixed
 vectors in any representation of 
\(G\). 
\item  For any proper action, the 
space of square roots of 
cut-off functions is convex, so the space of cut-off functions is contractible and 
hence any two 
projections \(p_X\) are homotopic. 
\end{itemize}

\end{remark}

\begin{lemma}\label{lem-projection}
Let $c:X\to [0,1]$ and $p_X\in M(C_0(X)\rtimes G)$ be as above and let 
$\Phi: C_0(X)\rtimes G\stackrel{\cong}{\longrightarrow} C_0(X\times_G\cK)$ denote the isomorphism
of Corollary \ref{cor-proper}. For each $x\in X$ let $c_x$ denote the unit vector in $L^2(G)$ given by $c_x(g)=\sqrt{\Delta(g^{-1})}c(gx)$ and let $p_x\in \cK(L^2(G))$ denote the image of $p_X$ under the 
evaluation map $q_x:C_0(X)\rtimes G\to \cK^{G_x}$ as described in Remark \ref{rem-evaluation}.
Then $c_x\in L^2(G)_{1_{G_x}}$ and $p_x$ is the orthogonal projection onto $\C \cdot c_x$.
\end{lemma}
\begin{proof}
The first assertion follows from the definition of $c_x$ together with the fact that the 
modular function vanishes on compact subgroups of $G$. 
For the second assertion observe that it follows from Remark  \ref{rem-evaluation}
that $p_x$ acts on $L^2(G)$ via convolution with  the function $p_x\in C_c(G\times G/G_x)
\subseteq C_0(G/G_x)\rtimes G$ given by 
$$
p_x(g, hG_x)={\Delta(g^{-1}h)}c_x(g^{-1}h)c_x(h).$$
If $\xi\in L^2(G)$ is arbitrary, we get
\begin{align*} (p_x\xi)(h)&=\int_G p_x(g, hG_x)\xi(g^{-1}h)\,dg\\
&=\int_G{\Delta(g^{-1}h)}c_x(g^{-1}h)c_x(h)\xi(g^{-1}h)\,dg\\
&=\left(\int_G c_x(g)\xi(g)\,dg\right)\,c_x(h)\\
&=\big(\lk  \xi, c_x \rk \, c_x\big)(h),
\end{align*}
where the second to last equation follows from the transformation $g^{-1}h\mapsto g$.
\end{proof}

It follows from the above lemma that the projection $p_X\in M(C_0(X)\rtimes G)$ constructed above
is a continuous field of rank-one projections on $G\backslash X$ such that under the 
decomposition of each fibre $\cK^{G_x}\cong \oplus_{\sigma\in \widehat{G}_x} \cK(\H_{\sigma})$ as in
part (iii) of Lemma \ref{lem-fixed}, the restriction of \(p_X\) to that fibre 
lies in the component $\cK(\H_{1_{G_x}})$.
It follows in particular that $C_0(G\backslash X)$ is isomorphic to the corner 
$p_X\big(C_0(X)\rtimes G\big)p_X$ via $f\mapsto f\cdot p_X$, and thus
$C_0(G\backslash X)\cong p_X\big(C_0(X)\rtimes G\big)p_X$ is 
Morita equivalent to the 
ideal $I_X=\big(C_0(X)\rtimes G\big)p_X\big(C_0(X)\rtimes G\big)$ 
of $C_0(X)\rtimes G$ generated by $p_X$ (this is a general fact about 
corners.)

Lemma \ref{lem-fixed}
implies that under the isomorphism $C_0(X)\rtimes G\cong C_0(X\times_G\cK)$ we get
\begin{equation}\label{eq-ideal}
I_X=\{F\in C_0(X\times_G\cK): F(x)\in \cK(L^2(G)_{1_{G_x}})\;\forall x\in X\}.
\end{equation}
It follows in particular that the ideal $I_X$ does not depend on the particular choice of the 
cut-off function $c:X\to [0,1]$ and the corresponding projection $p_X$.
If the action of $G$ on $X$ is {\em free and proper}, then 
it is immediate from \eqref{eq-ideal} and the description 
of \(C_0(X)\rtimes G\) in Corollary 
\ref{cor-proper} and the following Remark \ref{rem-evaluation}, that
 $I_X = C_0(X)\rtimes G$. We thus recover the well-known theorem, 
 due to Phillip Green (see \cite{Green0}) that 
$C_0(G\backslash X)\sim_M C_0(X)\rtimes G$ for a free and proper 
action of a locally compact group \(G\). 

\begin{example}
The above can be made rather explicit in the case of 
finite group actions. For definiteness, we let \(G = \Z/2\), 
\(X\) is compact. By 
Remark \ref{rem:cut-offs_finite}, we may take 
\(p_X\in C(X)\rtimes G\) to be \( p_X  (x,g) = \frac{1}{\abs{G}}\), 
while \(C(X)\rtimes G \cong C\bigl( X, \cK(\ell^2(G))\bigr)^G\). 
We can consider \(\cK(\ell^2(G))\) as \(2\times 2\)-matrices,
and the \(G\)-invariance says  the elements in $C\bigl( X, \cK(\ell^2(G))\bigr)^G$
 must have the
form
\[ a=  \left[\begin{matrix} f & g\\ \sigma (g) & \sigma (f) \end{matrix} \right]\]
for some \(f,g\in C(X)\), where $\sigma$ is the underlying order two automorphism 
of $C(X)$. The projection $p_X$ corresponds to
the matrix \( \frac{1}{\abs{G}} 
\left[\begin{matrix} 1 & 1\\ 1 & 1 \end{matrix} \right]\).
The ideal \(I_X =\big(C(X)\rtimes G\big)p_X\big(C(X)\rtimes G\big)\)
is then given by the closed linear hull of matrices of the form
$\left[\begin{matrix} fg& f\sigma(g)\\ \sigma(f)g &\sigma(fg)\end{matrix}\right]$, $f,g\in C(X)$,
while the corner \(p_X\bigl( C(X)\rtimes G\bigr)p_X\) consists of  all matrices of the 
form 
\(\left[\begin{matrix} f & f\\ f & f\end{matrix} \right]\)
where \(\sigma (f) = f\). This C*-algebra is isomorphic to 
\(C(G\backslash X)\). 
\end{example}

It is also useful to give a representation theoretic description of the ideal
$I_X$ in terms of the Mackey-Rieffel-Green machine of Theorem \ref{thm-mrg}.
For this recall that for any closed two-sided ideal $I$ in a C*-algebra $A$ 
the spectrum $\widehat{I}$ includes as an open subset of $\widehat{A}$ via 
(unique) extension of irreducible representations from $I$ to $A$, and the 
resulting correspondence $I\subseteq A\leftrightarrow \widehat{I}\subseteq\widehat{A}$ 
is one-to-one. Thus the ideal $I_X$ in $C_0(X)\rtimes G$
is uniquely determined by the set of irreducible representations of $C_0(X)\rtimes G$
which do not  vanish on $I_X$. The above results now combine to the following:

\begin{proposition}\label{prop-ideaIX}
The irreducible representations of $C_0(X)\rtimes G$ which correspond to 
the deal $I_X$ are precisely the representations of the form 
$\pi_x^{1_{G_x}}=\Ind_{G_x}^G(x, 1_{G_x})$, $x\in X$,
where $1_{G_x}$ denotes the trivial representation of the stabilizer $G_x$.
\end{proposition}

\begin{example}\label{ex-D4-2} 
The above representation theoretic description of the Ideal $I_X$ makes it easy to
identify this ideal in the case of the crossed product $C(\TT^2)\rtimes G$ with $G=D_4$
acting on $X=\TT^2$ as described in Examples \ref{ex-D4} an \ref{ex-D4-1}.
If we realize
$$C(\TT^2)\rtimes G=\{f\in C(Z,\cK(\ell^2(G))): f(s,t)\in \cK(\ell^2(G))^{G_{(s,t)}}\}$$
as in Example \ref{ex-D4-1}, then, with respect to the description of the fibers $\cK(\ell^2(G))^{G_{(s,t)}}$
as given in that example, the ideal $I_{\TT^2}$ consists of those functions which take arbitrary values
in the interior of $Z$ and which take 
values in the corners $A_1$ at the boundaries of $Z$ and in $d_1$ and $B_1$ at the corners
$(0,0)$, $(\frac12,\frac12)$ and $(0, \frac12)$, respectively. 
\end{example}

Recall from \S 1 that every proper $G$-space $X$ is locally induced by compact subgroups
of $G$, which means that each $x\in X$ has a $G$-invariant open neighborhood $U$ 
of the form $U\cong G\times_KY$. For later use we need to compare the 
Mackey-Rieffel-Green map of $C_0(G\times_KY)\rtimes G$ with that of $C_0(Y)\rtimes K$.
By a version of Green's imprimitivity theorem we know that
 $C_0(G\times_KY)\rtimes G$ is Morita equivalent to $C_0(Y)\rtimes K$.
The imprimitivity bimodule $E$ is given by a completion of $E_0=C_c(G\times Y)$ with underlying 
pre-Hilbert $C_c(K\times Y)$-structure given by
\begin{equation}\label{eq-inner}
\begin{split}
\lk \xi,\eta\rk_{C_c(K\times Y)}(k,y)&= \int_G \overline{\xi(g^{-1},y)}\eta(g^{-1}k, k^{-1}y)\,dg\\
\xi\cdot b(g,y)&= \int_K \xi(gk^{-1}, ky) b(k,ky)\, dk
\end{split}
\end{equation}
for $\xi,\eta\in E_0$ and $b\in C_c(K\times Y)\subseteq C_0(Y)\rtimes K$, 
and with left action of the dense subalgebra $C_c(G\times (G\times_KY))$ of 
$C_0(G\times_KY)\rtimes G$ on $X_0$ given by the covariant representation 
$(P, L)$ such that
\begin{equation}\label{eq-leftaction}(P(F)\xi)(g,y)=F([g,y])\xi(g,y)\quad\text{and}\quad (L(t)\xi)(g,y)=\Delta(t)^{1/2}\xi(t^{-1}g, y),
\end{equation}
for $F\in C_0(G\times_KY)$. These formulas follow from 
 \cite[Corollary 4.17]{Dana-book}  by identifying $C_0(G\times_KY)$ with $C_0(G\times_KC_0(Y))$
 ($=\Ind_K^GC_0(Y)$  in the notation of \cite{Dana-book}).
Induction
of representations from $C(Y)\rtimes K$ to $C_0(G\times_KY)\rtimes G$ via the imprimitivity
bimodule $E$ induces a homeomorphism between $\big(C_0(Y)\rtimes K\big)\dach$ and 
$\big(C_0(G\times_KY)\rtimes G\big)\dach$. 

\begin{proposition}\label{prop-commute}
Suppose that $K$ is a compact subgroup of $G$, $Y$ is a $\K$-space and  $X=G\times_KY$.
Then there is a commutative diagram of bijective maps
$$
\begin{CD}
K\backslash \SYhat   @> \iota >> G\backslash \SXhat\\
@V\Ind^K VV       @VV\Ind^G V\\
(C_0(Y)\rtimes K)\dach   @>>\Ind^E > (C_0(X)\rtimes G)\dach
\end{CD}
$$
where the vertical maps are the respective induction maps of Theorem \ref{thm-mrg},
the lower horizontal map is induction via the imprimitivity bimodule constructed above 
and the upper horizontal map is given by the map on orbit spaces induced by the 
inclusion $\iota: \SYhat\to \SXhat, (y,\sigma)\mapsto ([e,y], \sigma)$.
\end{proposition}
\begin{proof}
Let $(y,\sigma)\in \SYhat$. Let $\tau_y^\sigma$ denote the representation of 
$C_0(Y)\rtimes K$ induced from $(y,\sigma)$ and let $\pi_y^\sigma$ denote 
the representation of $C_0(G\times_KY)\rtimes G$ induced from $([e,y],\sigma)$.
Let  $\H_\sigma^K$ and $\H_\sigma^G$ denote the respective Hilbert spaces 
on which they act. 
We have to check that $\Ind^E\tau_y^\sigma\cong \pi_y^\sigma$.
Recall that $\Ind^E\tau_y^\sigma$ acts on the Hilbert space
$E\otimes_{C_0(Y)\rtimes K} \H_{\sigma}^K$ via the left action of $C_0(G\times_KY)\rtimes G$
on $E$, as specified in (\ref{eq-leftaction}).
Recall from (\ref{eq-ind1}) that
 $\H_\sigma^K=\{\ph\in L^2(K,V_\sigma): \ph(kl)=\sigma(l^{-1})\ph(k)\;\forall l\in K_y\}$
(since $K$ is unimodular) and similarly for $\H_{\sigma}^G$. We claim that there is a unique
unitary operator
$$\Phi: E\otimes_{C_0(Y)\rtimes K}\H_\sigma^K\to \H_\sigma^G$$
given on elementary tensors $\xi\otimes \ph$, $\xi\in E_0$, $\ph\in \H_\sigma^K$, by
$$\Phi(\xi\otimes \ph)(g)=\Delta(g)^{-1/2}\int_K\xi(gk^{-1}, ky)\ph (k)\,dk.$$
A quick computation shows that $\Phi(\xi\otimes\ph)(gl)=\Delta_G(l)^{-1/2}\sigma(l^{-1})\Phi(\xi\otimes \ph)(g)$.
To see that $\Phi$ preserves the inner products we compute for all $\xi,\eta\in E_0$ and 
$\ph,\psi\in \H_\sigma^K$:
\begin{align*}
\lk\Phi(\xi\otimes\ph),&\Phi(\eta\otimes\psi)\rk_{\H_\sigma^G}=
\int_G\lk \Phi(\xi\otimes\ph)(g),\Phi(\eta\otimes\psi)(g)\rk_{V_\sigma}\,dg\\
&=\int_G\Delta(g^{-1})\int_K\int_K\lk \xi(gk^{-1}, ky)\ph(k), \eta(gl^{-1}, ly)\psi(l)\rk_{V_\sigma} \,dl\,dk\,dg\\
&=\int_G\int_K\int_K\overline{\xi(g^{-1}k^{-1}, ky)}\eta(g^{-1}l^{-1}, ly)\lk \ph(k), \psi(l)\rk_{V_\sigma} \,dl\,dk\,dg\\
\end{align*}
while on the other side we get
$$
\lk \xi\otimes \ph,\eta\otimes\psi\rk_{E\otimes_{C_0(Y)\rtimes K}\H_\sigma^K}=
\lk \tau_y^\sigma(\lk\eta,\xi\rk_E)\ph,\psi\rk_{\H_\sigma^K}.
$$
For a function $f\in C_c(K\times Y)\subseteq C_0(Y)\rtimes K$ the operator $\tau_y^\sigma(f)$ acts on 
$\ph\in \H_\sigma^K$ by
$$\big(\tau_y^\sigma(f)\ph\big)(l)=\int_K f(k, ly)\ph(k^{-1}l)\,dk.$$
Applying this together with the formula for $\lk \xi,\eta\rk_E$ as given in (\ref{eq-inner}), we get
\begin{align*}
\lk \xi\otimes \ph,\eta\otimes&\psi\rk_{E\otimes_{C_0(Y)\rtimes K}\H_\sigma^K}=
\int_K \int_K\big\lk \lk \eta,\xi\rk_E(l,ky)\ph(l^{-1}k), \psi(k)\big\rk_{V_\sigma}\,dl\,dk\\
&=\int_K \int_K  \int_G  \overline{\xi(g^{-1}k,k^{-1}ly)}\eta(g^{-1}, ly)\lk \ph(k^{-1}l),\psi(l)\rk_{V_\sigma}\,
dg\,dl\,dk\\
&=\int_G\int_K\int_K\overline{\xi(g^{-1}k^{-1}, ky)}\eta(g^{-1}l^{-1}, ly)\lk \ph(k), \psi(l)\rk_{V_\sigma} \,dl\,dk\,dg\\
&=\lk\Phi(\xi\otimes\ph),\Phi(\eta\otimes\psi)\rk_{\H_\sigma^G},
\end{align*}
where the second to last equation follows from Fubini and  the transformation $g\mapsto lg$ followed
by the transformation $k\mapsto lk^{-1}$.

It follows now that $\Phi$ extends to a well defined isometry from $E\otimes_{C_0(Y)\rtimes K}\H_\sigma^K$ into  $\H_\sigma^G$. We now show that it intertwines the representations
$\Ind^E\tau_y^\sigma$ and $\pi_y^\sigma$. Since these representations are irreducible,
this will then also imply surjectivity of $\Phi$. For the left action of $F\in C_0(G\times_KY)$ we check
\begin{align*}
\Phi(P(F)\xi\otimes \ph)(g)&=\Delta(g)^{-1/2}\int_KF([gk^{-1}, ky])\xi(gk^{-1}, ky)\ph(k)\,dk\\
&=F([g,y])\Phi(\xi\otimes \ph)(g)=\big(P_{[e,y]}^\sigma(F)\Phi(\xi\otimes\ph)\big)(g),
\end{align*}
where we used the equations $[gk^{-1},ky]=[g,y]=g\cdot[e,y]$. Similarly, for the actions
of $G$ we easily check
$$\Phi(L(t)\xi\otimes\ph)(g)=\Phi(\xi\otimes\ph)(t^{-1}g)=\big(U_{y}^\sigma(t)\Phi(\xi\otimes \ph)\big)(g),$$
which now completes the proof.
\end{proof}

Recall that any $A-B$-imprimitivity bimodule $E$ induces a bijection of ideals in $A$ and $B$.
Under the correspondence between ideals in $A$ (resp. $B$) and open subsets of 
$\widehat{A}$ (resp. $\widehat{B}$) the correspondence of ideals of $A$ and $B$ 
induced by $E$ is compatible with the correspondence of open subsets in 
$\widehat{A}$ and  $\widehat{B}$ given by the homeomorphism 
$\ind^E:\widehat{B}\to \widehat{A}$. This all follows from the Rieffel-correspondence as explained in 
\cite[Chapter 3.3]{RW}. Using these facts together with the above 
Proposition \ref{prop-ideaIX} and Proposition \ref{prop-commute} we get

\begin{corollary}\label{cor-ideal}
Suppose that $K$ is a compact subgroup of $G$, $Y$ is a $\K$-space and  $X=G\times_KY$.
Then under the above described Morita equivalence between $C_0(X)\rtimes G$ and
 $C_0(Y)\rtimes K$ the ideal $I_X$ in $C_0(X)\rtimes G$ corresponds to the ideal 
 $I_Y$ in $C_0(Y)\rtimes K$.
 \end{corollary}

\begin{remark}
Analogues of Proposition \ref{prop-commute} and Corollary \ref{cor-ideal} are also true 
if the compact subgroup $K$ is replaced by any given closed subgroup $H$ of $G$ 
which acts properly on the space $Y$. The arguments are exactly the same---only some formulas 
become a bit more complicated due to the appearence of the modular function on $H$.
Since we only need the compact case below, we restricted to this case here.
\end{remark}

  \section{The spectrum of $C_0(X)\rtimes G$}
\label{sec-topology}
As  first step  to obtain any further progress for  $\K$-theory computations of crossed products by
proper actions with non-isolated free orbits, it should be useful to obtain a better understanding of the 
ideal structure of the crossed products. Since closed ideals in $C_0(X)\rtimes G$ correspond to open subsets of $(C_0(X)\rtimes G)\dach$, this problem is strongly related to a computation 
of the topology of the representation space $(C_0(X)\rtimes G)\dach$.

So in this section we  give a detailed 
description of the topology of $\big(C_0(X)\rtimes G\big)\dach$ 
in terms of the bijection with the parameter space  $G\backslash\SXhat$ 
as in Theorem \ref{thm-mrg}. To be more precise, we shall introduce a topology
on $\SXhat$ such that the bijection of Theorem \ref{thm-mrg} becomes a 
homeomorphism, if $G\backslash\SXhat$ carries the corresponding quotient topology.
Note that Baggett gives in \cite{Bag}
a general description of the 
topology of the unitary duals of semi-direct product groups $N\rtimes K$, with 
$N$ abelian and $K$ compact in terms of the Fell-topology on the set
of  subgroup representations of $K$. (This is the space of all pairs $(L,\tau)$, where $L$ is a subgroup of $K$ and $\tau$ is a unitary representation of $L$, see \cite{Glimm1, Fell} for the definition.)
Since
$(N\rtimes K)\dach=C^*(N\rtimes K)\dach\cong (C_0(\widehat{N})\rtimes K)\dach$
the study of such semi-direct products can be regarded as  a special case of the 
study of crossed products by proper actions.

Since Fell's topology on subgroup representations is not very easy to understand, we aim to define
a suitable topology on $\SXhat$ without using this construction. 
 To make this possible, we restrict our attention to actions which satisfy  Palais's slice property (SP). 
Recall that this means that the proper $G$-space $X$ is locally induced from actions 
of the stabilizers.
Recall also  that by Palais's Theorem (see Theorem  \ref{thm-Palais}), property (SP) is always satisfied if $G$ is a Lie-group.

Let us introduce some further
notation: 
If $X$ is a proper $G$-space with  property (SP), then for every $x\in X$ 
we define
$$S_x:=\{y\in X: G_y\subseteq G_x\}.$$
By an {\em almost slice} at $x\in X$ we shall understand any set of the form $W\cdot V_x$,
where $W$ is an open neighborhood of $e$ in $G$ and $V_x$ is an open neighborhood of
$x$ in $S_x$, i.e.,  $V_x=S_x\cap U_x$ for some open neighborhood $U_x$ of $x$ in $X$.
We denote by $\AS_x$ the set of all almost slices at $x$.

\begin{lemma}\label{lem-almostslice}
Let $X$ be a proper $G$-space with property (SP). Then 
the set $\AS_x$ of all almost slices at $x$ forms an open neighborhood base 
at $x$.
\end{lemma}
\begin{proof}
Let $WV_x$ be any almost slice at $x$. To see that it is open in $X$ let $y\in WV_x$ be arbitrary.
Let $g\in W$ such that $y\in gV_x$.
Let $Y_y$ be a local slice at $y$, i.e., there is an open neighborhood $U_y$ of $y$ such that 
$U_y=G\cdot Y_y\cong G\times_{G_y}Y_y$. 
Then $Y_y\subseteq S_y\subseteq S_{gx}$.
Thus, passing to a smaller local slice if necessary, we may assume that $Y_y\subseteq gV_x$.
Now choose an open neighborhood $W'$ of $e$ in $G$ such that $W'g\subseteq W$. Then 
$W'Y_y\subseteq W'gV_x\subseteq WV_x$ is an open neighborhood of $y$ 
contained in $WV_x$.
(Note that the map $G\times Y_y\to G\cdot Y_y$ is open
since it coincides with the quotient map 
$G\times Y_y\to G_y\backslash(G\times Y_y)=G\times_{G_y}Y_y$.)

Conversely, it follows from the continuity of the action that every open neighborhood $V$ of $x$
in $X$ contains an almost slice $WV_x$ at $x$, which finishes the proof.
\end{proof}
In what follows, if $\tau$ and $\sigma$ are representations, we write $\tau\leq \sigma$ 
if $\tau$ is a subrepresentation of $\sigma$. 

\begin{definition}\label{defn-topology}
Suppose that $X$ is a proper $G$-space which satisfies (SP). 
If $(x,\sigma)\in \SXhat$ and if $WV_x\in \mathcal \AS_x$
we say that a pair $(z,\tau)\in \SXhat$ 
lies in the set $U(x, \sigma, WV_x)$ if and only if there exists $y\in V_x$ and $g\in W$
such that $z=gy$ and $\tau\leq g\sigma|_{G_z}$. Alternatively one could describe the sets $U(x,\sigma, WV_x)$ as the product
$W\cdot U(x,\sigma, V_x)$ with
$$ U(x,\sigma, V_x):=\{(y,\tau)\in \SXhat: y\in V_x, \tau\leq \sigma|_{G_y}\}.$$
We further define
 $$\U_{(x,\sigma)}:=\{U(x,\sigma,WV_x): WV_x\in \mathcal \AS_x\}.$$
\end{definition}

%


%

\begin{lemma}\label{lem-topology}
There is a topology on $\SXhat$ such that the elements of
 $\U_{(x,\sigma)}$ form a base of open neighborhoods for the element 
 $(x,\sigma)\in \SXhat$ in this topology. Moreover, the canonical action of $G$ 
 on $\SXhat$ is continuous with respect to this topology.
 \end{lemma}
\begin{proof} 
We have to show that if  $(x,\sigma)$ lies in the intersection 
of two sets $U(x_1,\sigma_1, W_1V_1)$ and $U(x_2,\sigma_2, W_2V_2)$, then there 
exists an almost slice $WV$ at $x$ such that 
$$U(x,\sigma, WV)\subseteq U(x_1,\sigma_1, W_1V_1)\cap U(x_2,\sigma_2, W_2V_2).$$
If this is shown, then the union $\bigcup_{(x,\sigma)\in \SXhat} \mathcal U_{(x,\sigma)}$ forms a base
of a topology with the required properties.




For this let $g_1\in W_1$ and $g_2\in W_2$ such that $y_i:=g_i^{-1}x\in V_i$  and such that
$g_i^{-1}\sigma$ is a sub-representation of ${\sigma_i}|_{G_{y_i}}$ for $i=1,2$.
Since $g_iS_x =S_{g_ix}\subseteq S_{x_i}$ for $i=1,2$, we may choose an open neighborhood 
$V$ of $x$ in $S_x$ such that $g_iV\subseteq V_i$ for $i=1,2$. We may then also find 
a symmetric open neighborhood $W$ of $e$ in $G$ such that $g_iW\subseteq W_i$ and 
$WV\subseteq W_iV_i$ for $i=1,2$.
%
 We want to show that $U(x,\sigma, WV)\subseteq U(x_i,\sigma_i, W_iV_i)$
for $i=1,2$. Since $\U_{(x,\sigma)}$ is  closed under finite intersections, which follows easily from the definitions, 
 it is enough to show this for $i=1$.

So let $(y,\tau)\in U(x,\sigma, WV)$.
Let $g\in W$ such that $gy\in V$ and $g\tau$ is a sub-representation of $\sigma|_{G_{gy}}$. 
Then $g_1g\in W_1$ with $g_1gy\in V_1$ and 
 $g_1g\tau$ is a sub-representation of $g_1(\sigma|_{G_{gy}})$, which is a subrepresentation
 of $\big(\sigma_1|_{G_{g_1x}}\big)|_{G_{g_1gy}}=\sigma_1|_{G_{g_1gy}}$. Thus $(y,\tau)\in U(x_1, \sigma_1, W_1V_1)$.
 
 To see that the action of $G$ on $\SXhat$ is continuous let $U(gx,g\sigma, WV_gx)$ be a given neighborhood of 
 $(gx,g\sigma)$. Then, if $W_0$ is an open neighborhhood of $e$ in $G$ 
 such that 
 $gW_0^2g^{-1}\subseteq W$ and if we define $V_x:=g^{-1}\cdot V_{gx}$, we get 
 $(gW_0)\cdot U(x,\sigma, W_0V_x)=(gW_0^2g^{-1})\cdot U(gx,g\sigma, V_{gx})\subseteq 
 U(gx,g\sigma, WV_{gx})$ and we are done.
\end{proof}

\begin{remark}\label{rem-topdiscrete}
In the case where $G$ is a discrete group, the description of the topology on 
$\SXhat$ becomes easier to describe since for
every point $x\in X$ the set $S_x\subseteq X$ is open in $X$, and hence contains 
a neighborhood base of sets $V_x$. (Observe also that $V_x=WV_x$ for $W=\{e\}$.)
Therefore
a neighborhood base of a pair $(x,\sigma)\in \SXhat$ is given by the sets
$U(x,\sigma, V_x)$, where $V_x$ runs through all open neighborhoods of $x$ in $S_x$.
In this case, we may even replace the $V_x$ by local slices $Y_x$, since for discrete $G$
the local slices are also open in $X$.


\end{remark}

\begin{remark}\label{rem-neighborhood}
Another obvious approach to define basic neighborhoods for the topology of 
$\SXhat$ would be to consider the sets $U(x,\sigma, WY_x):=W\cdot U(x,\sigma, Y_x)$  
where the $Y_x$ are local slices at $x$ and 
$U(x,\sigma, Y_x):=\{(y,\tau): y\in Y_x, \tau\leq \sigma|_{G_y}\}$. 
 We actually believe that these sets do form a neighborhood base of the 
above defined topology, but we lack a proof. In particular, it is not clear to us whether the intersection
of two sets of this form will contain a third one of this form. The difficulty comes from the fact that 
the intersection of two local slices at $x$ might not be a local slice at $x$ -- in fact the intersection 
will very often only contain  the point $x$. However, it follows from the above remark that
these problems disappear if $G$ is discrete.
\end{remark}

In what follows, we always equip $(C_0(X)\rtimes G)\dach$  with the Jacobson topology
and $G\backslash \SXhat$ with the quotient topology of the above defined topology on $\SXhat$.

%

%
%
%
 
 \begin{theorem}\label{thm-top}
 Assume that $X$ is a proper $G$-space which satisfies Palais's slice property (SP).
 Let $\SXhat$ be equipped with the topology constructed above. Then the bijection
$[(x,\sigma)]\mapsto \pi_x^\sigma$  between $G\backslash \SXhat$ and $(C_0(X)\rtimes G)\dach$ of Theorem \ref{thm-mrg}
 is a homeomorphism.
 \end{theorem}
 
 Before we start with the proof, we have to recall the definition of the 
 \emph{Fell-topology} on the set $\Rep(A)$ of \emph{all} equivalence classes of 
 representations of a C*-algebra $A$ with dimension dominated by some fixed cardinal
 $\kappa$ ($\kappa$ is always chosen so big, that all representations we care for lie in 
 $\Rep(A)$). A neighborhood base for the Fell topology is given by
 the collection of all sets of the form
 $$U(I_1,\ldots, I_l)=\{\pi\in \Rep(A): \pi(I_i)\neq \{0\}\;\text{for all $1\leq i\leq l$}\},$$
 where $I_1,\ldots, I_l$ is any given finite collection of closed two-sided ideals in $A$.
If we restrict this topology to the set $\widehat{A}$ of  equivalence classes of irreducibe 
representations of $A$, we recover the usual Jacobson toplogy on $\widehat{A}$.

 Convergence of nets in $\Rep(A)$ can conveniently 
be described in terms of {\em weak containment}:
If $\pi\in \Rep(A)$ and $R$ is a subset of $\Rep(A)$, then 
$\pi$ is said to be {\em weakly contained} in $R$ (denoted $\pi\prec R$)
if $\ker\pi\supseteq\cap\{\ker\rho:\rho\in R\}$. 
Two subsets $S,R$ of $\Rep(A)$ are said to be {\em weakly equivalent}
($S\sim R$) if $\sigma\prec R$ for all $\sigma\in S$ and $\rho\prec S$ for
all $\rho\in R$. 

\begin{lem}[Fell]\label{lem-Felltop}
Let $(\pi_j)_{j\in J}$ be a net in $\Rep(A)$ and let $\pi,\rho\in \Rep(A)$. Then
\begin{enumerate}
\item $\pi_j\to\pi$ if and only if $\pi$ is weakly contained in every subnet
of $(\pi_j)_{j\in J}$.
\item If $\pi_j\to \pi$ and if $\rho\prec \pi$, then $\pi_j\to\rho$.
\end{enumerate}
\end{lem}
For the proof see \cite[Propositions 1.2 and 1.3]{Fell}.  We should also note that 
by construction of the Fell-topology, the topology can only see the kernel of a representation
and not the representation itself -- that means in particular that if we replace a 
 net $(\pi_i)$ by some other net $(\tilde\pi_i)$ with $\ker\tilde\pi_i=\ker\pi_i$ for all
 $i\in I$, then both nets have the same limit sets!

Suppose now that $A,B$ are two C*-algebras and let $_AE_B$ be a Hilbert $A-B$-bimodule.
By this we understand a Hilbert $B$-module $E_B$ together with a fixed $*$-homomorphism 
$\Phi:A\to \L_B(E)$. Then
$E$ induces an induction map (due to Rieffel)
$$\Ind^E: \Rep(B)\to \Rep(A)$$
which sends a representation $\pi\in \Rep(B)$ to the induced representation $\Ind^\pi\in \Rep(A)$,
which acts on the balanced tensor product $E\otimes_B\H_\pi$ via
$$\Ind^E\pi(a)(\xi\otimes v)=\Phi(a)\xi\otimes 1.$$
One can check that 
$$\ker(\Ind^E\pi)=\{a\in A: \Phi(a)E\subseteq E\cdot(\ker\pi)\}$$
from which it follows that induction via $_AE_B$ preserves weak containment, and 
hence the map $\Ind^E: \Rep(B)\to \Rep(A)$ is continuous. 
In particular, if $_AE_B$ is an imprimitivity bimodule with inverse module $_BE^*_A$,
then $\Ind^{E^*}$ gives a continuous inverse to $\Ind^E$, and therefore
induction via $E$ induces a homeomorphism between $\Rep(B)$ and $\Rep(A)$
(see \cite[Chapter 3.3]{RW}).

Basically all induction maps we use in this paper are coming in one way or the other
from induction via bimodules, so the above principles can be used. We need the
following observation:

\begin{lemma}\label{lem-repcompact}
Suppose that $(\pi_i)$ is a net in $\Rep(K)=\Rep(C^*(K))$ for some compact group 
$K$. Then $(\pi_i)$ converges to some $\sigma\in \widehat{K}$ if and only if
there exists an index  $i_0$ such that $\sigma$ is a sub-representation of $\pi_i$ for all $i\geq i_0$.
\end{lemma}
\begin{proof}
Using the Peter-Weyl theorem, we can write $C^*(K)=\oplus_{\tau\in \widehat{K}} \End(V_{\tau})$.
In this picture, given any representation $\pi$ of $K$, an irreducible representation $\tau$ is
a sub-representation of $\pi$ if and only if  the summand $\End(V_{\tau})$ is not in the kernel 
of $\pi$, viewed as a representation of $C^*(K)$. Thus weak containment 
and containment of $\tau$  are equivalent. 

Assume now that there exists no $i_0\in I$ 
such that $\tau\leq \pi_i$ for all $i\geq i_0$. We then construct a subnet $(\pi_j)$ of $(\pi_i)$
such that $\tau$ is not contained in any of the $\pi_j$, which then implies that $\End(V_\tau)$ 
lies in the kernel of all $\pi_j$. But then $\tau$ is not weakly contained in the subnet $(\pi_j)$ 
which by Lemma \ref{lem-Felltop} contradicts $\pi_i\to \tau$.
For the construction of the subnet, we define 
$$J:=\{(i,k)\in I\times I: k\geq i\;\text{and}\; \tau\not\leq \pi_k\},$$
equipped with the pairwise ordering. The projection to the second factor is clearly order preserving,
and if we define $\pi_{(i,k)}:=\pi_k$, we obtained a subnet $(\pi_{(i,k)})_{(i,k)\in J}$ with the desired 
properties.
\end{proof}

The following proposition will provide the main step towards the proof of Theorem \ref{thm-top}. 
 \begin{proposition}\label{prop-compact}
 Suppose that $K$ is a compact group acting on a locally compact space $Y$ and assume 
 that $y\in Y$ is fixed by $K$.
 Let $\sigma\in \widehat{K}$ be identified with the representation 
 $\pi_y^\sigma\in (C_0(Y)\rtimes K)\dach$ in the usual way and let 
 $(y_i, \sigma_i)$ be any net in $\SYhat$.
 Then the following are equivalent:
 \begin{enumerate}
 \item The net $\pi_{y_i}^{\sigma_i}=\Ind_{K_{y_i}}^K(y_i,\sigma_i)$ converges to 
 $\pi_y^\sigma$ in  $(C_0(Y)\rtimes K)\dach$.
 \item The net $(y_i,\sigma_i)$ converges to $(y,\sigma)$ in $\SYhat$.
 \end{enumerate}
 \end{proposition}
 
 Recall that for $\sigma\in \widehat{K}$ we denote by $p_\sigma=\dim(V_\sigma)\chi_{\sigma^*}$
 the central projection corresponding to $\sigma$.
 Before we give the proof, we should point out a well-known fact on isotypes 
 of unitary representations of compact groups: 
 if $\pi:K\to U(V_\pi)$ is any such representation, then $\pi(p_{\sigma})\in \B(V_\pi)$ 
  is the  projection onto the \emph{isotype} $V_{\pi}^\sigma$ of $\sigma$ in $\pi$, that is
 $$V_{\pi}^\sigma=\cup\{V\subseteq V_{\pi}: \pi|_V\cong \sigma\}.$$
 This follows easily from the fact that for any $\tau\in \widehat{K}$ we get
$\tau(p_\sigma) =1_{V_\sigma}$ if $\tau\cong \sigma$ and $\tau(p_\sigma)=0$ else,
and the well known fact that every representation 
of a compact group decomposes into irreducible ones.
 In particular, the isotype of an irreducible representation $\sigma\in \widehat{K}$ 
 in the left regular representation $\lambda: K\to U(L^2(K))$ is the finite dimensional space
 $V_\sigma\otimes V_{\sigma}^*$, viewed as a subspace of $L^2(K)$ via the 
 unitary embedding $v\otimes w^*\mapsto \xi_{v,w}$ with $\xi_{v,w}(k)=\sqrt{d_\sigma}\lk v,\sigma(k)w\rk$
 (compare with the proof of Lemma \ref{lem-fixed}). 
 
  \begin{proof}[Proof of Proposition \ref{prop-compact}]
  We may assume without loss of generality that $Y$ is compact, since otherwise 
 we may restrict ourselves to a $\K$-invariant compact neighborhood of $y$.
 We first show (i) $\Rightarrow$ (ii).  For this consider the canonical inclusion 
 $C^*(K)\to C(Y)\rtimes K$ which is induced by the $\K$-equivariant embedding $\CC\to C(Y);
\lambda\mapsto \lambda 1_Y$.
A representation $\pi=P\times U$ of $C(Y)\rtimes K$ restricts to the unitary representation 
$U$ of $K$ via this inclusion. Thus the map $\Rep(C(Y)\rtimes K)\to \Rep(K)$ which sends
$\pi=P\times U$ to $U$ can be viewed as an induction map via the 
$C^*(K)-C(Y)\rtimes K$-bimodule $C(Y)\rtimes K$, and therefore is continuous.

It follows that if $\pi_{y_i}^{\sigma_i}=P_{y_i}^{\sigma_i}\times U_{y_i}^{\sigma_i}$ 
(we use the notation of \S \ref{sec-mrg}) converges to $\sigma=\pi_y^\sigma$ as in (i),
then $U_{y_i}^{\sigma_i}$ converges to $\sigma$ in $\Rep(K)$, which by 
Lemma \ref{lem-repcompact} implies that there exists $i_0\in I$ such that 
$\sigma$ is a sub-representation of $U_{y_i}^{\sigma_i}$ for all $i\geq i_0$.
But by the Frobenius reciprocity theorem (e.g. see \cite[Theorem 7.4.1]{DE}) this 
implies that $\sigma_i$ is a sub-representation of $\sigma|_{K_{y_i}}$ for all 
$i\geq i_0$.

It remains to show that (i) implies that $y_i\to y$. For this we use the canonical inclusion 
of $C(K\backslash Y)$ into the center of $M\big(C(Y)\rtimes K\big)$. The associated 
``induction map'' from $\Rep(C(Y)\rtimes K)$ to $\Rep(C(K\backslash Y))$ 
sends the representation $\pi_{y_i}^{\sigma_i}=P_{y_i}^{\sigma_i}\times U_{y_i}^{\sigma_i}$
to the restriction of $P_{y_i}^{\sigma_i}$ to $C(K\backslash Y)\subseteq C(Y)$, which is equal 
to $\ev_{Ky_i}\cdot 1_{\H_{\sigma_i}}$. By continuity of ``induction'' we see that
$\ev_{Ky_i}\cdot 1_{\H_{\sigma_i}}$ converges to $\ev_{Ky}\cdot 1_{V_\sigma}$ in $\Rep(C(K\backslash Y))$
which just means that $Ky_i\to Ky=\{y\}$ in $K\backslash Y$.
Since $K$ is compact and $y$ is fixed by $K$, this implies that
$y_i\to y$ in $Y$.

We now  prove  (ii) $\Rightarrow$ (i). For this suppose that $(y_i,\sigma_i)$ is as 
in (ii). Since every subnet enjoys the same properties, it is enough to show that 
$\sigma=\pi_y^\sigma$ is weakly contained in $\{\pi_{y_i}^{\sigma_i}:i\in I\}$.

For this let $a\in C(Y)\rtimes K$ be any element 
such that $\pi_{y_i}^{\sigma_i}(a)=0$ for all $i\in I$. We need to show that 
$\pi_y^\sigma(a)=0$, too. For this recall first from Theorem \ref{thm-crossed}  that
$C(Y)\rtimes K\cong C(Y\times_K\cK(L^2(K)))$ is a continuous C*-algebra bundle
over $K\backslash Y$ with fiber $\cK(L^2(K))^{K_{y_i}}$ at
the orbit $Ky_i$. The projection of $C_0(Y)\rtimes K$ to the fiber 
$\cK(L^2(K))^{K_{y_i}}$ at
the orbit $Ky_i$ is given via 
 the representation $M_{y_i}\times \lambda$ with
$\lambda$ the regular representation of $K$ and
$$(M_{y_i}(\ph)\xi)(k)=\ph(ky_i)\xi(k).$$
It follows from this that the 
representation of  $a\in C(Y)\rtimes K$ on these fibers 
is given by a net $(a_{y_i})$ in $\cK(L^2(K))^{K_{y_i}}\subseteq \cK(L^2(K))$ which 
converges in the operator norm to
the element $a_y\in \cK(L^2(K))^K$. At the point $y$ we get the decomposition
$$C^*(K)\cong \cK(L^2(K))^K=\oplus_{\pi\in \widehat{K}} \cK(V_\pi)\otimes 1_{V_\pi^*}$$
and at all other points we get the decompositions
$$\cK(L^2(K))^{K_{y_i}}=\oplus_{\tau\in \widehat{K}_{y_i}}\cK(\H_{\tau})\otimes 1_{V_\tau^*},$$
where, for convenience, we write $\H_\tau$ for the Hilbert space $\H_{U^\tau}$ of $U^\tau$.

In particular, if $p_i:L^2(K)\to \H_{\sigma_i}\otimes V_{\sigma_i}^*\subseteq L^2(K)$ denotes the orthogonal
projection, then the 
representation $\pi_{y_i}^{\sigma_i}\otimes 1_{V_{\sigma_i}^*}=
(P_{y_i}^{\sigma_i}\times U_{y_i}^{\sigma_i})\otimes 1_{V_{\sigma_i}^*}$
is given by sending the element $a\in C(Y)\rtimes K$ to the element 
$p_{i}a_{y_i}p_i\in \cK(\H_{\sigma_i})\otimes 1_{V_{\sigma_i}^*}\subseteq \cK(L^2(K))$. 

Assume now that $a\in C(Y)\rtimes K$ such that $\pi_{y_i}^{\sigma_i}(a)=0$ for all $i\in I$.
Then all those elements $p_{i}a_{y_i}p_i$
vanish. We have to show that $\sigma(a_y)$ will vanish, too. 
For this recall that $\sigma(p_{\sigma})=1_{V_\sigma}$.
Thus we may replace $a$ by $p_\sigma a p_\sigma$, where 
we view $p_\sigma$ as an element of $C(Y)\rtimes K$ via the
embedding $C^*(K)\to C(Y)\rtimes K$.

Writing $1_i=1_{V_{\sigma_i}^*}$ we then have
$$\pi_{y_i}^{\sigma_i}(a)\otimes 1_i=\pi_{y_i}^{\sigma_i}(p_\sigma a p_\sigma)\otimes 1_i
=(U_{y_i}^{\sigma_i}(p_\sigma)\otimes 1_i)(\pi_{y_i}^{\sigma_i}(a)\otimes 1_i)(U_{y_i}^{\sigma_i}(p_\sigma)\otimes 1_i),$$
where $U_{y_i}^{\sigma_i}(p_\sigma)\otimes 1_i$ is the orthogonal projection  
from $\H_{\sigma_i}\otimes V_{\sigma_i}^*$ onto the isotype
$$W_i:=(\H_{\sigma_i}\otimes V_{\sigma_i}^*)^\sigma=
(\H_{\sigma_i}\otimes V_{\sigma_i}^*)\cap (V_\sigma\otimes V_\sigma^*)\subseteq L^2(K).$$
It thus follows that $\pi_{y_i}^{\sigma_i}\otimes 1_i$ represents  each $a_i:=a_{y_i}$ as an 
operator on the subspace $W_i$ of the fixed finite dimensionl space $V_\sigma\otimes V_\sigma^*$ of $L^2(K)$.

By Frobenius reciprocity, the  condition that 
$\sigma_i$ is a sub-representation of 
$\sigma|_{K_{y_i}}$ implies 
that $\sigma$ is a sub-representation of the representation 
$U_{y_i}^{\sigma_i}$, which implies that $W_i$ is non-zero for all $i\in I$.
Now let $q_i:V_\sigma\otimes V_\sigma^*\to W_i$ 
denote the orthogonal projection. Since $V_\sigma\otimes V_\sigma^*$ is finite dimensional,
we may pass to a subnet if necessary to assume that $q_i$ converges to some 
non-zero projection $q:V_\sigma\otimes V_\sigma^*\to W$. 
We then get $0=\pi_{y_i}^{\sigma_i}(a)\otimes 1_i= q_ia_{y_i}q_i\to qa_yq\in \cK(V_\sigma\otimes V_\sigma^*)$.
Moreover, since all 
$W_i$ are $\K$-invariant, the same is true for $W$. It follows that the representation 
$\lambda: C^*(K)\to \cK(L^2(K))^K$ restricts to a multiple of $\sigma$ on $W$.
But then we have $\sigma(a_y)=0\Leftrightarrow qa_yq=0$ and the result follows.
\end{proof}

\begin{proof}[Proof of Theorem \ref{thm-top}]
We first show that the map from $\SXhat$ to $(C_0(X)\rtimes G)\dach$ which assigns
$(x,\sigma)$ to the induced representation $\pi_x^\sigma$ is continuous. By the universal
property of the quotient topology on $G\backslash \SXhat$, this will imply that the map
$$\Ind: G\backslash \SXhat\to (C_0(X)\rtimes G)\dach$$
of Theorem \ref{thm-top} is continuous, too.
So let $(x_i,\sigma_i)$ be a net in $\SXhat$ which converges to some $(x,\sigma)\in \SXhat$.
By the definition of the topology on $\SXhat$ we may assume that $x_i=g_iy_i$ with $g_i\to e$ in $G$,
$y_i\in S_x$ (i.e., $G_{y_i}\subseteq G_x$) such that $y_i\to x$ and $\sigma_i=g_i\tau_i$ with $\tau_i\in \widehat{G}_{y_i}$ such that $\tau_i\leq \sigma|_{G_{y_i}}$. 
Since induction is constant on $G$-orbits, we may assume without loss of generality
that $y_i=x_i$ and $\tau_i=\sigma_i$. It follows then from Proposition \ref{prop-compact}
that the induced representations $\ind_{G_{x_i}}^{G_x}(x_i,\sigma_i)$ converge to 
 $(x,\sigma)$ in $(C_0(X)\rtimes G_x)\dach$. Using induction in steps and continuity of induction 
from a fixed subgroup (due to the fact that this induction can be performed via an appropriate 
Hilbert module), it then follows that
$$\pi_{x_i}^{\sigma_i}=\ind_{G_{x_i}}^G(x_i, \sigma_i)=
\ind_{G_x}^G\big(\ind_{G_{x_i}}^{G_x}(x_i, \sigma_i)\big)\to \ind_{G_x}^G(x, \sigma)=\pi_x^\sigma.$$


Conversely, assume that we have a net $(\pi_i)$ in $(C_0(X)\rtimes G)\dach$ which 
converges to some representation $\pi\in (C_0(X)\rtimes G)\dach$. 
Choose $x_i\in X$ and $\sigma_i\in (C_0(X)\rtimes G)\dach$ such that 
$\pi_i=\pi_{x_i}^{\sigma_i}$, and similarly we realize $\pi$ as $\pi_x^\sigma$ 
for some $(x,\sigma)\in \SXhat$. 
By imbedding
$C_0(G\backslash X)$ into $ZM(C_0(X)\rtimes G))$ we can see that 
$Gx_i\to Gx$ in $G\backslash X$. Since the quotient map $X\to G\backslash X$ is 
open, we can therefore pass to a subnet, if necessary, to assume that 
$g_ix_i\to x$ for a suitable net $g_i\in G$. Choosing a local slice $Y$ at $x$,
we may even adjust the situation to guarantee that $g_ix_i\in Y$ for all $i\in I$.
Thus, replacing $(x_i,\sigma_i)$ by $(g_ix_i, g_i\sigma_i)$ we may assume
from now on that
the net $(x_i,\sigma_i)$ and the pair $(x,\sigma)$ all lie in $\SYhat$.
Using the canonical Morita equivalence 
 $C(G\times_{G_x}Y)\rtimes G\sim_MC_0(Y)\rtimes G_x$ together with
Proposition \ref{prop-commute} it follows that
$\tau_{x_i}^{\sigma_i}=\Ind_{G_{x_i}}^{G_x}(x_i,\sigma_i)$ converges to 
$\tau_x^\sigma=\sigma$ in $(C_0(Y)\rtimes G_x)\dach$.
It is then a consequence of Proposition \ref{prop-compact} that
$(x_i,\sigma_i)$ converges to $(x,\sigma)$ in $\SYhat$.
But this also implies convergence in $\SXhat$ and thus of the respective orbits in
$G\backslash \SXhat$.
\end{proof}

Before discussing an example, we emphasize the 
following consequences of the above discussion. 

\begin{itemize}
\item For a 
proper action of a discrete group \(G\) on \(X\), 
a subset \(U \subset \SXhat\) is 
open if and only if the following is true: if \( (x,\pi) \in 
U\), then there exists an open slice \(V_x\) 
around \(x\) such that \(U\) 
contains all points \( (y,\tau)\) such that \(y\in V_x\) and 
\(\tau \le \pi |_{G_y}\). 
\item In any case, the set \(X\cong \{(x, 1_{G_x}): x\in X\}\subset \SXhat\) is open in this topology, and 
by projection, 
\(G\backslash X\) (identified with $G\backslash\{(x, 1_{G_x}):x\in X\}$) is open in \(G\backslash \SXhat\). 
\end{itemize}


\begin{example}
\label{ex:dihedral}
We come back to our example $C(\TT^2)\rtimes G$ with $G=D_4=\lk R,S\rk$ as 
discussed in Examples \ref{ex-D4-1}, \ref{ex-D4-1} and \ref{ex-D4-2}.
Consider the topological fundamental domain
$$Z:=\{(e^{2\pi i s}, e^{2\pi i t}): 0\leq t\leq\frac{1}{2}, 0\leq s\leq t\}\subseteq \T^2$$
for the action of $G$ on $\TT^2$. Recall that this means that 
the canonical map from $Z$ to $G\backslash \TT^2$ is a homeomorphism.
It then is easily checked that 
$$\SZhat=\{\big(z,w),\sigma\big)\in \SThat: (z,w)\in Z\}$$
is a topological fundamental domain for the action of $G$ in $ \SThat$.

Thus, in order to describe the topology of $C_0((\TT^2)\rtimes G)\dach\cong G\backslash\SThat$ it suffices 
to describe the topology of $\SZhat$. In what follows we identify $Z$ with the triangle
$$\{(s,t)\in \RR^2: 0\leq t\leq\frac{1}{2}, 0\leq s\leq t\}.$$
In Example \ref{ex-D4-1} we already computed the stabilizers and their representations.

Each point in the interior 
$\stackrel{\circ}{Z}$ has trivial stabilizer $\{E\}$. Since the trivial group has only the trivial representation
and since every representation of a group restricts 
obviously to a multiple of the
trivial representation of the trivial group, we see that if $(s_n,t_n)_{n\in \NN}$ is a sequence in 
$ \stackrel{\circ}{Z}$ which converges to some $(s,t)\in Z$, then 
$\big((s_n,t_n), 1_{\{E\}})$ converges to any $\big((s,t), \sigma\big)$ with $\sigma \in \widehat{G}_{(s,t)}$.
If $(s_n,t_n)$ converges to $0$, then the sequence $\big((s_n,t_n), 1_{\{E\}})$ 
has the five limit points $\big((0,0), 1_G\big),\big((0,0), \chi_1\big),\big((0,0), \chi_2\big),\big((0,0), \chi_3\big), \big((0,0), \lambda\big)$ with $\{1_G,\chi_1,\chi_2,\chi_3,\lambda\}$ 
as defined in Example \ref{ex-D4-1}.

Let us now restrict our attention to the three border lines. For example, if we consider the 
 line segment  $I_1:=\{(s,s): 0< s< \frac12\}$ we have the constant stabilizer $K_1=\lk RS\rk$
and we see that
$$\{\big((s,s),\sigma\big)\in \SZhat: (s,s)\in I_1\}=I_1\times\{1_{K_1}, \eps_{K_1}\}$$
topologically. Similar descriptions hold for the line segments 
$I_2:=\{(0,t): 0<t<\frac12\}$ and $I_3:=\{(s,\frac12): 0<s<\frac12\}$.

In order to describe the topology of $\SZhat$ at the corners 
 $(0,0), (\frac12,\frac12)$ and $(0,\frac12)$ of $Z$, we observe that
\begin{align*}
1_{K_1} \le \res^G_{K_1}(1_G),  \res^G_{K_1}(\chi_2)), \res^G_{K_1}(\lambda),&\;\textup{and}\; 
\eps_{K_1} \le  \res^G_{K_1}(\chi_1), \,
\res^G_{K_1}(\chi_3), \res^G_{K_1}(\lambda),\\
1_{K_2} \le \res^G_{K_2}(1_G),  \res^G_{K_2}(\chi_3)), \res^G_{K_1}(\lambda),&\;\textup{and}\; 
\eps_{K_2} \le \res^G_{K_2}(\chi_1), \,
\res^G_{K_2}(\chi_2), \res^G_{K_1}(\lambda),\\
1_{K_3} \le \res^G_{K_3}(1_G), \res^G_{K_3}(\chi_3)), \res^G_{K_1}(\lambda),&\;\textup{and}\; 
\eps_{K_3} \le  \res^G_{K_3}(\chi_1), \,
\res^G_{K_3}(\chi_2), \res^G_{K_1}(\lambda),
\end{align*}
which implies, for example, that if a sequence $(s_n,s_n)_{n\in \NN}\subseteq I_1$ converges to $(0,0)$ 
the sequence $\big((s_n,s_n), 1_{K_1}\big)$ has the 
limit points $\big( (0,0),1_G\big), \big( (0,0),\chi_2\big), \big( (0,0),\lambda\big)$ 
and the sequence $\big((s_n,s_n), \eps_{K_1}\big)$ has the 
limit points $\big( (0,0),\chi_1\big), \big( (0,0),\chi_3\big),$ $\big( (0,0),\lambda\big)$.
Similar descriptions follow from the above list for sequences in $I_1, I_2, I_3$ which converge 
to any of the corners $(0,0)$ or $(\frac12,\frac12)$.

At the corner $(0,\frac12)$ with stabilizer $H=\lk R^2, S\rk$ and characters $1_H,\mu_1,\mu_2,\mu_3$
of $H$, as described in Example \ref{ex-D4-1}, we get the relations
\begin{align*}
1_{K_2} = \res^H_{K_2}(1_H), \res^H_{K_2}(\mu_2),&\;\textup{and}\; 
\eps_{K_2} = \res^H_{K_2}(\mu_1),\, \res^H_{K_2}(\mu_3), \\
1_{K_3} = \res^H_{K_3}(1_H), \res^H_{K_3}(\mu_3), &\;\textup{and}\; 
\eps_{K_3} = \res^H_{K_3}(\mu_1),\, \res^H_{K_3}(\mu_2), 
\end{align*}
which implies, for example,  that for a sequence $(s_n,\frac12)$ in $I_3$ which converges to
$(0,\frac12)$, the sequence 
$\big((s_n,\frac12), 1_{K_3}\big)$ has the 
limit points $\big( (0,\frac12),1_H\big), \big( (0,\frac12),\mu_3\big)$ 
and the sequence $\big((s_n,\frac12), \eps_{K_3}\big)$ has the 
limit points $\big( (0,\frac12),\mu_1\big), \big( (0,\frac12),\mu_2\big)$.

From this description we obtain an increasing sequence of open subsets
of $\SZhat$ 
$$\emptyset=\mathcal O_0\subseteq \mathcal O_1\subseteq \mathcal O_2
\subseteq \mathcal O_3=\SZhat$$
(and corresponding open subsets of $(C_0(\TT^2)\rtimes G)\dach$ via induction) 
with
\begin{align*}
\mathcal O_1&=\{((x,y), 1) \mid  (x,y) \in Z\}\cong Z\\
\mathcal O_2\smallsetminus \mathcal O_1&=
\big\{\big((s,s),\eps_{K_1}\big): 0<s<1/2\big\}\cup
\big\{\big((0,t),\eps_{K_2}\big): 0<t<1/2\big\}\\
&\quad \cup
\big\{\big((s,1/2),\eps_{K_3}\big): 0<s<1/2\big\}\\
&\quad \cup
\big\{\big((0,0),\chi_1\big), \big((1/2,1/2),\chi_1\big), \big((0,1/2),\mu_1\big)\big\}\cong \partial Z\\
\mathcal O_3\smallsetminus \mathcal O_2&=
\{( 0,0), (1/2,1/2)\}\times \{ \chi_2, \chi_3, \lambda\}\cup \{(0,1/2)\}\times \{\mu_2,\mu_3\}\cong \{1,\ldots, 8\}.
\end{align*}
 Thus we obtain a corresponding sequence of ideals
$$\{0\}=I_0\subseteq I_1\subseteq I_2\subseteq I_3=C_0(\TT^2)\rtimes G$$
with 
$$\widehat{I_1}\cong Z,\quad\widehat{I_2/I_1}\cong \partial Z,\quad \widehat{I_3/I_2}\cong\{1,\ldots, 8\}.$$
Indeed, if we combine the above description of the topology of $\SZhat\cong \big(C(\TT^2)\rtimes G\big)\dach$
with the description of the algebra 
$$C(\TT^2)\rtimes G\cong \{f\in C\big(Z,\cK(\ell^2(G))\big): f(s,t)\in \cK(\ell^2(G))^{G_{(s,t)}}\}$$
as given in Example \ref{ex-D4-1}, we see that 
$I_1=I_{\TT^2}$ as described in Example \ref{ex-D4-2}.
 It is Morita-equivalent to $C(Z)\cong C(G\backslash \TT^2)$ and  consists of those functions $f\in 
 C(Z,\cK(\ell^2(G)))$ which only  take non-zero values in the matrix blocks corresponding to 
 the trivial representations in each fiber $\cK(\ell^2(G))^{G_{(s,t)}}$.
 The ideal $I_2$ takes arbitrary values in $\cK(\ell^2(G))^{G_{(s,t)}}$ at each
  $(s,t)\in Z\smallsetminus \{(0,0), (\frac12,\frac12), (0,\frac12)\}$ and it takes 
  non-zero values only in the diagonal entries $d_1, d_{\chi_1}$  at the corners 
  $(0,0)$ and $(\frac12,\frac12)$, with respect to the 
  block decomposition of $\cK(\ell^2(G))^G$ as given in Example \ref{ex-D4-1}, and it takes
  only non-zero values in the $2\times 2$ block entries $B_1, B_{\mu_1}$ at the corner $(0,\frac12)$.
 It is then clear that $I_2/I_1$ is isomorphic to the algebra of continuous functions 
 $f: \partial Z\to M^8(\CC)$ such that on the line segments $I_i$, $i=1,2,3$, $f(s,t)$ has non-zero entries only in the $4\times 4$ matrix blocks $A_{\eps_{K_i}}$, at the corners $(0,0)$ and 
 $(\frac12,\frac12)$ it has non-zero entry only in the diagonal entry $d_{\chi_1}$
and at the corner   $(0,\frac12)$ it has non-zero entries only in the $2\times 2$ block
 $C_{\mu_1}$.
 It follows that $I_2/I_1$ is Morita equivalent to $C(\partial Z)$. Finally, the quotient 
 $(C(\TT^2)\rtimes G)/I_2$ is  isomorphic to $M_2(\CC)^4\oplus\CC^4$.
 \end{example}

\begin{remark}\label{rem-trivialization}
If $G$ acts freely and properly on $X$, then Green's theorem (\cite{Green0})
implies that $C_0(X)\rtimes G\cong C_0(X\times_G\cK)$ is Morita equivalent to 
$C_0(G\backslash X)$. If everything in sight is second countable, this implies that 
the bundle $C_0(X\times_G\cK)$ is (stably) isomorphic to the trivial bundle 
$C_0(G\backslash X,\cK)$ (after stabilization, if necessary, we may assume 
$\cK\cong \cK(\ell^2(\NN))$ everywhere).

So one may wonder, whether a similar result can be true in general, i.e., is there a chance
to show that for any proper action of $G$ on a locally compact space $X$ the 
bundle $C_0(X\times_G\cK)$ is stably isomorphic to a subbundle of the trivial bundle
$C_0(G\backslash X,\cK)$, so that the fibre over the orbit $Gx$ would be 
a suitable subalgebra of $\cK=\cK(\ell^2(\NN))$ stably isomorphic to $\cK(L^2(G))^{\Ad\rho(G_x)}$.
Actually it follows from
Proposition \ref{prop:fundamentaldomain} that this is 
always the case  if there exists a topological fundamental domain $Z\cong G\backslash X$ for the action 
of $G$ on $X$ as in Definition \ref{def-fundamentaldomain}.

Unfortunately, such a trivialization is not possible in general. Indeed, the problem already appears 
in the case of linear actions of finite groups on the closed unit ball in $\RR^n$. We shall present a 
concrete counter-example in the following section.
 \end{remark}

\section{\(\K\)-theory of proper actions}\label{sec-K-theory}


In this chapter we will consider the problem of calculating
equivariant \(\K\)-theory for proper actions. 

\begin{definition}
Let \(X\) be a proper \(G\)-space where \(G\) is a locally compact
group. The \emph{\(G\)-equivariant \(\K\)-theory of \(X\)}, denoted \(\K^*_G(X)\), is the 
\(\K\)-theory of the crossed product C*-algebra \(C_0(X)\rtimes G\). 
\end{definition}

There are very few \emph{general} results about equivariant 
\(\K\)-theory.  The ones we discuss 
below all treat simplications of the problem. There are 
several kinds of simplications possible. One involves 
ignoring torsion in \(\K^*_G(X)\). Thus, one can aim for a 
computation of equivariant \(\K\)-theory with 
rational coefficients \(\K^*_G(X)\otimes_\Z \Q\). For 
technical purposes it is often convenient to tensor with 
\(\C\) instead, this gives equivalent results. 
Another 
simplication is to consider only compact groups.  
A further simplification is to restrict to finite groups. 

We 
start with the question of whether or not 
equivariant \(\K\)-theory can be described, 
as with non-equivariant \(\K\)-theory, in terms of vector bundles. 

\medskip

Let \(G\) be possibly non-compact, locally compact group, acting 
properly on \(X\) \emph{with compact space \(G\backslash X\) of 
orbits.}  
Let \(E\) be a \(G\)-equivariant 
vector bundle on \(X\). 
 Because the action is proper, there is a 
 \(G\)-invariant Hermitian structure on 
 \(E\). We can equip the space 
\(\Gamma_c(E)\) of continuous sections of \(E\) with compact supports
with the 
right \(C_c(G\times X)\)-module structure via 
$$ (s\cdot a)(x)=\int_G g\cdot s(g^{-1}x) a(g^{-1},g^{-1}x)\sqrt{\Delta(g^{-1})}\, dg,\quad s\in \Gamma_c(E), a\in C_c(G\times X)$$
where we regard $C_c(G\times X)$ as a dense subalgebra of $C_0(X)\rtimes G$.
One can also define a right \(C_0(X)\rtimes G\)-valued inner 
product on \(\Gamma_0(E)\), by the formula 
\begin{equation}
\label{eq:inner_product}
 \langle s_1, s_2\rangle (x,g) \defeq 
\sqrt{\Delta(g^{-1})}\langle s_1(x), g\cdot s_2(g^{-1}x)\rangle_E.
\end{equation}
So the completion \(\widetilde\Gamma(E):=\overline{\Gamma_c(E)}\) with respect to this inner product
becomes a 
right \(C_0(X)\rtimes G\)-Hilbert module. If 
\(G\backslash X\) 
is compact, this Hilbert module can be checked to be 
finitely generated.
The co-compactness assumption 
implies that the identity operator is 
a compact 
Hilbert module operator. See \emph{e.g.} \cite{Em-K} for 
a proof, and the following example 
for the case where $E=1_X$ is the trivial line bundle.

\begin{example}
\label{ex:line_bundle}
Let \(G\) be locally compact, and act properly and co-compactly on 
\(X\). If \(E\) is the trivial line bundle \(1_X\) over \(X\), with the 
trivial action of \(G\) on the fibres then 
the finitely generated projective Hilbert \(C_0(X)\rtimes G\)-module 
\(\widetilde{\Gamma}(1_X)\) as above is isomorphic to the 
range of any of the idempotents \(p_X\in C_0(X)\rtimes G\), 
constructed by a cut-off function \(c\). Recall that, since $G\backslash X$ is compact,
 $c:X\to[0,1]$ is a compactly supported 
continuous function such that $\int_G c(g^{-1}x)^2\,dg=1$ for all $x\in X$. 
Then $p_X:=\lk c,c\rk$ is a projection in $C_0(X)\rtimes G$ (compare with the 
 discussion around \eqref{eq-projection}) and we get
\[ \widetilde{\Gamma}(1_X) \cong p_X\cdot \bigl( C_0(X)\rtimes G\bigr)\]
as right \(C_0(X)\rtimes G\)-modules. 
For the proof, regard \(c\) as an element of 
\(\widetilde{\Gamma}(1_X)\). Define 
\[ Q\colon\widetilde{\Gamma}(1_X)\to C_0(X)\rtimes G, \; 
Q(s) \defeq \langle c, s\rangle,\; \; \; R\colon C_0(X)\rtimes G \to \widetilde{\Gamma}(1_X), \; 
R(a) \defeq ca,\]
using the inner product in \eqref{eq:inner_product}. 
Then for all \(s\in\widetilde{\Gamma}(1_X), \; a\in C_0(X)\rtimes G\), 
\[ RQ (s) =c\cdot \lk c,s\rk = s,\; \textup{and}\; \; QR (a) = \langle c, c\rangle a = p_X a.\]
In particular this shows that \(\widetilde{\Gamma}(1_X)\) is a 
rank-one module.  It also gives another proof that 
the \(\K\)-theory classes \([p_X]\in \K^0_G(X) \) are independent 
of the choice of cut-off function. Note also that the natural 
representation of \(C(G\backslash X)\to \L_{C_0(X)\rtimes G}
\bigl(\widetilde{\Gamma}(1_X)\bigr) \) by letting \(G\)-invariant functions act as multiplication operators on 
\(\widetilde{\Gamma}(1_X)\), 
gives an isomorphism 
\[C(G\backslash X) \cong \cK\bigl(\widetilde{\Gamma}(1_X)\bigr).\]
Consequently, we get an isomorphism  
\[ C(G\backslash X)\cong \cK (p_X\cdot C_0(X)\rtimes G)\\  \cong 
p_X\cdot C_0(X)\rtimes G \cdot p_X\]
(see the discussion around \eqref{eq-ideal}.)

\end{example}



Let \(\VK^0_G(X)\) be the Grothendieck group of the 
monoid of \(G\)-equivariant complex vector bundles 
on \(X\). Let 
\(\VK^1_G(X)\) be defined to be the kernel of the 
map \(\VK^0_G(X\times S^1) \to \VK^0_G(X)\) by 
restriction of \(G\)-equivariant 
vector bundles to \(X\times \{1\}\), where we let $G$ act trivially on $S^1$. Then 
application of the above procedure to cycles yields a map 

\begin{equation}
\label{eq:Green_Julg}
\VK^*_G(X) \to \K_*(C_0(X)\rtimes G).
\end{equation}

Is \eqref{eq:Green_Julg} an isomorphism in general? 
Corollary 5.3 of \cite{Em-K} shows that for any 
proper action of a locally compact group \(G\), the  
monoid of isomorphism classes of \(G\)-equivariant 
vector bundles on \(X\) is isomorphic to the monoid
of projections in \(\big(C_0(X)\rtimes G\big)\otimes \cK\). 
However, 
the \(\K\)-theory of \(C_0(X)\rtimes G\) 
is defined in terms of the monoid of projections 
in \(\bigl( C_0(X)\rtimes G \bigr)^+\otimes \cK\), 
where the \(+\) denotes unitization, 
so that \emph{a priori} there 
could be classes which do not come from equivariant 
vector bundles. In fact, the question is somewhat 
delicate. See the analysis 
\cite{Em-K}, following work of L\"uck and 
Oliver in \cite{Lueck-Oliver}, who are partially 
responsible for the following result. 

\begin{theorem}
\label{thm:good_groups}
If \(G\) is discrete, compact or an almost-connected
group, and if \(X\) is \(G\)-compact, 
then the map \eqref{eq:Green_Julg} is an isomorphism. 
\end{theorem}


For groups not satisfying one of the conditions of Theorem 
\ref{thm:good_groups}, equivariant vector bundles need not 
generate all of equivariant 
\(\K\)-theory.  
 
A counter-example to isomorphism of 
\eqref{eq:Green_Julg} can be built using
the following ideas, essentially due to Juliane 
Sauer. Let \(G\) be the semi-direct product
\(\TT^2\rtimes_A \Z\) where \(\Z\) acts by 
the automorphism of the compact group 
\(\TT^2\) induced by a matrix 
\(A \in \mathrm{GL}(2, \Z)\). Then \(G\) acts properly
on \(\R\) with \(\Z\) acting by translations and 
\(\TT^2\) acting trivially. The restriction of any
 \(G\)-equivariant 
vector bundle on \(\R\) to a point of \(\R\) determines a 
finite-dimensional representation of the compact 
subgroup \(\TT^2\subset G\). It can be shown that 
if \(A\) is a hyperbolic matrix, then 
only multiples of the trivial representation can be 
obtained in this way. This shows that there 
are rather few \(G\)-equivariant vector bundles on 
\(\R\). Of course \(G\) is neither compact, nor a Lie group. 

Equivariant Kasparov theory offers another 
point of view on equivariant \(\K\)-theory for 
\(G\)-compact spaces. The 
\emph{representable} \(G\)-equivariant 
\(\K\)-theory of a proper \(G\)-space \(X\) --  
denoted \(\textup{RK}^*_G(X) \) -- is defined 
in terms of \(G\)-equivariant maps from \(X\) to  
suitable spaces of Fredholm operators on \(G\)-Hilbert 
spaces; such maps may always be normalized so
that the Hilbert space may be taken fixed, so the 
representable \(\K\)-theory has cycles maps 
\(X\to \textup{Fred}(H_G)\). An equivariant 
version of the Kuiper theorem implies such maps 
may be taken norm continuous. For \(G\)-compact 
spaces, it can be shown that 
\(\textup{RK}^*_G(X) \cong \K^*_G(X)\) (see \cite[Theorem 3.8]{Em-K}). 
This gives a reasonably 
satisfactory homotopy-theoretic description of 
equivariant \(\K\)-theory for some purposes, but 
it is not very concrete. 

For general proper actions of locally compact 
groups \(G\) on spaces \(X\), 
\(\K^*_G(X)\) is a \(G\)-compactly supported theory and \(\textup{RK}^*_G(X)\) 
is not, so the theories definitely differ in this case. 

\begin{remark}
Unlike non-equivariant \(\K\)-theory, which is rationally isomorphic to 
cohomology, 
 equivariant \(\K\)-theory is 
definitely different from equivariant \emph{cohomology} -- even 
for finite groups, and even 
after tensoring both with the complex numbers. 
Equivariant cohomology for a compact group action
is defined
 to be the ordinary cohomology of \(X\times_G EG\), 
where \(EG\) is the classifying space for \emph{free} 
actions of \(G\).  It is not hard to check that 
equivariant cohomology with \emph{complex} (or rational)
 coefficients for \emph{finite} 
groups \(G\) is the same as the 
\(G\)-invariant part of ordinary cohomology, and hence (by 
Proposition \ref{prop:invariants} below) it agrees with just the ordinary cohomology of the 
quotient space, whereas equivariant \(\K\)-theory does not behave like this, 
even for finite groups, for already integrally
\(\K^*_G(\pt) \cong \Rep (G)\) is free abelian wiht one 
generator for 
each irreducible representation of \(G\).
What is true is that 
equivariant cohomology with complex coefficients and \emph{finite} 
groups, is isomorphic to the 
\emph{localization} of equivariant \(\K\)-theory with complex 
coefficients at 
the identity conjugacy class of the group, \emph{c.f.} 
\cite{Baum-Connes:Finite}). Equivariant \(\K\)-theory 
with complex coefficients may thus be viewed as 
`de-localized' equivariant cohomology -- the point of 
view taken in the article \cite{Baum-Connes:Finite} of 
Baum and Connes, and explained below in 
\S \ref{subsec:bc}.

\end{remark}

We start by justifying one of the statements made in the 
above Remark. 

\begin{proposition}
\label{prop:invariants}
For any \(G\)-space \(X\), 
\(G\) finite, 
\begin{equation}
\label{eq:invariants}
\K^*(X)^G\otimes_\Z \Z[\frac{1}{\abs{G}}]
\cong \K^*(G\backslash X)\otimes_\Z \Z[\frac{1}{\abs{G}}]
\end{equation}
as \(\Z[\frac{1}{\abs{G}}]\)-modules. 
The isomorphism is induced by the pull-back map 
\(\pi^*\colon \K^*(G \backslash X) 
\to \K^*(X)\).  
In particular, \(X\mapsto \K^*(X)^G\) and \(X \mapsto \K^*(G\backslash X)\) 
agree for finite groups, after tensoring by \(\C\). 

\end{proposition}


Proposition \ref{prop:invariants} does not hold for compact 
groups in general. It is easy to find a counter-example: take
the \(\TT\)-action on the \(2\)-sphere which rotates around the 
\(z\)-axis. The quotient space is the closed interval \([0,1]\), 
so \(\K^*(\TT\backslash S^2) \cong \Z\). But since \(\TT\) is connected, 
it acts trivially on \(\K\)-theory of \(S^2\), so \(\K^*(S^2)^\TT\cong 
\K^*(S^2)\cong \Z\oplus \Z\). These are obviously different even after 
tensoring with \(\C\). 

Nor does the Proposition hold for infinite discrete groups, for a 
different reason. If \(G\) is infinite and 
\(X = G\) with \(G\) acting by translation, then 
\(\K^*(G\backslash X) \cong \K_*(\C)\cong \Z\) but 
since there are no nonzero finitely supported 
\(G\)-invariant 
maps \(G \to \Z\), and since \(\K\)-theory is compactly 
supported, \(\K^*(X)^G = 0\).

\begin{proof}[Proof of Proposition \ref{prop:invariants}]
{Closed \(G\)-invariant subspaces of \(X\) are in 1-1
correspondence with closed subspaces of \(G \backslash X\). 
If \(Z\subset X\) is a closed invariant subspace, then 
the exact sequence 
\[ 0 \rightarrow C_0(X\smallsetminus Z) \to C_0(X) \to C_0(Z)\to 0\]
of \(G\)-C*-algebras induces a long exact sequence 
of \(G\)-equivariant \(\K\)-theory groups
\begin{equation}\label{eq-quotient-sequence}
 \cdots \to \K^* (G \backslash Z) \to \K^*(G \backslash X) 
\to \K^*(G \backslash Z) \xrightarrow{\partial} \K^{*+1}(G \backslash (X\smallsetminus Z))\to \cdots
\end{equation}
and also an exact sequence 
\[ \cdots \to \K^*(X\smallsetminus Z) \to \K^*(X) 
\to \K^*( Z) \xrightarrow{\partial} \K^{*+1}( X\smallsetminus Z)\to \cdots\]
of non-equivariant groups, each of which carries an action of
\(G\), which therefore induces a  long exact sequence
\begin{equation}\label{eq-invariant-sequence} \cdots \K^*(X)^G \otimes_\Z \Z[\frac{1}{\abs{G}}]
\to \K^*(Z)^G\otimes_\Z \Z[\frac{1}{\abs{G}}] \xrightarrow{\partial} \K^{*+1}(X\smallsetminus Z)^G \otimes_\Z \Z[\frac{1}{\abs{G}}]\to \cdots
\end{equation}
of the $G$-invariant elements.
Note that we have to tensor with $\Z[\frac{1}{\abs{G}}]$ to guarantee that  this sequence is exact.
The map \(\pi^*\) of the theorem induces a map between the sequence (\ref{eq-quotient-sequence}),
tensored by  $\Z[\frac{1}{\abs{G}}]$, and sequence (\ref{eq-invariant-sequence}), so it
follows from the Five-lemma that $\pi^*$ is an isomorphism for $X$, if this is true
for $Z$ and $X\smallsetminus Z$.}

We now proceed by induction on $|G|$. If $Z\subseteq X$ is the closed set of $G$-fixed points,
we have $\K^*(G\backslash X)=\K^*(X)=\K^*(X)^G$, so by the above discussion 
we may assume that all stabilizers for the action of $G$ on $X$ are proper subgroups of $G$.
We then find a cover $\mathcal U$ of $X$ by invariant open sets $U$ each 
isomorphic to some induced $G$-space $G\times_HY$ with $|H|<|G|$.

For each such set $U\cong G\times_HY$ we have \(\K^*(G\backslash U) \cong 
\K^*(H\backslash Y)\).  On the other hand \(G\times_H Y\) fibres 
by the first coordinate projection over the finite \(G\)-set 
\(G/H\); the fibre over \(gH\) is \(Y\). 
Hence the ordinary \(\K\)-theory of \(U\) decomposes as a direct sum 
\[\K^*(G\times_H Y) \cong \oplus_{gH\in G/H} \K^*(Y).\]
The group \(G\) permutes the summands by left translation and 
the invariant part $\K^*(G\times_H Y)^G$ is isomorphic to \(\K^*(Y)^H\).
By induction, the formula of the lemma is true for $H$, and 
and since $\abs{H}$ devides $\abs{G}$ we get
 $$\K^*(U)^G\otimes_\ZZ\Z[\frac{1}{\abs{G}}]\cong \K^*(G\backslash U)\otimes_\ZZ\Z[\frac{1}{\abs{G}}].$$

Now suppose a \(G\)-space \(X\) is covered by \(n\) induced open sets \(U_1, \ldots , U_n\). 
Put \(Z = X\smallsetminus U_n\). Then \( Z\) is covered by the \(n-1\) 
open sets \(U_1, \ldots ,  U_{n-1}\). 
By Meyer-Vietoris and the above discussion for induced sets, the lemma holds for $X$ if
it holds for $Z$. Thus by induction on $n$ we may therefore assume  the lemma to be proved for all $G$-sets which are 
covered by finitely many properly  induced open sets. 
The result then follows from the fact that the theories $\K^*(X)^G$ and $\K^*(G\backslash X)$ 
both respect inductive limits.
\end{proof}

We are going to start our investigation of equivariant \(\K\)-theory 
with a summary of what is known about finite group actions. 
We will also neglect torsion for the moment -- in fact, it will be 
helpful to tensor all equivariant \(\K\)-theory groups in the following by \(\C\), 
since this has the consequence that 
if \(G\) is finite, then 
\(\Rep(G)\otimes_\Z \C\), as a ring, 
is simply the ring of complex valued functions on the 
set of conjugacy classes in \(G\) (this would not be true even 
if we tensored by \(\Q\).) Now 
\(\K^*_G(X)\) is always a module over \(\Rep (G)\). 
Tensoring everything by \(\C\) then gives 
 \(\K^*_G(X)\otimes_\Z \C\) the structure of a module over 
\(\Rep (G)\otimes_\Z \C\), and any module over the ring 
of complex-valued functions on a finite set of points 
decomposes as a direct sum over the spectrum of the 
ring (which in this case is the set of conjugacy classes in \(G\).) 
This provides some additional 
algebraic structure 
which proves to be very useful. 

In what follws, we write \(\K^*_G(X)_\CC:=\K^*_G(X)\otimes_\ZZ\CC\) for 
 the usual integral \(G\)-equivariant 
\(\K\)-theory of \(X\), tensored by the complex 
numbers and we write $\Rep(G)_\CC$ for \(\Rep(G)\otimes_\Z \C\).

We start by simplifying even further, and discuss the
\emph{ difference} of the ranks of \(\K^0_G(X)_\CC\) and 
\(\K^1_G(X)_\CC\). This integer is called the 
\emph{equivariant Euler characteristic} of \(X\). 
Denote the equivariant Euler characteristic by 
\(\Eul (G\ltimes X)\). 
Although a crude invariant, it is at least 
easily geometrically 
computable, by a version of the Lefschetz fixed-point theorem
(proved by Atiyah, Theorem \ref{thm:AS_Euler} below.)

Clearly the equivariant Euler characteristic of a 
free action is the ordinary Euler characteristic of the 
quotient space. On the other hand, the Euler characteristic 
is multiplicative under coverings, so for a free action of a 
finite group 
\[\Eul(G\ltimes X) =  \Eul (G\backslash X) = \frac{1}{\abs{G}}\Eul (X).\]
The following lemma describes \(\chi (G\backslash X)\)
geometrically, even in the presence of isotropy, with the additional 
hypothesis of a smooth action on a smooth manifold. 

\begin{lemma}
\label{lem:Euler_quotient}
Let \(X\) be a smooth compact  
manifold and \(G\) a finite 
group acting smoothly on \(X\). Then 
\begin{equation}
\label{eq:Euler_quotient}
\Eul (G\backslash X) 
= \frac{1}{\abs{G}}\sum_{g \in G} \Eul(X^g),
\end{equation}
where 
\(X^g\) is the fixed-point submanifold of \(g\). 

\end{lemma}


\begin{proof}
Using the 
standard formula for the dimension of the 
 space of \(G\)-invariants in a representation and the fact that 
\(\K^*(G\backslash X)\otimes_\Z \C\) is the same as the 
\(G\)-invariant part of \(\K^*(X)\otimes_\Z \C\), 
 \[ \Eul (G\backslash X) = \frac{1}{\abs{G}}\sum_{g\in G} \chi^{\textup{Eul}} (g)\]
 where \(\chi^{\textup{Eul}}\) is the (virtual) character of the \(\Z/2\)-graded 
 representation of 
 \(G\) on \(\K^*(X)_\CC \; (\cong H^*(X, \C))\), that is, the difference of the characters of \(G\) acting
 on even and odd \(\K\)-theory tensored with \(\C\), or cohomology with 
 coefficients in \(\C\). 
  
 Now by the Lefschetz fixed-point formula, 
 \(\chi^{\textup{Eul}}(g) \defeq \textup{trace}_s (g) = \Eul (X^g)\), which 
 proves the result.  
\end{proof}

In the last proof, we used the very strong version of the Lefschetz 
fixed-point formula applicable to an \emph{isometry} of a compact 
Riemannian manifold. This classical fact seems quite well-known 
to topologists, but for lack of a reference, we cite 
the second author's article \cite{Emerson-Meyer:Equi_Lefschetz}
for this (see Theorem 10. For the connection between the 
Lefschetz map and traces, see \cite{Emerson:Localization}.)  
It also follows immediately from the Atiyah-Segal-Singer 
index theorem (see Theorem 2.12 of \cite{Atiyah-Segal:Index}.)

The following gives a number of examples of interesting 
finite group actions on surfaces.
  
\begin{example}
\label{ex:connor_floyd}
This example is due P.E. Conner and 
E.E. Floyd in \cite{CF}. Let \(p\) and \(q\) be two odd primes,  
let \(n = pq\) and \(\lambda\) be a primitive \(n\)th root of 
unity. Let \(S\subset \C P^2\) be the zero locus of 
the homogeneous polynomial 
\(f(z_1, z_2, z_3) = z_1^n + z_2^n + z_3^n\). Then \(S\) is a 
smooth complex submanifold of \(\C P^2\) of complex 
dimension \(1\), \emph{i.e.} is a curve. Let 
\(T([z_1, z_2, z_3]) \defeq [z_1, \lambda^pz_2, \lambda^qz_3]\). 
Then \(T\) has order \(n\). Let \(\Gamma= \langle T\rangle\). 
Let \[\tau\colon S \to S,\; \tau ([z_1, z_2,z_3]) \defeq
[\lambda z_1, z_2, z_3].\] It commutes with \(\Gamma\), has 
order \(n\). Thus we get an action of \(\Z/n\times \Z/n\) on 
\(S\). This is the restriction of an action on \(\mathbb{CP}^2\) of course. 
The fixed-point set of \(\tau\) in $\mathbb{CP}^2$ is points with first homogeneous 
coordinate zero, so it is a copy of \(\mathbb{CP}^1\). The quotient \(X\defeq \Gamma\backslash S\) is a closed
Riemann surface. Its genus is deduced from its Euler characteristic, 
which by the formula above is 
given by \( \frac{1}{n} \sum_{0\le j <n} \Eul(X^{T^j}).\)
Each \(T^j\) has a finite number of fixed-points, and, counting them,
one computes that \(\chi (X) = 2- (p-1)(q-1)\) and 
hence the genus of \(X \defeq G\backslash S\) is \( \frac{(p-1)(q-1)}{2}\). 

The map \[\tau\colon S \to S,\; \tau ([z_1, z_2,z_3]) \defeq
[\lambda z_1, z_2, z_3]\] descends to a self-map \(\dot{\tau}\colon X\to X\). 
We obtain an
action of the group \(G \defeq \Z/n\) on a surface, 
such that exactly one point has non-trivial isotropy. 
\end{example}
 
 \begin{remark}
There are simpler examples of a finite group action on a 
compact space with exactly one fixed point. Take a rotation 
of the Euclidean plane \(V\), extend this uniquely to an orientation-preserving
isometry of \(V\oplus \R\), and consider the induced map on the 
projectivisation \(\mathbb{P}(V\oplus \R) \cong \mathbb{RP}^2\).  
This fixes exactly one point (the point at infinity.) 
\end{remark}

 By contrast, the \emph{equivariant} Euler characteristic 
 is more complicated. 
 Atiyah and Segal proved the following result about it in \cite{Atiyah-Segal:Euler}. 
 The proof is more or less immediate from the Lefschetz 
 fixed-point theorem and a 
  theorem of Baum and Connes in \cite{Baum-Connes:Finite}
  discussed in \S \ref{subsec:bc}.

\begin{theorem}[{see \cite{Atiyah-Segal:Euler}}]
\label{thm:AS_Euler} Let \(G\) be a finite 
group acting smoothly on a smooth compact manifold 
\(X\). Then 
\[ \Eul (G\ltimes X) = \frac{1}{\abs{G}}\sum_{ g_1, g_2} \; \Eul (X^{g_1, g_2}),\]
where the sum is over all \emph{commuting pairs}
of group elements, and where \(X^{g_1, g_2}\) denotes the joint fixed-point set of 
\(g_1, g_2\). 
\end{theorem}

This local expression for the equivariant, or 
\emph{orbifold} Euler characteristic, had already appeared in the 
physics literature in connection with string theory before Atiyah and 
Segal interpreted it in terms of equivariant 
\(\K\)-theory in \cite{Atiyah-Segal:Euler}.

\subsection{The theorem of Baum and Connes}
\label{subsec:bc}
\emph{In this subsection, \(\K^*_G(X)_\CC\) continues to denote 
\(\K\)-theory with complex coefficients.} If \(G\) is a group, 
\([G]\) will denote the set of conjugacy classes in \(G\). 

Let \(G\) be finite, \(X\) a \(G\)-space. 
Then, as mentioned above, the equivariant \(\K\)-theory \(\K^*_G(X)\) is a module 
over the representation ring \(\Rep (G)\), and similarly for the complexifications. 
This fact is the main ingredient to the arguments which follow. 

The complex representation 
ring \(\Rep (G)_\CC\) is naturally isomorphic to the ring of 
complex class functions on \(G\), \emph{i.e.} the ring of functions 
on the finite set \([G]\) of conjugacy classes in \(G\). 
The isomorphism is given by sending an irreducible representation
$\tau\in \widehat{G}$ to its character $\chi_\tau=\trace\tau$. 

Any module \(M\) over a ring of this form, \emph{i.e.}
of the form \(\C^{[G]}\) with pointwise multiplication and 
addition, decomposes into a 
direct sum of modules \(M_{[g]}\), one component 
for each summand, \emph{i.e.} conjugacy class,
 \([g] \in [G]\). The components are the 
\emph{localizations} of \(M\) at these conjugacy classes: 
 \[M_{[g]} = M\otimes_{\Rep (G)}\C\]
where \(\C\) is understood as a left \(\Rep (G)\)-module by 
evaluation of characters at \([g]\).   

\begin{remark}
\label{rem:trivial_spaces_localized}
If \(G\) is a finite group and \(X\) is a trivial 
\(G\)-space, then it follows from 
\cite[Proposition 2.2]{Segal} that
$\K_G^*(X)\cong \K^*(X)\otimes\Rep(G)$, where the 
isomorphism sends the class $[V]$ of a $G$-equivariant 
 vector bundle $V$ in $\K_G^0(X)$ (assuming w.l.o.g. that $X$ is compact) to the 
sum $\sum_{\tau\in \widehat{G}} [\Hom^G(V_\tau, V)\otimes V_\tau]\in \K_0(X)\otimes\Rep(G)$,
where $V_\tau$ denotes the representation space of $\tau$ and $\Rep(G)$ is regarded as the 
free abelian group with generators the irreducible representations of $G$
(recall that $\dim(\Hom(V_\tau, V)_x)$ gives the multiplicity of $\tau$ in the fiber $V_x$ of $V$ at $x\in X$).
After complexification and using the above explained realization of $\Rep(G)_\CC$ as $\CC^{[G]}$
we obtain an isomorphism
$$\Phi^0: \K_G^0(X)_\CC\stackrel{\cong}{\longrightarrow}\K^0(X)_\CC\otimes\CC^{[G]}\cong\bigoplus_{[g]\in [G]} \K^0(X)_\CC$$
given on the class of an equivariant vector bundle $[V]\in \K_G^0(X)$ by
$$[V]\mapsto \oplus_{[g]\in [G]} \big(\sum_{\tau\in \widehat{G}} \chi_\tau(g)[ \Hom^G(V_\tau,V)]\big).$$
Replacing $X$ by $X\times S^1$, we obtain a similar description for an isomorphism
$$\Phi^1:\K_G^1(X)_\CC\stackrel{\cong}{\longrightarrow} \K^1(X)_\CC\otimes\Rep(G)_\CC.$$
In particular, we obtain surjective localization maps 
$\Phi^*_{[g]}:\K_G^*(X)_\CC\to \K^*(X)_\CC$
given by composing $\Phi^*$ with evaluation at $[g]$.
\end{remark}

The following lemma allows a more transparent description of the localization maps $\Phi^*_{[g]}$ 
of Remark \ref{rem:trivial_spaces_localized}.

\begin{lemma}\label{lem:Phi-trivial}
Suppose that the finite group $G$ acts trivially on the compact space $X$ and let $V$ be 
a $G$-equivariant complex vector bundle over $X$ giving a class $[V]\in \K_G^0(X)_\CC$.
Fix $g\in G$, let $d:=\ord(g)$ denote the order of $g$,
and let $C_d\subseteq\TT$ the set of $d$th roots of unity.
Then $V$ decomposes into a direct sum of vector bundles $V=\oplus_{\zeta\in C_d}V_\zeta$ 
such that each $V_\zeta$ is the eigenbundle for the eigenvalue $\zeta$ for the action of $g$ on $V$
and we get
$$\Phi^0_{[g]}([V])=\sum_{\zeta\in C_d} \zeta[V_\zeta]\in \K^0(X)_\CC.$$
Replacing $X$ by $X\times S^1$ (resp.~the one-point compactification $X^+$), a similar description is obtained for $\Phi^1_{[g]}$ (resp. for the case of locally compact spaces $X$).

\end{lemma}
\begin{proof}
We have to prove that $\sum_{\zeta\in C_d} \zeta[V_\zeta]=
\sum_{\tau\in \widehat{G}} \chi_\tau(g)[ \Hom^G(V_\tau,V)]$ in $\K^0(X)_\CC$.  
For this first observe that by the arguments of
\cite[Proposition 2.2]{Segal} we get 
$$V\cong \oplus_{\tau\in \widehat{G}}V_\tau\otimes \Hom^G(V_\tau,V)$$
as $G$-equivariant vector bundles, where $G$ acts trivially on the vector bundles 
$\Hom^G(V_\tau,V)$. For each $\zeta\in C_d$ let $\chi_\zeta:<g>\to \TT$ denote the 
corresponding character of the cyclic group $\lk g\rk$ generated by $g$, i.e., $\chi_\zeta(g^l)=\zeta^l$.
Let $V_\tau^\zeta$ denote the eigenspace in $V_\tau$ for the eigenvalue $\zeta$ of $\tau(g)$.
Its dimension  $d_\zeta^\tau$ gives the multiplicity of the character 
$\chi_\zeta$ in the restriction of $\tau$ to $\lk g\rk$. Computing $\chi_\tau(g)=\trace\tau(g)$ 
with respect to a corresponding basis, we  get $\chi_\tau(g)=\sum_{\zeta\in C_d} d_\zeta^\tau\zeta$.
We therefore get
\begin{equation}\label{eq;formula}
\Phi^0_{[g]}([V])=\sum_{\tau\in \widehat{G}} \chi_\tau(g)[ \Hom^G(V_\tau,V)]
=\sum_{\zeta\in C_d} \sum_{\tau\in \widehat{G}} d_\zeta^\tau\zeta [\Hom^G(V_\tau, V)].
\end{equation}
On the other hand, using the isomorphism $V\cong \oplus_{\tau\in \widehat{G}}\big(V_\tau\otimes \Hom^G(V,V_\tau)\big)$ we get 
$$V_\zeta\cong \oplus_{\tau\in \widehat{G}} \big(V_\tau^\zeta\otimes \Hom^G(V_\tau,V)\big)
\cong \oplus_{\tau\in \widehat{G}}\Hom^G(V_\tau,V)^{d_\zeta^\tau},$$ so that 
we get the equality 
$[V_\zeta]=\sum_{\tau\in\widehat{G}} d_\zeta^\tau [\Hom^G(V_\tau,V)]$
in $\K^0(X)_\CC$. Together with (\ref{eq;formula}) this gives
$$\sum_{\zeta\in C_d} \zeta[V_\zeta]=\sum_{\zeta\in C_d}\sum_{\tau\in\widehat{G}}  \zeta d_\zeta^\tau [\Hom^G(V_\tau,V)]=\Phi^0_{[g]}([V]).$$
\end{proof}

%



If a finite group acts on a compact space $X$, the module \(\K^*_G(X)_\CC\) 
decomposes as
\[ \K^*_G (X)_\CC \cong \bigoplus_{[g] \in [G]}
\K^*_G(X)_{\CC,[g]}\]
and to classify the module it suffices to 
analyse the contributions $\K^*_G(X)_{\CC,[g]}$ from each conjugacy 
class $[g]$. The above lemma does this job in case where $G$ acts trivially on $X$.
In what follows next, we want to extend this description to the case of arbitrary
compact $G$-spaces $X$. 

\begin{remark}\label{rem-forgetful}
Suppose that a finite group $G$ acts on the compact space $X$. Then we 
always get a projection from 
$\K_G^*(X)$ to the $G$-invariant part $\K^*(X)^G$ of $\K^*(X)$ which is given
by sending the class $[V]$ of a $G$-equivariant vector bundle to the class $[V]$
with forgotten $G$-action. To see that $V$ gives a $G$-invariant class in $\K^0(X)$ 
observe that the action of an element $g\in G$ on $V$ provides a continuous family of isomorphisms
$\alpha_x^g: V_x\to V_{gx}=g^*(V)_x$, hence an isomorphism 
between the bundles $V$ and $g^*V$. 
Note that the projection $\K_G^*(X)\to \K^*(X)^G$ becomes surjective after complexification, since 
for any $G$-invariant class $[V]$ the class $\frac{1}{|G|}\big[\oplus_{g\in G} g^*V\big]\in \K_G^0(X)_\CC$ 
maps to $[V]$ in the complexification of $\K^0(X)^G$. 
Moreover, by Proposition \ref{prop:invariants} we see that
the complexification of $\K^*(X)^G$ is isomorphic to  $\K^*(G\backslash X)_\CC$, so we obtain
a canonical surjective projection
\begin{equation}\label{eq:projection}
P_G^*: \K_G^*(X)_\CC\to \K^*(G\backslash X)_\CC
\end{equation}
which sends a class $[V]\in \K_G^0(X)$ to the class $\frac{1}{|G|}\sum_{g\in G} [g^*V]\in \K^*(G\backslash X)_\CC$. By passing to one-point compactifications, we obtain a similar projection  for locally compact
$G$-spaces.
\end{remark}

Suppose now that  $X$ is a compact $G$-space for the finite group $G$.
Let \(g\in G\). Any \(G\)-equivariant vector 
bundle on  \(X\) restricts to the set of $g$-fixed points
\(X^g=\{x\in X: gx=x\}\). It then decomposes into \(g\)-eigenbundles $V_\zeta$
as in Lemma \ref{lem:Phi-trivial}. Let $Z_g$ denote the centralizer of $g$ in $G$.
Each of the bundles $V_\zeta$ are $Z_g$-equivariant. Define a map 
\begin{equation}
\phi_{[g]}^0\colon \K^0_G(X)_\CC \longrightarrow \K^0_{Z_g}(X^g)_\CC, \; \; \;
\phi_{[g]}([V]) \defeq \sum_\zeta \zeta [V_\zeta].
\end{equation}
Composing this with the projections $P_{Z_g}^0: \K^0_{Z_g}(X^g)_\CC\to \K^0(Z_g\backslash X^g)_\CC$
and summing over \([g]\) yields a map 
\(\phi^0:\K^0_G(X)_\CC \to \oplus_{[g]\in [G]} \K^0(Z_g\backslash X^g)_\CC. \)

Replacing \(X\) by \(X\times S^1\) (or by the one-point compactification $X^+$) and repeating the 
construction gives a map 

\begin{equation}
\label{eq:BC_map}
\phi^*\colon \K^*_G(X)_\CC \to \oplus_{[g]\in [G]} \K^*(Z_g\backslash X^g)_\CC. 
\end{equation}
for any locally compact $G$-space $X$.

Consider now the group stabilizer bundle \(\SX=\{(x,g): x\in X, g\in G_x\}\). 
$G$ acts on $\SX$ via $h\dot(x,g)=(h\cdot x, hgh^{-1})$ and $\SX$ decomposes 
as a disjoint union
$$\SX\cong \bigsqcup_{[g]\in [G]}\bigsqcup_{[h]\in G/Z_g} h\cdot X^g\cong \bigsqcup_{[g]\in [G]} (G\times_{Z_g} X^g),$$  which implies  a decomposition
\begin{equation}\label{eq:SX}
G\backslash \SX\cong \bigsqcup_{[g]\in [G]} Z^g\backslash X^g.
\end{equation}
From this we obtain a decomposition
$$
\K^*(G\backslash\SX)_\CC\cong
 \oplus_{[g]\in [G]} \K^*(Z_g\backslash X^g)_\CC.
$$
Therefore we may regard the map 
\eqref{eq:BC_map} as having target 
\(\K^*(G\backslash\SX)_\CC\) and we obtain a well defined map
\begin{equation}
\label{eq:BC_main}
 \Phi_X\colon \K^*_G(X)_\CC \longrightarrow \K^*(G\backslash \SX)_\CC \cong 
\oplus_{[g]\in [G]} \K^*(Z^g\backslash X^g)_\CC.
\end{equation}

 Baum and Connes \cite{Baum-Connes:Finite} prove 
the following theorem. 

\begin{theorem}[Baum-Connes]
\label{thm:BC}
If \(G\) is a finite group, \(X\) a \(G\)-space, then \eqref{eq:BC_main} is 
an isomorphism 
\[ \K^*_G(X)_\CC \xrightarrow{\cong} 
\K^*(G\backslash\SX)_\CC\]
which sends the component $\K^*_G(X)_{\CC,[g]}$ of 
$\K_G^*(X)_\CC$  
to the component $\K^*(Z^g\backslash X^g)_\CC$ of $\K^*(G\backslash\SX)_\CC$.

\end{theorem}

The idea of the proof is given by observing that for any induced $G$-space
$X=G\times_HY$ the assertion of the theorem  is true if (and only if) it is true
for the $H$-space $Y$. Together with the fact that for trivial $G$-spaces 
the theorem has been checked already in Lemma \ref{lem:Phi-trivial},
this allows to combine induction over the order of $G$ with a Meyer-Vietoris 
argument to obtain the proof.

We first want to verify that the map \eqref{eq:BC_main} is a 
natural transformation between two (compactly supported) 
cohomology theories on 
\(G\)-spaces, namely \(X\mapsto \K^*_G(X)_\CC\) and 
\(X\mapsto \K^*(G\backslash \SX)_\CC\). Indeed, let 
\(Z \subseteq X\) be a closed, \(G\)-invariant subspace of 
a \(G\)-space \(X\). Then \(\SZ \subset \SX\) equivariantly 
so \(G\backslash \SZ\subset G\backslash \SX\) naturally. 
Hence a \(G\)-invariant closed subspace of \(X\) 
yields a \(6\)-term exact sequence 
\[ \cdots \stackrel{\partial}{\longrightarrow}
 \K^*\bigl(G\backslash \textup{Stab}(X\smallsetminus Z)\bigr) \longrightarrow
\K^*\bigl(G\backslash \SX\bigr) \longrightarrow
\K^*\bigl(G\backslash \SZ\bigr) \stackrel{\partial}{\longrightarrow}\cdots\]
We leave the proof of the following lemma to the reader. 

\begin{lemma}
\label{lem:excision}
If \(Z \subseteq X\) is a closed, \(G\)-invariant subspace, then the 
diagram 
{\footnotesize \begin{equation}
\label{eq:excision_compatible}
\xymatrix{\cdots\ar[r] & \K^*_G(X\smallsetminus Z)_\CC \ar[r]\ar[d]^{\Phi_{X\smallsetminus Z}}&
 \K^*_G(X)_\CC\ar[r]\ar[d]^{\Phi_X} & \K^*_G(Z)_\CC\ar[r]\ar[d]^{\Phi_Z} & \cdots\\
\cdots \ar[r] & 
\K^*\bigl( G\backslash \textup{Stab}(X\smallsetminus Z)\bigr)_\CC  \ar[r] & 
\K^*\bigl(G\backslash \SX\bigr)_\CC \ar[r] & \K^*\bigl( G\backslash \textup{Stab}(Z)\bigr)_\CC\ar[r] & \cdots }
\end{equation}}
commutes. 
\end{lemma}

\begin{remark}\label{rem-moduleH}
If $X=G\times_HY$ for some $H$-space $Y$, we obtain an isomorphism 
$$R: \K_G^*(X)\to \K_H^*(Y)$$ given by
composing the obvious restriction map $\K_G^*(X)\to \K_H^*(X)$ with the 
map  $\K_H^*(X)\to  \K_H^*(Y)$ coming from the inclusion of the open 
subset $Y$ in $X$ (see \cite[p. 132]{Segal}). Its inverse $\Ind: \K_H^*(Y)\to \K_H^*(X)$
is given on the level of vector bundles by $[V]\mapsto [G\times_H V]$.
 It is then easy to check that
the diagram
\begin{equation}\label{eq:repG-module}
\begin{CD}
\K_G^*(X)\otimes\Rep(G)   @> m_G >>    \K_G^*(X)\\
@V R\otimes\res VV   @VV R V\\
\K_H^*(Y)\otimes\Rep(H)   @> m_H >>   \K_H^*(Y)
\end{CD}
\end{equation}
commutes, 
where $m_G$ and $m_H$ denote the respective module actions of $\Rep(G)$ and $\Rep(H)$
and $\res:\Rep(G)\to \Rep(H)$ is the restriction map.
\end{remark}
 
 \begin{lemma}\label{lem:inducedspace}
Let \(G\) be a finite group, \(H \subseteq G\) a subgroup
 and \(X \defeq G\times_H Y\) be an induced $G$-space. 
 Then the result of Baum and Connes for \(H\) acting on 
 \(Y\) implies the result of Baum and Connes for 
 \(G\) acting on \(X\). 
 \end{lemma}

\begin{proof}
Recall that $G\times_HY$ is the orbit space of the space 
\(G\times Y\)  for the (right) action  
\( (\gamma,y)h = (\gamma h, h^{-1}y)\), of the subgroup 
\(H\) with $G$-action given by $g[\gamma, y]=[g\gamma, y]$.

Fix $g\in H$ and observe that \[X^g
= \{ [\gamma, y]\in G\times_H Y \mid\gamma^{-1}g\gamma \in H_y\}.\]

We therefore obtain a map
$$\Phi': X^g\to H\backslash \SY; [\gamma, y]\mapsto [\gamma^{-1}g\gamma, y].$$
which is well defined, because if we replace $(\gamma, y)$ by the pair $(\gamma h, h^{-1}y)$ 
for some $h\in H$, then $(\gamma^{-1}g\gamma, y)$ is replaced by 
$(h\gamma g\gamma^{-1}h, h^{-1}y)=(\gamma g\gamma^{-1},y)h$.
Moreover, if $z\in Z^g$, the centralizer of $g$ in $G$, then one checks 
that $\Phi'(z[\gamma,y])=\Phi'([\gamma, y])$, so that $\Phi'$ descends 
to a map
$$\Phi: Z^g\backslash X^g\to H\backslash\SY.$$

%

%

As discussed above (see (\ref{eq:SX})), 
we can fibre the quotient of the group stabilizer bundle 
\(H\backslash \SY\) over the set of conjugacy classes in 
\(H\), 
\[ H\backslash \SY = \bigsqcup_{[h]\in [H]} (Z^h\cap H)\, \backslash \, Y^h\]
where \(Z^h\) is the centralizer \emph{in \(G\)} 
 of a chosen representative 
\(h\) of \([h]\).

In this picture, the range of \(\Phi\) is exactly the 
union of the components of \(H\backslash \SY\) 
 \emph{which are labelled by \([h]\) with 
\(h\) conjugate in \(G\) to 
\(g\),} and $\Phi$ induces a homeomorphism between
$Z^g\backslash X^g$ and $\bigsqcup_{[h]\in [H], [h]\subseteq [g]} (Z^h\cap H)\, \backslash \, Y^h$.
This shows that \(\Phi\) induces an isomorphism 
\begin{equation}
\label{eq:bc_geometric_iso}
 \Phi^*\colon \bigoplus_{[h]\in [H], \, [h]\subset [g]} \, \K^*((Z^h\cap H)\, \backslash \, Y^h)
\stackrel{\cong}{\longrightarrow} \K^*(Z^g\backslash X^g).
\end{equation}
This can be viewed, in a rather trivial way, as a 
\(\Rep (G)\)-module isomorphism, 
with \(\Rep (G)\)-module structure given (on each side) by 
evaluation of characters at \([g]\). 

Now, if $g\in G$ such that the conjugacy class $[g]_G$ of $g$ in $G$ does not intersect with 
$H$, then it follows from the commutativity of diagram (\ref{eq:repG-module}) that 
$\K_G^*(X)_{\CC,[g]}=\{0\}$, since the restriction of the characteristic function of $[g]_G$ to $H$ vanishes.
On the other hand, if $g$ is not conjugate to any element in $H$, it follows that $X^g=\emptyset$,
so that we also have $\K^*(Z^g\backslash X^g)_\CC=\{0\}$.  

Using the above results together with the fact 
 that  the action of $\Rep(G)_\CC$ on $\K_H^*(Y)_\CC$ is given via 
the restriction $\res:\Rep(G)_\CC\to\Rep(H)_\CC$, we now obtain a 
diagram of $\Rep(G)_\CC$-module maps
$$
\begin{CD}
\K_H^*(Y)_\CC\quad @>\Phi_Y>> \oplus_{[g]\in [G]}\left(\oplus_{[h]\in [H], [h]\subseteq [g]} \K^*((Z^h\cap H)\backslash Y^h)_\CC\right)\\
@V \Ind  VV        @VV\Phi V\\
\K_G^*(X)_\CC\quad   @>>\Phi_X> \oplus_{[g]\in [G]} \K^*(Z^g\backslash X^g)_\CC
\end{CD}
$$
We leave it as an exercise to the reader to check that this diagram commutes (do this first 
in case where $Y$ is compact, and then obtain the general case by passing from $Y$ to $Y^+$). 
Assume that the theorem of Baum and Connes holds for $H$ acting on $Y$, i.e.,
$\Phi_Y$ is an isomorphism. 
Since $\Phi$ and $\Ind$ are also isomorphisms, the same must then be true for 
$\Phi_X$.

\end{proof}

We are now ready for the

\begin{proof}[Proof of Theorem \ref{thm:BC}] 
We induct on the order of the group \(G\). For trivial 
groups there is nothing to prove. Suppose the result is true for all 
groups \(G\) of cardinality \(\le n\).  Let \(X\) be a \(G\)-space, \(G\) finite, 
of cardinality \(\le n+1\). Let 
\(F\subset X\) be the stationary set. From Lemma \ref{lem:excision} the 
diagram 
{\small\begin{equation}
\xymatrix{\cdots\ar[r] & \K^*_G(X\smallsetminus F)_\CC \ar[r]\ar[d]^{\Phi_{X\smallsetminus Y}}& \K^*_G(X)_\CC\ar[r]\ar[d]^{\Phi_X} & \K^*_G(F)_\CC\ar[r]\ar[d]^{\Phi_F} & \cdots\\
\cdots \ar[r] & 
\K^*\bigl( \textup{Stab}(X\smallsetminus F)\bigr)_\CC  \ar[r] & 
\K^*\bigl(\SX\bigr)_\CC \ar[r] & \K^*\bigl( \textup{Stab}(F)\bigr)_\CC\ar[r] & \cdots }
\end{equation}}
commutes. The vertical map $\Phi_F$ is an isomorphism by Lemma \ref{lem:Phi-trivial}.
Thus, by the Five Lemma, it suffices to show that the 
first vertical map 
\(\Phi_{X\smallsetminus F} \colon \K^*_G(X\smallsetminus F)\longrightarrow 
\K^*\bigl( \textup{Stab}(X\smallsetminus F)\bigr)\)
is an isomorphism, too. 

So from now on we may assume without loss of generality, that $X$ does not contain any $G$-fixed points.
In this case we find a cover $\{U_i:i\in I\}$ of $G$-invariant open subsets $U_i$ of $X$ 
such that each of these sets is $G$-homeomorphic to $G\times_{H_i}Y_i$ for 
some proper subgroup $H_i$ of $G$. It follows then from the induction hypothesis together
with Lemma \ref{lem:inducedspace} that the theorem holds for 
each $U_i$. By Meyer-Vietoris, using Lemma \ref{lem:excision}, this implies the 
theorem for all unions $U_F:=\cup_{i\in F}U_i$,  with $F\subseteq I$ finite.
The result now follows from continuity of $\K$-theory with respect to inductive limits and 
the fact that
$C_0(X)\rtimes G=\lim_FC_0(U_F)\rtimes G$ and $C_0(G\backslash \SX)=\lim_FC_0(G\backslash \textup{Stab}(U_F))$.
\end{proof}

\begin{example}\label{ex-D4-3}
Consider the action of the dihedral group $G=D_4$ on $\TT^2$ as 
introduced in Example \ref{ex-D4}. Recall that it is generated by 
the matrices $R,S\in \GL(2,\ZZ)$ with $R=\left(\begin{smallmatrix} 0&-1\\1&0\end{smallmatrix}\right)$
and $S=\left(\begin{smallmatrix} 1&0\\0&-1\end{smallmatrix}\right)$ and the canonical action 
of $\GL(2,\ZZ)$ on $\TT^2=\RR^2/\ZZ^2$.
This group of order eight has five conjugacy classes given by
$$[E], [R^2], [R], [S], [SR].$$
The corresponding centralizers are
\begin{align*}
&Z^{E}=Z^{R^2}=G,\quad  Z^{R}=\lk R\rk, \quad Z^{S}=\{E,R^2, S, SR^2\},\\
&  Z^{RS}=\{E,R^2, RS, SR\},
\end{align*}
while the corresponding fixed-point sets are
\begin{align*}
&X^E=\TT^2,\quad X^{R^2}=\{\left(\begin{smallmatrix} 1\\1\end{smallmatrix}\right),
\left(\begin{smallmatrix} -1\\-1\end{smallmatrix}\right), \left(\begin{smallmatrix} -1\\1\end{smallmatrix}\right),
\left(\begin{smallmatrix} 1\\-1\end{smallmatrix}\right)\},
X^R=\{\left(\begin{smallmatrix} 1\\1\end{smallmatrix}\right),
\left(\begin{smallmatrix} -1\\-1\end{smallmatrix}\right)\},\\
&X^S=\{(z,1) : z\in \TT\}\cup \{(z,-1): z\in \TT\},\quad
X^{RS}=\{(z,z): z\in \TT\}.
\end{align*}
With $Z=\{(s,t): 0\leq s,t\leq\frac12, 0\leq t\leq s\}$ as in Example \ref{ex-D4} and $I=[0,1]$ we get
$$Z^E\backslash X^E=G\backslash \TT^2\cong Z,\quad Z^S\backslash X^S\cong I\sqcup I,
\quad Z^{RS}\backslash X^{RS}\cong I,$$
the set $Z^{R^2}\backslash X^{R^2}$ has three elements and $Z^R\backslash X^R=X^R$ has two elements.
Thus, as a consequence of the theorem of Baum and Connes we see that
$$\K_G^0(\TT^2)_\CC=\CC^9\quad\text{and}\quad \K_G^1(\TT^2)_\CC=\{0\}.$$
\end{example}
\medskip

%
 
 The formula of Atiyah and Segal 
 given in Theorem \ref{thm:AS_Euler} follows 
 from a little manipulation similar to that of the 
 proof of Lemma \ref{lem:Euler_quotient}, and Theorem 
 \ref{thm:BC}. This is not hard to check, and the
  original reference explains 
 this quite clearly, so we will omit the proof. 
 
 The result of Baum and Connes is of course 
 stronger; it gives the dimensions of \(\K^0_G(X)_c\) and 
 \(\K^1_G(X)_c\) separately.

\subsection{Computation of integral \(\K\)-theory}
We now return to integral equivariant \(\K\)-theory. 

In the previous section, we explained the result of Baum and 
Connes, which gave a formula for the \emph{ranks} of 
\(\K^i_G(X)\), \(i=0,1\). A finitely generated abelian group 
is determined by its free part and torsion part. In this section 
we will discuss the torsion part of \(\K^*_G(X)\). Computing 
this -- even for finite group actions -- seems to be a much 
more difficult problem than computing the free part. 

Since part of what we are going to describe is quite 
general and works for locally compact groups, we 
now let \(G\) be an arbitrary locally compact group and 
\(X\) a proper and \(G\)-compact \(G\)-space. 

Let $I_X$ be the ideal in $C_0(X)\rtimes G$ as in 
(\ref{eq-ideal}).
Let $Q_X:=\big(C_0(X)\rtimes G\big)/I_X$. Then the Morita equivalence 
$C_0(G\backslash X)\cong I_X$ and excision yields a 
six-term exact sequence
\begin{equation}\label{eq-K}
\begin{CD}
\K^0(G\backslash X) @>>> \K^0_G(X) @>>> \K_0(Q_X)\\
@A\partial AA   @.    @VV\partial V\\
\K_1(Q_X)  @<<<  \K^1_G(X)   @<<<  \K^1(G\backslash X).
\end{CD}
\end{equation}

In the case of group actions in which fixed-points are isolated 
(more precisely below), this exact sequence becomes quite tractable and 
we will study this situation in more detail below.
In effect, this is a situation in which the Fell topology on 
\(G\backslash \SXhat\) is very easy to understand. 

When fixed-point sets are not just zero-dimensional, the problem 
obviously becomes more complicated. We are not going to 
consider the general case here. However, we will analyse
Example \ref{ex-D4final} in some detail.

If the set of points in \(G\backslash X\) with non-trivial isotropy is a 
discrete subset, then by  Lemma \ref{lem-projection}, the 
C*-algebra $Q_X$ is isomorphic to a direct sum
of compact operators, with one summand contributed by each 
pair \( (Gx, \sigma)\), where \(Gx\) is an orbit 
 with nontrivial stabilizer subgroup $G_x$, and
 $\sigma\neq 1_{G_x}$ in $\widehat{G}_x$ is an irreducible 
 represention different from the trivial representation. 

More formally, if $\mathcal I$ denotes the
set of orbits with nontrivial stabilizers,

\begin{equation}\label{eq-Qp}
Q_X\cong \bigoplus_{Gx \in \mathcal I}\left(\bigoplus_{\sigma\in \widehat{G}_x\smallsetminus\{1_{G_x}\}}\mathcal K(\mathcal H_{U^\sigma})\right),
\end{equation}
where for each orbit $Gx \in \mathcal I$ we choose one representative $x$ of that orbit and
where $\mathcal H_{U^\sigma}$ is
as in  (\ref{eq-ind1}). If we write $\Rep^*(G_x)$ for the subgroup of $\Rep(G)$ generated by the 
non-trivial rreducible representations of $G_x$ it follows that
$\K_0(Q_X)\cong \oplus_{Gx\in \mathcal I} \Rep^*(G_x)$ is a free abelian group 
and 
$\K_1(Q_X)=\{0\}$, so that the  six-term sequence (\ref{eq-K})
becomes 
\begin{multline}\label{eq-KQp}
0\longrightarrow \K^0(G\backslash X)\longrightarrow 
\K^0_G(X)\\ \longrightarrow 
\bigoplus_{Gx \in \mathcal I}\Rep^*(G_x)  \stackrel{\partial}{\longrightarrow} 
 \K^1(G\backslash X)
 \longrightarrow \K^1_G(X) \longrightarrow 0.
 \end{multline}

In particular we see that \(\K^0(G\backslash X)\) injects into 
\(\K^*_G(X)\) and \(\K^1(G\backslash X)\) surjects to 
\(\K^1_G(X)\), in the case of isolated fixed-points. Note that,
in case of finite $G$,  
according to Theorem 
\ref{thm:BC} of 
Baum and Connes; \(\K^i(G\backslash X)\) 
injects in both after complexification: it is the contribution of the trivial 
conjugacy class \([1]\). Therefore, our statement 
adds to theirs that this is an injection in dimension zero
\emph{integrally}, and that 
\(\K^1(G\backslash X) \to \K^1_G(X)\) is a surjection, likewise, even 
integrally. 

As we will see, non-vanishing of \(\partial ([\sigma])\) for a generator in $\K_0(Q_X)$ 
corresponding to
some non-trivial representation \(\sigma\) of one of the 
isotropy groups \(G_x\), obstructs extending the 
corresponding induced \(G\)-equivariant vector bundle 
\([V_\sigma]\) on the orbit of \(x\) to a \(G\)-equivariant vector bundle on 
\(X\). The statement is that these obstructions are 
torsion: they vanish after multiplication by a suitable 
integer. Thus \(nV_\sigma\) can always be extended, for 
appropriate \(n\). 

We also note that the point of view suggested by 
\eqref{eq-KQp} yields a 
somewhat different formula for the Euler 
characteristic. For a finite group \(G\), let
 \(\widehat{G}^*\) denote 
\(\widehat{G}-\{1\}\) where \(1\) is the trivial 
representation. 

\begin{proposition}
Let \(G\) be a discrete group acting properly and co-compactly 
on \(X\) such that the 
set $D$ of points $x\in X$ with non-trivial isotropy is discrete.  Then 
\[ \Eul (G\ltimes X) = \Eul (G\backslash X) + \sum_{Gx\in G\backslash D} 
\abs{\widehat{G}_x^*}.\]
\end{proposition}

We point this out mainly to emphasise that we are counting 
in a `transverse' direction to Atiyah, Baum and Connes, where our 
theorems intersect, when \(G\) itself is a finite group. They 
fibre \(G\backslash \SX\) over the conjugacy classes, while we are
fibring it over \(G\backslash X\). To check that the two formulas for 
the Euler characteristic are the same, 
 it suffices to observe that both -- clearly -- agree with the 
Euler characteristic in ordinary \(\K\)-theory of the same Hausdorff space 
\(G\backslash \SX\). 

In any case, since both the Baum-Connes formula and ours 
contain the term \(\Eul (G\backslash X)\), we can remove 
it from each side. The identity
\begin{equation}
\label{eq:BC_vs_ours}
\sum_{[g]\in [G]\smallsetminus\{[1]\}} \textup{card}(Z^g\backslash X^g) 
= \sum_{Gx\in I} 
\abs{\widehat{G}_x^*}.
\end{equation}
is a good exercise to prove directly (the two sides of it 
are, roughly speaking, in a relation of duality, since if
\(X\) is a point, the left-hand-side computes the number of 
nontrivial conjugacy classes in \(G\) and the right-hand-side 
computes the number of nontrivial irreducible representations.)

To compute equivariant \(\K\)-theory in the case of isolated 
fixed points, it remains to solve the following 

\bigskip
\noindent
{\bf Problem:} Describe the boundary map 
\[\partial\colon \bigoplus_{Gx \in \mathcal I}\bigoplus_{\sigma\in \widehat{G}_x\smallsetminus\{1_{G_x}\}}\Z  \\ \longrightarrow
 \K^1(G\backslash X)\] in 
\eqref{eq-KQp}.

\begin{remark}\label{rem-finite} 
In the case of a \emph{free} action there is of course nothing to 
do. In this case the quotient term vanishes because \(Q_X\) is itself 
the zero C*-algebra. 

If (\emph{c.f.} Example \ref{ex:connor_floyd}) there is a \emph{single} point \(x_0\) in \(G\backslash X\) with non-trivial 
isotropy, and if this point is fixed by the entire group \(G\)
 (so that \(G\) must 
be compact) then \eqref{eq-KQp} becomes 
\begin{multline}
0\longrightarrow \K^0(G\backslash X)\longrightarrow 
\K^0_G(X)\longrightarrow \Rep^*(G)\stackrel{\partial}{\longrightarrow} 
 \K^1(G\backslash X)
 \\ \longrightarrow \K_G^1(X) \longrightarrow 0.
 \end{multline}

The quotient map 
\(\K^0_G(X) \to \Rep^*(G)\) in this sequence is induced by 
the \(G\)-map \(\pt \to X\) corresponding to the stationary point. 
{If $X$ is compact, this} map is split by the map \(X\to \pt\) and hence 
the boundary map \(\partial\) vanishes in this case. 
 Hence 
$$\K^0_G(X)\cong \K^0(G\backslash X)\oplus \Rep^*(G), \; \; 
\K^1_G(X) \cong \K^1(G\backslash X).$$
\end{remark} 

The general case is as follows.

\begin{theorem}\label{thm-Kproper}
Suppose that the locally compact group $G$ acts properly on $X$ such that the following are satisfied:
\begin{enumerate}
\item $(G,X)$ satisfies Palais's slice property (SP);
\item the orbits $Gx $ with nontrivial stabilizers $G_x$ are isolated in $G\backslash X$; and
\item all stabilizers are finite.
\end{enumerate}
Then the boundary map  $\partial: \K_0(Q_X)\to \K^1_G( X)$ of (\ref{eq-KQp})
is rationally trivial (i.e., the image of $\partial$ is a torsion subgroup of
 $\K^1(G\backslash X)$). 

Moreover, if an element $\nu\in \K_0(Q_X)$ is an element of the summand $\oplus_{\sigma\in \widehat{G}_x\setminus\{1_{G_x}\}}\K_0(\cK(\H_{U^\sigma}))$ in the direct-sum decomposition (\ref{eq-Qp}) of $\K_0(Q_X)$, then the 
order of $\partial (\nu)$ is a divisor of the order $|G_x|$. 
\end{theorem}

The following is a well-known fact about flat vector bundles. 

\begin{lemma}\label{lem-free}
Let $H$ be a finite group acting freely on the compact space $Y$. 
Then for each $\sigma\in \widehat{H}$, the difference of classes 
\[[Y\times V_\sigma ] - \dim(V_\sigma)[1_Y] \in \K^0_H(Y)\]
is torsion of order a divisor of \(\abs{H}\). Thus, 
$$0=|H|\cdot \big(\dim(V_\sigma)\cdot [Y\times \CC]- [Y\times V_\sigma]\big)\in \K^0_H(Y).$$

\end{lemma}
%

\begin{proof}
If $H$ acts freely on $Y$ we have $\K^0_H(Y)\cong \K^0(H\backslash Y)$ via 
sending the class of  an equivariant vector bundle $V$ over $Y$ to the class of the
bundle
$H\backslash V$ over $H\backslash Y$ (e.g. see \cite[Proposition 2.1]{Segal}). Consider the sequence 
$$ \K^0(H\backslash Y)\stackrel{q^*}{\to} \K^0(Y)\stackrel{q_*}{\to} \K^0(H\backslash Y)$$
in which the first map is given by the pull-back of vector bundles with respect to the quotient map
$q:Y\to H\backslash Y$ and the second map is the transfer map which sends a vector bundle 
$V$ over $Y$  to the bundle with fibre $\oplus_{g\in H} V_{gy}$ over 
$Hy\in H\backslash Y$. It is clear that the composition $q_*\circ q^*$ acts on 
$\K^0(H\backslash Y)$ as multiplication by $|H|$.  

If $\sigma\in \widehat{H}$, then the $H$-bundle $[Y\times V_\sigma]\in \K^0_H(Y)$ 
(with respect to the diagonal action) corresponds to the (flat) bundle
$[Y\times_HV_\sigma]\in \K^0(H\backslash Y)$. The
 pull-back $q^*([Y\times_HV_\sigma])\in \K^0(Y)$
is the class of the trivial bundle $[Y\times V_\sigma]=\dim(V_\sigma)[Y\times \CC]$ 
which is mapped to  $|H|\dim(V_\sigma)[(H\backslash Y)\times \CC]$ by the transfer map $q_* :\K^0(Y)\to \K^0(H\backslash Y)$.
Thus we see that 
$$|H|[Y\times_HV_\sigma]=q_*\circ q^*([Y\times_HV_\sigma])=|H|\dim(V_\sigma) [(H\backslash Y)\times \CC]
$$
and the result follows.
\end{proof}

In what follows next we want to consider the case of an induced space $X=G\times_HY$,
where $H$ is a finite subgroup of the locally compact group 
$G$ and $Y$ is a compact $H$-space with isolated
$H$-fixed point $y\in Y$ such that $H$ acts freely on $Y\smallsetminus \{y\}$.
Regarding $Y$ as a closed subspace of $X$, it is clear that $Gy$ is the only non-free orbit
in $X$, and its stabilizer is $G_y=H$, so that the $\K$-theory of  $Q_X=(C_0(X)\rtimes G)/I_X$ 
is the free abelian group $\Rep^*(H)$
generated by the representations $\sigma\in \widehat{H}\smallsetminus\{1_H\}$.

Note also that if $X$ is a $G$-space and $Z$ is a closed $G$-invariant subspace
of $X$, then the restriction map $\res_Z:C_0(X)\to C_0(Z)$ induces a quotient map 
$$\res_Z\rtimes G:C_0(X)\rtimes G\to C_0(Z)\rtimes G.$$

\begin{proposition}\label{prop-local}
 Let $X=G\times_HY$ and $Q_X=(C_0(X)\rtimes G)/I_X$ be as above. 
Let $\nu_\sigma\in \K_0(Q_X)$ be the
class corresponding to the given representation $\sigma\in \widehat{H}\smallsetminus \{1_H\}$.
Then there exists a class 
$\mu_\sigma\in \K^0_G(X)$
such that the following are true
\begin{enumerate}
\item If $q_X:C_0(X)\rtimes G\to Q_X$ denotes the quotient map, then $q_{X,*}(\mu_\sigma)=|H|\nu_\sigma$.
\item if $Z$ is any closed $G$-invariant subspace of $X$ which does not contain the 
orbit $Gy$, then $(\res_Z\rtimes G)_*(\mu_{\sigma})=0$ in $K_0(C_0(Z)\rtimes G)$.
\end{enumerate}
\end{proposition}
\begin{proof}
If we regard $Y$ as a closed subspace of $X=G\times_HY$ via $y\mapsto [e,y]$, we 
see that $Z\cong G\times_HY_Z$ with $Y_Z:=Y\cap Z$. It is not difficult to check that the canonical 
 Morita equivalence 
$C_0(G\times_HY)\rtimes G\sim_M C(Y)\rtimes H$, as explained before  Proposition \ref{prop-commute},
restricts to the canonical Morita equivalence
$C_0(G\times_HY_Z)\rtimes G\sim_M C(Y_Z)\rtimes H$ and it follows from Corollary \ref{cor-ideal}
and the Rieffel correspondence (see the discussion before Corollary \ref{cor-ideal}) that this 
factors through a Morita equivalence $Q_X\sim_M Q_Y$. Thus these Morita equivalences 
provide us with a commutative diagram
$$
\begin{CD}
\K_0(C(Y_Z)\rtimes H)  @<(\res_{Y_Z}\rtimes H)_*<<  \K_0(C(Y)\rtimes H)  @>q_{Y,*}>> \K_0(Q_Y)\\
@V\cong VV   @V\cong VV   @VV\cong V\\
\K_0(C_0(Z)\rtimes G)  @<<(\res_Z\rtimes G)_*<  \K_0(C(X)\rtimes G)  @>>q_{X,*}> \K_0(Q_X).
\end{CD}
$$
Thus we may assume that $G=H$ and $Z=Y_Z$. This also allows the use of vector bundles
for the description of $\K$-theory classes.
In this picture,  the quotient map
$q_{Y,*}: \K_0(C(Y)\rtimes H) \to \K_0(Q_Y)$  translates to the map $\tilde{q}:\K_H^0(Y)\to \K_0(Q_Y)$
which maps an $H$-bundle $V$ over $Y$ to the class 
$\sum_{\tau\neq 1_H} n_\tau\cdot \nu_\tau\in \K_0(Q_Y)$, where the sum is over the non-trivial irreducible representations of $H$ and $n_\tau$ denotes 
the multiplicity of $\tau$ in the representation of $H$ on the fibre $V_y$.
Define 
$$\mu_\sigma:=|H|\cdot \big( [Y\times V_\sigma]-\dim(V_\sigma)\cdot [Y\times \CC]\big)\in \K_H^0(Y)\cong \K^0_G(X).$$ 
Then  $\tilde{q}(\mu_\sigma)=|H|\cdot \nu_\sigma$ and 
$(\res_{Y_Z}\rtimes F)_*(\mu_\sigma)=|H|\cdot\big( [Y_Z\times V_\sigma]-\dim(V_\sigma)\cdot [Y_Z\times \CC]\big)
=0$ by Lemma \ref{lem-free} (because \(H\) acts freely on $Y\smallsetminus \{y\}$.)
\end{proof}

We are now ready for the

\begin{proof}[Proof of Theorem \ref{thm-Kproper}]
For the proof it suffices to show that the canonical generators of $\K_0(Q_X)$ are 
mapped to elements of finite order in $\K_1(C_0(G\backslash X))$ under the boundary map
$\partial: \K_0(Q_X)\to \K_1(C_0(G\backslash X))$. If the action of $G$ on $X$ is free,
then $Q_X=\{0\}$ and the result is trivial. So assume now that $Gy$ is any fixed orbit with
non-trivial stabilizer $G_y$ and let $\nu_\sigma$ denote the generator of $\K_0(Q_X)$ corresponding 
to the representation $\sigma\in \widehat{G_y}\smallsetminus\{1_{G_y}\}$. We claim that 
$|G_y|\partial(\nu_\sigma)=0$. This will also prove the last statement of the theorem.

Since, by assumption, there exists a neighborhood of $Gy$ in $G\backslash X$ such that 
all orbits in this neighborhood are free, and since the actions of $G$ on $X$ satisfies 
the slice property (SP), we may choose an open $G$-invariant 
neighborhood $U$ of $y$ with $G$-compact closure $W=\bar{U}$ with the following 
properties:
\begin{enumerate}
\item $W\cong G\times_{G_y}Y$ for some compact $G_y$-space $Y$.
\item $G$ acts freely on $W\smallsetminus \{Gy\}$.
\end{enumerate}
Let $Z=W\smallsetminus U$.
Then Proposition \ref{prop-local} implies that we can find a class $\mu_\sigma\in \K_0(C_0(W)\rtimes G)$ 
such that $q_{W,*}(\mu_\sigma)=|G_y|\cdot\nu_\sigma$ and $(\res_{Z}\rtimes G)_*(\mu_\sigma)=0$.
Applying the Meyer-Vietoris sequence in $\K$-theory (e.g. see \cite[Theorem 21.5.1]{Black}) to the
 pull-back diagram
$$
\begin{CD} C_0(X)\rtimes G  @>\res_{X\smallsetminus U}\rtimes G>> C_0(X\smallsetminus U)\rtimes G\\
@V\res_{W}\rtimes GVV      @VV\res_Z\rtimes G V\\
C_0(W)\rtimes G  @>>\res_Z\rtimes G  >  C_0(Z)\rtimes G
\end{CD}
$$ 
we see that we may glue the class $\mu_\sigma\in \K_0(C_0(W)\rtimes G)$  with the zero-class
in $\K_0(C_0(X\smallsetminus U)\rtimes G)$ to obtain a class $\mu\in \K_0(C_0(X)\rtimes G)$ 
such that $q_{X,*}(\mu)=|G_y|\cdot \nu_\sigma\in \K_0(Q_X)$. This implies that
$\partial(|G_y|\cdot \nu_\sigma)=0$ in $\K^1(G\backslash X)$.
\end{proof}

In the above argument, we showed that \(\Rep^*(H)\) 
is a direct summand of \(\K_0(Q_X)\) in the case of isolated 
fixed points with finite stabilizer $H$, and we computed that the composition 
\[\Rep^*(H)\subseteq \Rep(H) \cong\K^0_G(Gy) \to \K_0(Q_X) 
\stackrel{\partial}{\longrightarrow} \K^{1}(G\backslash X)\]
maps any generator \([\sigma]\in \Rep^*(H)\) to 
a torsion class in \(\K^1(G\backslash X)\).

If \(X\) is a smooth manifold and the Lie group \(G\) acts smoothly, this 
composition may be made slightly more 
explicit using the language of differential topology. 
In the notation of 
the proof of Theorem \ref{thm-Kproper}, take a point 
\(y \in X\) with isotropy \(H\). Let 
\(\nu\) be the normal bundle to \(Gy\); we may equip it with 
a \(G\)-invariant Riemannian metric, then a re-scaling 
composed with the exponential map determines a tubular 
neighbourhood embedding \(\nu\cong U \subset X\) for an 
open \(G\)-invariant neighbourhood \(U\) of \(Gy\). Thus, 
\(U \cong G\times_H Y\) for the 
corresponding 
orthogonal linear action of \(H\) on a Euclidean space 
\(Y \defeq \nu_y \cong \R^k\). Let \(S\nu\) be the unit sphere
bundle of \(\nu\), let \(i\colon S\nu \to U\smallsetminus Gy\) 
be the corresponding smooth equivariant embedding; 
its normal bundle is equivariantly trivializable (it is isomorphic to a 
trivial \(G\)-vector bundle with trivial \(G\)-action). Thus, 
there is a \(G\)-equivariant smooth open embedding 
\(S\nu \times \R \to U\smallsetminus Gy\). If fixed-points 
are isolated, then \(G\) acts freely on \(S\nu\) and 
\(U\smallsetminus Gy\) so that we obtain an open embedding 
\[ \varphi \colon G\backslash S\nu \times \R
 \longrightarrow G\backslash X.\]

Now the boundary map \(\partial\colon \Rep^*(H) \to 
\K^1(G\backslash X)\) may be described simply as 
follows: it is the composition 
\begin{multline}
 \Rep^*(H) \subset \Rep (H) \cong 
\K^0_G(Gy) \xrightarrow{p^*}\K^0_G(S\nu) 
\cong \K^0_G(G\backslash S\nu) 
\\ \xrightarrow{\otimes \beta} \K^1(G\backslash S\nu \times \R) 
\xrightarrow{\varphi!} \K^1(G\backslash X),
\end{multline}
where \(p\colon S\nu \to Gy\) is the projection map, 
\(\beta\in \K^1(\R)\) the Bott class.

This construction, as mentioned above, produces torsion 
classes in \(\K^1(G\backslash X)\) of order a divisor of 
\(\abs{H}\). We will see in Example \ref{ex-nontrivialboundary} below
 that these classes 
are not always trivial, so that the boundary map in 
\eqref{eq-KQp} does not always vanish in general.


But we first present an interesting example of an action with isolated 
orbits with non-trivial stabilizers, in which $K^1(G\backslash X)=\{0\}$, so that 
the exact sequence  \eqref{eq-KQp} computes everything.

\begin{example}\label{ex-sphere} In this example we consider the cyclic group $H=\lk R\rk$ of order four
with $R=\left(\begin{smallmatrix} 0& -1\\1&0\end{smallmatrix}\right)\in \GL(2,\ZZ)$ acting
on $\TT^2=\RR^2/\ZZ^2$.
This is the action of the dihedral group $G=\lk R,S\rk$ on $\TT^2$  as considered 
in Examples \ref{ex-D4}, \ref{ex-D4-1} and \ref{ex:dihedral} restricted to the subgroup $H\subseteq G$.
 Deformations of the crossed product $C(\TT^2)\rtimes H$, known as non-commutative $2$-spheres,
have been studied extensively in the literature, and it is shown in \cite{ELPW}
 that the $\K$-theory groups of these
deformations are isomorphic to the $\K$-theory groups of $C(\TT^2)\rtimes H$.

If we study this action on the fundamental domain $\{(s,t):-\frac12\leq s,t\leq \frac12\}\subseteq\RR^2$ for the action of $\ZZ^2$ on $\RR^2$, we see that (the image in $\TT^2$ of) $\{(s,t): 0\leq s,t\leq \frac12\}$ in $\TT^2$
 is a fundamental domain (but not a topological fundamental domain as in Definition 
 \ref{def-fundamentaldomain}) for the action of $H$ on $\TT^2$ such that 
 in the quotient $H\backslash \TT^2$ the line $\{(s,0): 0\leq s \leq\frac12\}$ in the boundary 
 is glued to $(0,t) : 0\leq t\leq \frac12\}$ and the line $\{(s,\frac12): 0\leq s \leq\frac12\}$ is glued to
 $\{(\frac12,t): 0\leq t\leq \frac12\}$. Thus we see that the quotient $H\backslash \TT^2$ is 
 homeomorphic to the $2$-sphere $S^2$. 
 
 There are only three orbits in $H\backslash\TT^2$ with nontrivial stabilizers: the points corresponding to
  $(0,0)$ and $(\frac12,\frac12)$ in the fundamental domain have full stabilizer $H$ and the 
orbit of the point corresponding to $(0,\frac12)$ has stabilizer $\lk R^2\rk$.
Since $\K^0(S^2)=\ZZ$,  $\K^1(S^2)=\{0\}$ and since $H$ has three non-trivial characters 
and $\lk R^2\rk$ has one non-trivial character, we see instantly from 
the exact sequence \eqref{eq-KQp}, that $\K_0((C(\TT^2)\rtimes H)\cong \ZZ^8$ and $\K_1(C(\TT^2)\rtimes H)=\{0\}$. 

In \cite{ELPW} we used a quite different and more complicated method for computing the 
$\K$-theory groups of $C(\TT^2)\rtimes H$. See \cite{ELPW} for an explicit 
description of the \(\K\)-theory generators. 
\end{example}

We proceed with an example of
the equivariant \(\K\)-theory computation  for the action of $G=D_4$ on $\TT^2$
as studied earlier in Examples \ref{ex-D4}, \ref{ex-D4-1} and \ref{ex:dihedral}.
In this case the orbits with non-trivial stabilizers are not isolated and we 
use the description of the ideal structure as given in Example \ref{ex:dihedral} 
for the computation.

\begin{example}\label{ex-D4final}
Consider the crossed product $C(\TT^2)\rtimes G$. It is shown in
Example \ref{ex:dihedral}  that we get a sequence of ideals
$$\{0\}=I_0\subseteq I_1\subseteq I_2\subseteq I_3=C(\TT^2)\rtimes G$$
with $I_1$ Morita equivalent to $C(Z)$, $I_2/I_1$ Morita equivalent to 
$C(\partial Z)$ and $I_3/I_2$ Morita equivalent to $\CC^8$.

We first compute the $\K$-theory of $I_2$. Since $\K_0(I_1)=\ZZ$ and $\K_1(I_1)=\{0\}$
and $\K_0(I_2/I_1)=\K_1(I_2/I_1)=\ZZ$, the
six-term sequence
with respect to the ideal $I_1\sim_MC(Z)$ reads
$$
\begin{CD} 
\ZZ  @>>> \K_0(I_2) @>>> \ZZ\\
@AAA     @.        @VVV\\
\ZZ  @<<<  \K_1(I_2)  @<<< 0
\end{CD}
$$
On the other hand, the  ideal $J:=C_0(\stackrel{\circ}{Z},M^8(\CC))\subseteq I_2$
with quotient $I_2/J$ Morita equivalent to $C(\TT\times\{0,1\})$ gives a six-term sequence
$$
\begin{CD} 
\ZZ  @>>> \K_0(I_2) @>>> \ZZ^2\\
@AAA     @.        @VVV\\
\ZZ^2  @<<<  \K_1(I_2)  @<<< 0
\end{CD}
$$
It follows that  $\K_0(I_2)\cong \ZZ^2$ and $\K_1(I_2)=\ZZ$.
We then get the six-term sequence
$$
\begin{CD} 
\ZZ^2  @>>> \K_0(C(\TT^2)\rtimes G) @>>> \ZZ^8\\
@AAA     @.        @VVV\\
0  @<<<  \K_1(C(\TT^2)\rtimes G)  @<<< \ZZ
\end{CD}
$$
We leave it as an interesting  exercise for the reader (using the structure of $I_3=C(\TT^2)\rtimes G$
as indicated in Example \ref{ex-D4-1}) to check 
directly that  the class of $K_0(I_3/I_2)$  corresponding to the character $\chi_2\in \widehat{G}$ 
at the point $(0,0)\in Z$ cannot be lifted to a class in $\K_0(I_3)$, and thus maps to
a non-zero element in $\K_1(I_2)\cong \ZZ$. Indeed, with a bit of work one can show that it maps 
to a generator of $\K_1(I_2)\cong \ZZ$ which then implies that 
$\K_1(C(\TT^2)\rtimes G)=\{0\}$ and  $\K_0(C(\TT^2)\rtimes G)=\ZZ^9$.
\end{example}

We should mention that the algebra $C(\TT^2)\rtimes D_4$ of the above example is just the group
algebra $C^*(\ZZ^2\rtimes D_4)$ of the 
 cristallographic group $\ZZ^2\rtimes D_4$. The $\K$-theory of this group algebra 
(together with the $\K$-theories of all other crystallographic groups of rank $2$)
 has been computed by completely different methods by L\"uck and Stamm 
 in \cite{LS}. Even before that, the $\K$-theory of the group algebras of crystallographic groups of 
 rank $2$ were computed by Yang in \cite{Ya} by methods much closer to ours.
 In fact, many of the general results obtained in this article have been obtained in \cite{Ya}
 in case  of finite group actions. 
 \medskip

 We now provide the  promised counter-example for the trivialization problem as 
posed in Remark \ref{rem-trivialization}. 

\begin{example}\label{ex-trivialization}
Let $G=\ZZ/n$ be a finite cyclic group acting linearly on some $\RR^m$ such that 
the action on $\RR^m\setminus \{0\}$ is free (it follows that $n=2$ or $m=2k$ is even). 
Consider the crossed-product $C(B)\rtimes G$,
where $B$ denotes the closed unit ball in $\RR^n$. Let $\res:C(B)\to C(S^{m-1})$ denote the restriction map.
We then obtain a map
\begin{equation}\label{eq-trivialization}
\phi: R(G)\cong \K_0(C(B)\rtimes G)\stackrel{\res_*}{\longrightarrow} \K_0(C(S^{m-1})\rtimes G)\cong 
\K^0(G\backslash S^{m-1}).
\end{equation}
It is shown in \cite{Atiyah} (see also \cite[Theorem 0.1]{Gilkey}) that this map is surjective. 

Using the bundle structure of $C(B)\rtimes G\cong C(B\rtimes_G M_n(\CC))$, this map can be described as
follows: First of all consider $C(B\rtimes_G M_n(\CC))$ as a bundle over $[0,1]$ with
fibre $C^*(G)\cong M_n(\CC)^{G}$ at $0$ and fibre $C(S^{m-1})\rtimes G\cong C(S^{m-1}\times_GM_n(\CC))$
at every $t\neq 0$. This bundle is trivial outside zero, and we obtain the above sequence of $\K$-theory maps
 by first extending a projetion of the 
zero fiber to a projection on a small neighborhood, then use triviality outside zero to extend
the projection to the whole bundle, and finally evaluate this extended projection at the fiber at $1$.

Assume now that $C(B\times_GM_n(\CC))$ would be trivializable in the sense that it
 would be stably isomorphic to some subbundle of the  trivial bundle 
$C(G\backslash B,\cK)$ with full fibres outside (the orbit of) the origin.  
Then any projection in the fiber at the origin has a trivial extension (by the constant section) to 
all of $C(G\backslash B,\cK)$ and the restriction of such projection to 
$G\backslash S^{m-1}$  would lie in the subgroup 
$\ZZ\cdot 1_{G\backslash S^{m-1}}$ of $\K^0(G\backslash S^{m-1})$. 
Thus the existence of a trivialization together with 
surjectivity of the map $\res_*$ in (\ref{eq-trivialization}) would imply that 
$\K^0(G\backslash S^{m-1})\cong \ZZ\cdot 1_{G\backslash S^{m-1}}$. 
But one can see in the appendix of \cite{Gilkey} that 
$\widetilde{\K}^0(G\backslash S^{m-1}):=\K^0(G\backslash S^{m-1})/ \ZZ\cdot 1_{G\backslash S^{m-1}}$
is non-trivial (of finite order) for many choices of groups $\ZZ_n$ acting on a suitable $\RR^m$.
For instance: if $m=4$ and $n=2$, we get $\widetilde{\K}^0(G\backslash S^{3})\cong \ZZ/2$.
\end{example}

The above example also provides the basis for the following example of an action with
isolated fixed points, such that the boundary map in \eqref{eq-KQp} does not vanish. We are 
grateful to Wolfgang L\"uck for pointing out this example to us.

\begin{example}\label{ex-nontrivialboundary}
Choose $n$ and $m$ and an action of $G=\ZZ/n$ on the unit ball $B\subseteq \RR^m$ as in the previous example such that 
$\widetilde{\K}^0(G\backslash S^{m-1}):=\K^0(G\backslash S^{m-1})/ \ZZ\cdot 1_{G\backslash S^{m-1}}$
 is non-trivial. Notice that the map $\phi$ of (\ref{eq-trivialization}) factors through a surjective 
 map 
 $$\widetilde\phi:\Rep^*(G)\to \widetilde{\K}^0(G\backslash S^{m-1}),$$
 since 
 $\ZZ\cdot 1_{G\backslash S^{m-1}}$ is the image of $\ZZ\cdot 1_G\subseteq \Rep(G)$.
Glueing two such balls at the boundary $\partial B=S^{m-1}$ we obtain an action of $G$ on $S^m$
which fixes the two points $(0,\ldots, 0,\pm1)$ and which is free on 
$S^m\smallsetminus\{(0,\ldots,0,\pm1)\}$. 

Let us consider 
 $C(S^m)\rtimes G$ as a bundle over 
$[-1,1]$ with fiber $C(S^{m-1})\rtimes G$ over each $t\neq\pm1$ and fiber $C^*(G)$ at $t=\pm1$.
We write $\Rep(G)_+$ and $\Rep(G)_-$ for the representation ring of $G$
when identified with the $\K$-theory of the fiber at $t=1$ and $t=-1$, respectively.
Then it follows from the description of the map $\phi$ in (\ref{eq-trivialization}) of the previous example that
for a given pair $([p_+],[p_-])\in \Rep(G)_+\oplus\Rep(G)_-$ there exists a
class $[p]\in \K_0(C(S^m)\rtimes G)$ restricting to the pair if and only if $\phi([p_+])=\phi([p_-])$.

Now let $I_{S^m}$ and $Q_{S^m}$ as in Theorem \ref{thm-Kproper}. Then 
$$\K_0(Q_{S_m})\cong \Rep^*(G)_+\oplus \Rep^*(G)_-.$$
and the sequence (\ref{eq-KQp}) becomes
\begin{multline}\label{eq-KQSm}
0\longrightarrow \K^0(G\backslash S^m)\longrightarrow 
\K_0(C(S^m)\rtimes G)\\ \longrightarrow 
\Rep^*(G)_+\oplus\Rep^*(G)_-  \stackrel{\partial}{\longrightarrow} 
 \K^1(G\backslash S^m)
 \longrightarrow \K_1(C(S^m)\rtimes G) \longrightarrow 0.
 \end{multline}
Since a pair $([p_+],[p_-])\in \Rep^*(G)_+\oplus \Rep^*(G)_-$ lies in the kernel of the boundary map
$\partial$ if and only if it extends to a class in $\K_0(C(S^m)\rtimes G)$, it follows from the 
above considerations that this is possible if and only if the values $\phi([p_+])$ and $\phi([p_-])$ differ 
by some class in $\ZZ\cdot 1_{G\backslash S^{m-1}}$. Thus we see that
$$\ker\partial= 
\{([p_+],[p_-])\in \Rep^*(G)_+\oplus\Rep^*(G)_-: \widetilde{\phi}([p_+])=\widetilde{\phi}([p_-])\}$$
which is a proper subgroup of $\Rep^*(G)_+\oplus\Rep^*(G)_-$.
Thus it follows that the boundary map in (\ref{eq-KQSm}) is not trivial.

Note that a more elaborate study  of this action (which we omit) reveals that $ \K^1(G\backslash S^m)\cong 
\widetilde{\K}^0(G\backslash S^{m-1})$ and that the boundary map in  (\ref{eq-KQSm})
is given by sending a pair $([p_+],[p_-])$ to the difference $\widetilde{\phi}([p_+])-\widetilde{\phi}([p_-])$
and hence is surjective. Thus it follows that
$$\K_G^0(S^m)=\K_0(C(S^m)\rtimes G)\cong \K^0(G\backslash S^m)\oplus \ker\partial
\quad\text{and}\quad \K_G^1(S^m)=\{0\}.$$

\end{example}

\def\mathcs{{\normalshape\text{C}}^{\displaystyle *}}


\begin{thebibliography}{10}


\bibitem{Abels}
H. Abels. \emph{A universal proper $G$-space.} Math. Z. 159 (1978), no. 2, 143--158.


\bibitem{Atiyah}
M.F. Atiyah.
$K$-theory. Amsterdam: Benjamin 1967.

\bibitem{Atiyah-Bott:Moment}
M.F. Atiyah, R. Bott. 
\emph{The moment map in equivariant 
cohomology}. Topology 23, no.1 (1984), 1--28. 

\bibitem{Atiyah-Segal:Index}
M.F. Atiyah, G. Segal.
\emph{The index of elliptic operators II}, Ann. Math., 2nd series, 87, no. 3 (1968),
531--545. 


\bibitem{Atiyah-Segal:Euler} 
M.F. Atiyah, G. Segal.
 \emph{On equivariant Euler characteristics}, 
J.G.P., {\bf 6}, no. 4, 1989, 671--677. 


\bibitem{Bag}
L. Baggett.
\emph{A description of the topology on the dual spaces of certain locally compact groups}. 
Trans. Amer. Math. Soc. {\bf 132} (1968), 175--215. 

\bibitem{Baum-Connes:Finite}
P. Baum, A. Connes.
\emph{Chern character for discrete groups}.  
A f\'ete of topology. 
Academic Press, Boston, MA (1988), pp. 163--232.





  \bibitem{BCH} P. Baum, A. Connes, N.  Higson. \emph{Classifying
    space for proper actions and $\K$-theory of group C*-algebras},
  Contemporary Mathematics, \textbf{167}, 241-291 (1994).
  
  \bibitem{Black}
B.  Blackadar. 
\emph{$\K$-theory for operator algebras}. Second edition. Mathematical Sciences Research Institute Publications, 5. Cambridge University Press, Cambridge, 1998. xx+300 pp. ISBN: 0-521-63532-2
  
  \bibitem{Blatt}
R.J.  Blattner.
\emph{On a theorem of G. W. Mackey}. 
Bull. Amer. Math. Soc. {\bf 68} (1962), 585--587.
  

\bibitem{Combes}
F. Combes.
 \emph{Crossed products and Morita equivalence}. Proc. London Math. Soc. (3) {\bf 49} (1984), no. 2, 289--306.
 

\bibitem{CF}
P.E. Conner, E.E.  Floyd.
\emph{Maps of Odd Period}.
Ann. Math., Ser. 2, {\bf 84}, no. 2 (1966) 132--156. 
  
  \bibitem{DE}
A.  Deitmar, S. Echterhoff. \emph{Principles of harmonic analysis}. Universitext. Springer, New York, 2009. xvi+333 pp. ISBN: 978-0-387-85468-7

\bibitem{ech-induced} S. Echterhoff.
 \emph{On induced covariant systems}. Proc. Amer. Math. Soc. {\bf 108} (1990), no. 3, 703--706.
 
\bibitem{Ech}
S. Echterhoff.
\emph{Crossed products, the Mackey-Rieffel-Green machine and applications.} 
Preprint: arXiv:1006.4975v1 [math.OA]

\bibitem{EEK1}
S. Echterhoff, H. Emerson, H-J. Kim.
\emph{KK-theoretic duality for proper twisted actions.} Math. Ann. {\bf 340} (2008), no. 4, 839Ð873. 


\bibitem{ech-Lefschetz}
S. Echterhoff, H. Emerson, H-J. Kim.
\emph{A Lefschetz fixed-point formula for certain orbifold C*-algebras}
J. Noncommut. Geom. 4 (2010) (To appear). 



\bibitem{ELPW}
S. Echterhoff, W. L\"uck, N.C. Phillips, S.  Walters. \emph{The structure of crossed products of irrational rotation algebras by finite subgroups of ${\rm SL}_2(\Bbb Z)$}. J. Reine Angew. Math. {\bf 639} (2010), 173--221. 

\bibitem{EP}
S. Echterhoff, O. Pfante.
\emph{Equivariant $\K$-theory of Þnite dimensional real 
vector spaces}. M\"unster J. of Math. {\bf 2} (2009), 65Ð94.

\bibitem{EchWil}
S. Echterhoff, D.P. Williams.
\emph{Inducing primitive ideals.} Trans. Amer. Math. Soc. 360 (2008), no. 11, 6113--6129.

\bibitem{Emerson:Localization}
H. Emerson.
\emph{Localization techniques in circle-equivariant \(\KK\)-theory}.
arXiv preprint 0024882. 

\bibitem{Emerson-Meyer:Equi_Lefschetz}
H. Emerson, R. Meyer.
\emph{Equivariant Lefschetz maps for simplicial complexes and smooth manifolds}.
Math. Ann., \textbf{345} (2009), no. 3, pp. 599--630. 

\bibitem{Emerson-Meyer:Equivariant_embeddings}
H. Emerson, R. Meyer.
\emph{Equivariant embedding theorems and topological 
index maps}. 
arXiv preprint 0908.1465. 

\bibitem{Em-K}
H. Emerson, R. Meyer.
\emph{Equivariant representable \(\K\)-theory.} 
J. Topol. \textbf{2} (2009), no. 1, 123--156. 

\bibitem{Evans}
B. Evans. \emph{C*-bundles and compact transformation groups.} 
Memoirs Amer. Math. Soc. {\bf 39} no 269 (1982).

\bibitem{Gilkey}
P.B. Gilkey. 
\emph{The eta invariant and $\tilde{K}O$ of lens spaces.}
Math. Z. {\bf 194} (1987), 309--320.


\bibitem{Glimm0}
J. Glimm.
\emph{Locally compact transformation groups}. 
Trans. Amer. Math. Soc. {\bf 101} (1961) 124--138. 

\bibitem{Glimm1}
J. Glimm.
\emph{Families of induced representations}. 
Pacific J. Math. {\bf 12} (1962), 885--911.

\bibitem{Green0} P. Green,
\emph{$C^*$-algebras of transformation groups with smooth orbit space.}
 Pacific J. Math. 72 (1977), no. 1, 71--97.

\bibitem{Green1} P. Green,  \emph{The local structure of twisted covariance algebras.}
 Acta Math. {\bf 140} (1978), no. 3-4, 191--250. 

\bibitem{GootRos}
E.C. Gootman, J. Rosenberg.
\emph{The structure of crossed product $C^{\ast} $-algebras: a proof of the generalized Effros-Hahn conjecture.}  Invent. Math. 52 (1979), no. 3, 283--298.


 \bibitem{Fell}
 J.M.G. Fell.
\emph{ Weak containment and induced representations of groups}. II. Trans. Amer. Math. Soc. {\bf 110} (1964), 424--447. 
 
\bibitem{Julg}
P. Julg.
\emph{ $\K$-th\'eorie \'equivariante et produits crois\'es}.  C. R. Acad. Sci. Paris S\'er. I Math. {\bf 92} (1981), no. 13, 629--632.

\bibitem{Kar}
M. Karoubi. \emph{Equivariant $K$-theory of real vector spaces and real projective spaces.}
 Topology Appl. {\bf 122} (2002), no. 3, 531--546.
 
 
\bibitem{Kas} G. Kasparov.
 \emph{Equivariant
      $\KK$-theory and the Novikov conjecture}, Invent. Math.
    \textbf{91},  147-201 (1988).



\bibitem{KasSk} G. Kasparov, G. Skandalis.
\emph{Groups acting properly on ``bolic'' spaces and the Novikov conjecture.}
Annals of Math. {\bf 158} (2003), 165--206.

\bibitem{Lueck-Oliver} W. L\"uck and B. Oliver. 
\emph{The completion theorem in \(\K\)-theory for proper 
actions of a discrete group}. 
Topology \textbf{40} (2001), no. 3, 585--616.

\bibitem{LS} W. L\"uck and R. Stamm.
\emph{Computations of $K$- and $L$-theory of cocompact planar groups.}
$K$-theory {\bf 21} (2000), 249--292.

  \bibitem{Meyer-fix} R. Meyer.
   \emph{Generalized fixed-point algebras and square integrable representations.}
  J. Funct. Anal. {\bf 186}, (2001), 167--195.



\bibitem{Pal} R.S. Palais.
 \emph{On the existence of slices for actions of non-compact Lie groups.} Ann. of Math. (2) {\bf 73} (1961) 295--323.

\bibitem{RW-pull-back}
I. Raeburn, D.-P. Williams.
\emph{ Pull-backs of $C^*$-algebras and crossed products by certain diagonal actions.}
 Trans. Amer. Math. Soc. {\bf 287} (1985), no. 2, 755--777.
 
 \bibitem{RW}
 I. Raeburn, D.-P. Williams.
 \emph{Morita equivalence and continuous-trace $C^*$-algebras}. Mathematical Surveys and Monographs, 60. American Mathematical Society, Providence, RI, 1998. xiv+327 pp. ISBN: 0-8218-0860-5 
 
 \bibitem{Rief-proper}
 M.A. Rieffel.
\emph{ Proper actions of groups on $C^*$-algebras}. Mappings of operator algebras (Philadelphia, PA, 1988), 141--182, Progr. Math., 84, BirkhŠuser Boston, Boston, MA, 1990.
 
  \bibitem{Rief-proper1}
  M.A. Rieffel.
\emph{  Integrable and proper actions on $C^*$-algebras, and square-integrable representations of groups}. Expo. Math. {\bf 22} (2004), no. 1, 1--53.
  
 \bibitem{Rief-finite}  M.A. Rieffel. 
Actions of finite groups on $C^{\ast} $-algebras. Math. Scand. {\bf 47} (1980), no. 1, 157--176.
 
 \bibitem{Rief-heis}  M.A. Rieffel. 
\emph{On the uniqueness of the Heisenberg commutation relations}. Duke Math. J. {\bf 39} (1972), 745--752.

\bibitem{Segal}
G. Segal. 
\emph{Equivariant $\K$-theory}. Inst. Hautes \'Etudes Sci. Publ. Math. No. {\bf 34} (1968), 129--151.


\bibitem{Dana-book} D.P. Williams. 
{Crossed products of C*-algebras}.
Mathematical Surveys and Monographs, Vol 134. AMS 2007.

\bibitem{Dana-cont} D.P. Williams. 
\emph{The structure of crossed products by smooth actions}. J. Austral. Math. Soc. Ser. A {\bf 47} (1989), no. 2, 226--235.
    
    \bibitem{Ya} M. Yang. \emph{Crossed products by finite groups acting on low dimensional complexes and applications.} Ph.D Thesis, University of Sasketchewan, Saskatoon, 1997.
    
    \end{thebibliography}
\end{document}